\documentclass[12pt, a4paper]{amsart}
\usepackage{amscd}
\usepackage{amsmath}
\usepackage{latexsym}
\usepackage{graphicx}
\usepackage{amsfonts}
\usepackage{amssymb}
\usepackage[all]{xy}

\renewcommand\atop[2]{\genfrac{}{}{0pt}{}{#1}{#2}}

\providecommand{\U}[1]{\protect\rule{.1in}{.1in}}
\nonstopmode
\newcounter{fig}

\newcommand{\coker}{\mathop{\mathrm{coker}}}

\newcommand{\R}{\mathbb{R}}
\newcommand{\C}{\mathbb{C}}
\renewcommand{\H}{\mathbb{H}}
\newcommand{\Q}{\mathbb{Q}}
\newcommand{\Z}{\mathbb{Z}}

\newcommand{\GW}{\mathbb{G}W}
\newcommand{\GWH}{\mathbb{GWH}}

\newcommand{\bbK}{\mathbb{K}}
\newcommand{\bKU}{\mathbb{KU}}
\newcommand{\cE}{\mathcal{E}}
\newcommand{\cO}{\mathcal{O}}
\newcommand{\red}{\textrm{red}}
\newcommand{\tors}{\textrm{tors}}
\newcommand{\eps}{\varepsilon}

\numberwithin{equation}{section}
\theoremstyle{plain}
\newtheorem{theorem}[equation]{Theorem}

\newtheorem{corollary}[equation]{Corollary}
\newtheorem{proposition}[equation]{Proposition}
\newtheorem{lemma}[equation]{Lemma}
\newtheorem{substuff}{\bf Remark}[equation] 
\newtheorem{sublem}[substuff]{Lemma}
\newtheorem{subcor}[substuff]{Corollary}
\newtheorem{subvariant}[substuff]{Variant}
\theoremstyle{definition}

\newtheorem{definition}[equation]{Definition}
\newtheorem{examples}[equation]{Examples}
\newtheorem{example}[equation]{Example}
\newtheorem{subex}[substuff]{Example}

\newtheorem{defn}[equation]{Definition}
\newtheorem{Void}[equation]{}
\theoremstyle{remark}
\newtheorem{remark}[equation]{Remark}

\newtheorem{rem}[equation]{Remark}
\newtheorem*{remm}{Remark}
\newtheorem{subremark}[substuff]{Remark} 

\newtheorem*{notations}{Notation}

\def\smap#1{\ {\buildrel #1 \over \rightarrow}\ }
\def\map#1{{\buildrel #1 \over \longrightarrow}}
\newcommand{\mathdot}{{\mathbf{\scriptscriptstyle\bullet}}}
\newcommand{\Spec}{\operatorname{Spec}}
\newcommand{\Pic}{\operatorname{Pic}}

\newcommand{\Hom}{\operatorname{Hom}}
\newcommand{\pt}{{\mathrm{pt}}}
\newcommand{\nis}{{\mathrm{nis}}}
\newcommand{\ret}{{\mathrm{ret}}}
\newcommand{\topl}{{\mathrm{top}}}
\newcommand{\RP}{\mathbb{RP}}
\newcommand{\Gm}{\mathbb{G}_{m}}
\renewcommand{\L}{\mathbb{L}}

\newcommand{\Sper}{\operatorname{Sper}}
\newcommand{\Max}{\operatorname{Max}}
\newcommand{\Sh}{\operatorname{Sh}}

\usepackage{color}
\newcommand{\mGW}{GW}  
\newcommand{\kGW}{\GW} 
\newcommand{\uGW}{GW}  

\begin{document}
\title[Witt group of real varieties]{The Witt group of real algebraic
varieties}
\date{\today}

\author{Max Karoubi}
\address{Universit\'e Denis Diderot Paris 7 \\
Institut Math\'ematique de Jussieu --- Paris Rive Gauche}
\email{max.karoubi@gmail.com}
\urladdr{http://webusers.imj-prg.fr/~max.karoubi}

\author{Marco Schlichting}
\address{Math.\ Institute, University of Warwick,Coventry CV4 7AL, UK}
\email{M.Schlichting@warwick.ac.uk}
\urladdr{http://http://homepages.warwick.ac.uk/~masiap}
\thanks{Schlichting was supported by EPSRC grant EP/M001113/1}

\author{Charles Weibel}
\address{Math.\ Dept., Rutgers University, New Brunswick, NJ 08901, USA}
\email{weibel@math.rutgers.edu}\urladdr{http://math.rutgers.edu/~weibel}
\thanks{Weibel was supported by NSA and NSF grants, 
and by the IAS Fund for Math} 

\begin{abstract}
The purpose of this paper is to compare the algebraic Witt group 
$W(V)$ of quadratic forms for an algebraic variety $V$ over $\R$ 
with a new topological invariant, $WR(V_{\C})$, 
based on symmetric forms on Real vector bundles (in the sense of Atiyah)
on the space 
of complex points of $V$. 
This invariant lies between $W(V)$ and the group $KO(V_{\R})$
of $\R$-linear topological vector bundles on the space $V_{\R}$
of real points of $V$.

We show that the comparison maps $W(V)\to WR(V_{\C})$ and
$WR(V_{\C})\to KO(V_{\R})$ 
are isomorphisms modulo bounded 2-primary torsion. We give 
precise bounds for their exponent of the kernel and cokernel, 
depending upon the dimension of $V.$ 
These results improve theorems of Knebusch, Mah\'e and Brumfiel.

Along the way, we prove comparison theorem between algebraic 
and topological hermitian $K$-theory, and 
homotopy fixed point theorems for the latter.
We also give a new proof (and a generalization) of a theorem of Brumfiel.
\end{abstract}
\maketitle

\pagestyle{myheadings} 
\setcounter{section}{0}%
\vspace{-30pt}
\tableofcontents

\goodbreak
Let $V$ be an algebraic variety defined over $\R$. 
The computation of its Witt ring $W(V)$ of quadratic forms
is a classical problem which 
has attracted much attention. For instance, if the topological space 
$V_\R$ of $\R$-points of $V$ has $c$ components, there is a classical
``signature map'' $W(V)\to\Z^c$ \cite[pp.\,186--188]{Knebusch.Queens}. 
To define it, note that a nondegenerate
symmetric bilinear form $\varphi$ on an algebraic vector bundle over $V$
yields a continuous family of bilinear forms $\varphi_x$ 
on the vector spaces over the points $x$ of $V_\R$; the signatures of the
$\varphi_x$  are constant on each connected component, and the 
signature map sends $\varphi$ to the sequence of these signatures in $\Z^c$.

Knebusch  \cite{Knebusch.Queens}, in analogy with results of 
Witt \cite{Witt}, asked whether the image of this signature map
is a sugbroup of finite index.
This was established in special cases
by Knebusch \cite{Knebusch} (curves) and by 
Colliot-Th\'el\`ene--Sansuc \cite{CTS} (surfaces).
The general case was obtained by Mah\'e \cite{Mahe, Mahe1, Mahe2}
($V$ affine) and Mah\'e--Houdebine \cite{MH} ($V$ projective).  
Building on the work of Mah\'e, 
this theorem was greatly generalized by 
Brumfiel \cite{Brumfiel84}, using the topological $KO$-group
associated to the space $V_\R$. Brumfiel constructed a map
\[
\gamma: W(V)\to KO(V_\R) = \Z^c\oplus \widetilde{KO}(V_\R),
\]
and showed that the kernel and cokernel of this map are
2-primary torsion groups. It follows that the cokernel is a finite group.
A recent result of Totaro \cite{Totaro} shows that the kernel
of $\gamma$ can be infinite.

In this paper we introduce a finer invariant, $WR(V_\C)$, which uses 
the space $X=V_\C$ of complex points of $V$, endowed with
the involution coming from complex conjugation.  Our construction
is based on the fundamental paper of Atiyah \cite{Atiyah},
who introduced the term {\it Real space} (with a capital R) for a space $X$
with involution and defined the {\it Real $K$-groups} $KR(X)$
of Real vector bundles on $X$.
We mimick the definition of $W(V)$ in the setting of Witt groups 
to define a finitely generated abelian group for any Real space $X$, 
called the {\it Real Witt group} $WR(X)$.

We show that Brumfiel's map $\gamma$ factors through $WR(V_\C)$,
and that the kernel of the maps $\theta: W(V)\to WR(V_\C)$ and
$W(V)\to KO(V_\R)$ are (2-primary) torsion groups
of {\it bounded} exponent $2^e$ and $2^{e+f+1}$ respectively,
only depending upon the dimension $d$ of $V$:
$e=2+4\lceil(d-2)/8\rceil$, where $\lceil x\rceil$ is the 
least integer $n\ge x$, and $f$ is the number of positive
$i\le2d$ with $i=0,1,2$ or $4$ mod $8$ 
(see Theorems \ref{W-WR.exponent} and \ref{signature:8m+2}).
Even better, if $V$ is a curve we have $W(V)\cong WR(V_\C)$; 
see Theorem \ref{curves:W=WR}. 

Brumfiel's map $\gamma$ is not an isomorphism for arbitrary $V$;
for instance, $\gamma$ is zero when $V$ has no $\R$-points.
In contrast, the map $W(V)\to WR(V_\C)$ is always nonzero
(unless $V=\emptyset$); 
a complex point $\Spec(\C)\to V$ yields a surjection 
$WR(V_\C)\to WR(\C)=\Z/2$ which is nonzero on $W(V)$. 
In short, $WR(V_\C)$ is a better invariant than $KO(V_\R)$.

A key step in the proofs of our results is a description of Real
Grothendieck-Witt theory $GR$ as a homotopy fixed point set, given in
Theorem \ref{Williams-GR}; it is an 
analogue of a result for schemes \cite{BKOS} conjectured by B. Williams.
The proofs also use the authors' computation of $K$-theory of real 
varieties in \cite{KW}
and some older results related to Bott periodicity \cite{MKAnnalsH}.

To our knowledge, the group $WR(X)$ has not been studied in the literature.
We find this a bit surprising.

Similar results also hold for skew-symmetric forms on a variety over $\R$.
We write  $_{-1}GW(V)$,~ ${}_{-1}GR(X),$ etc.\ for the associated theories.
See Theorems \ref{-GR=KRH} and \ref{GW=GR{-1}}, for example.
When we deal with both symmetric and antisymmetric forms, we use the
notation $_{\eps}GW(V),{}_{\eps}GR(X),$ etc., with $\eps=\pm1$. 
(This essentially follows the notations in \cite{MKAnnalsH}.)

\medskip 
Here is a more detailed description of our paper.
In Section \ref{sec:GR-WR}, we define the groups $GR(X)$ and $WR(X)$.
In Section \ref{sec:GR=KOG}, we identify the Real Grothendieck-Witt
group $GR(X)$ of a Real space $X$ with the Grothendieck group $KO_{G}(X)$ of
equivariant real vector bundles for $G=C_{2}$.
In Section \ref{sec:KR}, 
we improve the results in our earlier paper \cite{KW}, 
relating the algebraic $K$-theory of a variety $V$ over $\R$ 
to the $KR$-theory of its associated Real space.

In Section \ref{sec:curves}, 
we determine the Real Witt groups of smooth projective curves over $\R$;
see Theorems \ref{W(curve)} and \ref{W(real curve)}.
(The calculation for any curve over $\R$ may be determined from this.)
Using results in later sections, we show that $W(V)\cong WR(V)$
for any curve. Given this identification, we recover 
some results of Knebusch \cite{Knebusch} via topological arguments.
We have placed these computations here because they illustrate and
motivate the general results proved in other sections. 
Our explicit computation of $WR(V)$ uses a seemingly unknown relation 
between Atiyah's $KR$-theory and equivariant $KO$-theory, 
which we prove at the end of Appendix \ref{app:Banach}. 

In Section \ref{sec:williams}, 
we verify an analogue of a conjecture of Bruce Williams concerning the
forgetful map $GR(X)\to KR(X)$.

In Section \ref{sec:Brumfiel},  
we reprove and generalize Brumfiel's theorem \cite{Brumfiel84},
mentioned above.
We then show (in Section \ref{sec:WR})  
that $WR(X)$ is isomorphic to $KO(V_{\R})$ modulo {\it bounded} 2-groups.
Combining these results, this implies that the map
$W(V)\to WR(V)$ is an isomorphism modulo 2-primary torsion
for every variety $V$ over $\R$, with a finite cokernel.

In Section \ref{sec:exponents} 
we consider the signature map $W(V)\to KO(V_\R)$ associated to
an algebraic variety $V$ over $\R$, mentioned above.
We show that the kernels of both the signature map 
and the canonical map $W(V)\to WR(V)$ are 2-primary torsion 
groups of {\it bounded} exponent, with a bound which
depends only on $\dim(V)$. Jeremy Jacobson \cite{Jeremy} 
has given a different proof of a similar result.
%
We also bound the exponent of the kernel and cokernel of 
$W(V)\to WR(V)$ when $\dim(V)\leq 8$. 
Analogous but weaker results are proved when $\dim(V)>8.$

For $n>0$, we define the higher Witt group $W_n(V)$ to be the
cokernel of the hyperbolic map $K_n(V)\to GW_n(V)$. We also define
the co-Witt groups $W'_n(V)$ to be the kernels of the
forgetful maps $GW_n(V)\to K_n(V)$. They differ from the Witt groups
by groups of exponent~2. 

In Section \ref{sec:co-Witt}, 
we briefly consider the co-Witt groups $W'_n(V)$.
If $V$ is a smooth curve, we show that
the kernel and cokernel of $W'_n(V)\to WR{\,}'_n(V)$ have exponent~2.

In Section \ref{sec:WnR},    
we determine the Witt groups $W_n(\R)$ and $W_n(\C)$ for $n>0$. 
(This is for the trivial involution on $\C$.)
In this range, we show that the map $W_n(\R)\to KO_n$ is an isomorphism,
except for $n\equiv0\pmod4$, when $W_n(\R)$ injects into $KO_n\cong\Z$
as a subgroup of index~2. We also determine the
co-Witt groups $W'_n(\R)$ for $n>0$.

The appendices, which are of independent interest, introduce technical
results needed in the paper.
In Appendix \ref{app:errata}, we correct some statements in our 
earlier paper \cite{KW}, which we use in Example \ref{WRfree}.
%
In Appendix \ref{app:Bott}, 
we recall some basic facts about Bott elements in Hermitian $K$-theory.
These will be used in Sections \ref{sec:exponents}, 
and \ref{sec:WnR}.

Appendix \ref{app:Banach} 
is devoted to the \textquotedblleft fundamental theorem\textquotedblright\
in topological Hermitian $K$-theory in the context of involutive Banach
algebras and Clifford modules as in Atiyah, Bott and Shapiro (see \cite{ABS}
and \cite{MKThesis}). This appendix is in fact a recollection and a
rewriting in a more readable form of an old paper 
of the first author \cite[Section III]{MKSLN343}.

Appendix \ref{sec:Marco} 
establishes a Hermitian analogue of the
Lichtenbaum-Quillen conjecture in the framework of involutive Banach
algebras.

Throughout this paper, we use the expression ``$A$ has exponent $e$''
to mean that $e\cdot a=0$ for every $a\in A$. When talking about
vector bundles, we make a distinction between ``Real'' 
(for $\C$-antilinear) and ``real'' (for $\R$-linear).
Another convention we use throughout this paper is to write $X$ for the
space $V_{\C}$ of complex points of a variety $V$ defined over $\R$, 
while the space of real points is written $V_{\R}.$

\begin{notations}
The notation $K_0(V)$ and $\mGW_0(V)$ (resp., $KR(X)$ and $GR(X)$)
refers to abelian groups: the Grothendieck and Grothendieck-Witt 
groups of algebraic vector bundles and their symmetric forms on $V$
(resp., of Real topological bundles and their symmetric forms on $X$).
The corresponding spectra are $\bbK(V)$, 
$\GW(V)$, $\mathbb{KR}(X)$ and $\mathbb{GR}(X)$. The spectrum $\uGW(V)$ 
has the same connective cover as $\GW(V)$, and will appear
in Section \ref{sec:Brumfiel} and Appendix \ref{sec:Marco}.

We write $W(V)$ and $WR(X)$ for the Witt group of $V$ and the
Real Witt group of $X$.  We will also abuse notation and write
$KR(V)$, $GR(V)$ and $WR(V)$ for $KR(V_\C)$, $GR(V_\C)$ and $WR(V_\C)$.
\end{notations}

\smallskip
\paragraph{\textit{Acknowledgements}}

The authors are grateful to 
J.\,Rosenberg for pointing out the
misstatement in \cite[1.8]{KW} which we describe in \ref{errata1}.
Conversations with P.\,Balmer, J.-L.\,Colliot-Th\'el\`ene, M. Coste,
J.\,Jacobson, Parimala, A.\,Ranicki and B.\,Totaro were also helpful.

\bigskip\goodbreak
\section{Real Witt and Real Hermitian $K$-theory}
\label{sec:GR-WR}

By a \emph{Real space} we mean a CW complex $X$ with an involution 
$\sigma$, i.e., an action of the cyclic group $C_{2}$. 
By a \textit{Real vector bundle} on $X$ we mean a complex vector 
bundle $E$ which is equipped with an involution 
(also called $\sigma$, by abuse),
such that the projection $p:E\to X$ satisfies $p\sigma=\sigma p$ 
and for each point $x$ of $X$ the map 
$E_{x}\map{\sigma} E_{\sigma x}$ is $\C$-antilinear. 
Following Atiyah \cite{Atiyah}, we define $KR(X)$ as the Grothendieck group of 
the category $\cE_X$ of Real vector bundles on $X$. 

The dual $E^{\ast}$ of a Real vector bundle is the dual of the underlying
complex vector bundle; if $\phi_x\in E_x^*$ then
$(\sigma\phi_x)(u)=\sigma(\phi_x(\sigma u))$, as in \cite{Atiyah}.
A \textit{Real symmetric form} $(E,\varphi)$ is a
Real vector bundle $E$ together with an isomorphism 
$\varphi:E\map{\simeq}E^{\ast}$ of Real vector bundles
such that $\varphi=\varphi^{\ast}.$ 
The hyperbolic forms $H(E)=(E\oplus E^{\ast },h)$ play a special role.
In fact, the category $\cE_X$ of Real vector bundles on $X$ is a 
\emph{Hermitian category} (= exact category with duality), 
which essentially means that
it has a duality $\cE\map{*}\cE^{op}$ and a natural isomorphism 
$E\map{\simeq}E^{\ast\ast}$. (See \cite[2.1]{Schlichting.Herm} for
the precise definition.)

\begin{defn}\label{def:GR} 
If $\cE$ is a Hermitian category, its
Grothendieck-Witt group $\mGW_0(\cE)$ is the Grothendieck group of the
category of symmetric forms $(E,\varphi)$ with $E$ in $\cE$, modulo
the relation that $[E,\varphi]=[H(L)]$ when $E$ has a Lagrangian $L$
(a subobject such that $L=L^\perp$).

The Witt group $W(\cE)$ is the cokernel of the hyperbolic map 
$H:K_0(\cE)\to \mGW_0(\cE)$, which sends $[E]$ to its
associated hyperbolic form $H(E)$. 
Similarly, forgetting $\varphi$ induces a functor 
from symmetric forms to $\cE$, 
and hence a \emph{forgetful} homomorphism 
$F:{\mGW_0}(\cE)\!\to \!K_0(\cE)$.
The co-Witt group $W'(\cE)$ is the kernel of the map $F$.

If $\cE$ has an exact tensor product, then ${\mGW_0}(\cE)$ is a ring and 
$W(\cE)$ is a quotient ring, because the image of the hyperbolic map
is an ideal of ${\mGW_0}(\cE)$. In this case, we refer to ${\mGW_0}(\cE)$ as the
Grothendieck-Witt ring and refer to $W(\cE)$ as the Witt ring of $\cE$.
\end{defn}

If $V$ is a scheme, we write ${\mGW_0}(V)$ for the Grothendieck-Witt ring 
of the Hermitian category of vector bundles on $V$, 
with the usual duality $E\mapsto Hom_{\cO_V}(E,\cO_V)$.
The classical Witt ring $W(V)$ is the Witt ring of this category.

\begin{defn}\label{def:WR} 
If $X$ is a Real space, the Grothendieck-Witt ring of the
Hermitian category $\cE_X$ of Real vector bundles on $X$ is 
written as $GR(X)$. The
Real Witt ring of $X$, $WR(X)$, and the Real co-Witt group, 
$WR{\,}'(X)$, are defined to be 
\begin{align*}
WR(X) & =\mathop{\mathrm{coker}}(KR(X)\map{H} GR(X)),\qquad \\
WR{\,}'(X) & =\ker(GR(X)\map{F} KR(X)).
\end{align*}
\end{defn}

\goodbreak
If $V$ is a variety over $\R$, its \textit{associated Real space} is
the topological space $V_{\C}$ consisting of the complex points of 
$V$, provided with the involution induced by complex conjugation. 
As Atiyah observed \cite{Atiyah}, any algebraic vector bundle on $V$ 
determines a Real vector bundle on $V_\C$, and the resulting functor
sends the dual vector bundle on $V$ to the dual Real vector bundle on
$V_\C$. That is, we have a Hermitian functor from algebraic
vector bundles on $V$ to $\cE_{V_\C}$. This induces natural maps
\[
{\mGW_0}(V)\to GR(V_\C),~ W(V)\to WR(V_\C) 
 ~\textrm{and}~ W'(V)\to WR{\,}'(V_\C).
\]
By abuse of notation, we shall write $KR(V)$, $GR(V)$, $WR(V)$ and 
$WR{\,}'(V)$ for the abelian groups $KR(V_\C),GR(V_\C),$ etc. 
(The corresponding topological spectra are written in boldface
or blackboard bold. 
We will occasionally write
$X$ for $V_\C$.  As we showed in \cite[1.6]{KW}, if $V$ is 
quasi-projective but not projective then $V_{\C}$  has a compact
Real subspace $X_0$ as a Real deformation retract, constructed as 
the complement of an equivariant regular neighborhood of the complement
of $V_{\C}$ in a projective variety $\bar V_{\C}$.

\medskip
\paragraph{\it Cohomology theories}
Following Atiyah \cite[\S 2]{Atiyah}, we may form cohomology theories on
(finite dimensional)
Real spaces by defining $KR^{-n}(X)$ and $GR^{-n}(X)$ for $n>0$ as in
topological $K$-theory, \textit{i.e.}, by considering bundles on 
suspensions $S^n\wedge X^+$, where $X^{+}=X\cup\{\pt\}$ and
$S^n$ has the trivial involution. However, we will
often write $KR_{n}(X)$ for $KR^{-n}(X)$ and $GR_{n}(X)$ for $GR^{-n}(X)$
in order to have a consistent notation with the algebraic analogues.
As usual, $KR_0(X)$ and $GR_0(X)$ agree with $KR(X)$ and $GR(X)$.

For example, if $V=\Spec(\R)$ (so $X=\pt$ is a point), 
$KR_{n}(\pt)\cong KO_{n}(\pt)$, $GR_{n}(\pt)\cong KO_{n}(\pt)^{2}$ 
and $WR_{n}(\pt)\cong KO_{n}(\pt)$. When $n=0$, we have 
${\mGW_0}(V)\cong GR(\pt)$ and $W(V)\cong WR(\pt)\cong\Z$.
If $V=\Spec(\C)$, so $X$ is $S^{1,0}$ (two points with the nontrivial 
involution), we have $KR_n(S^{1,0})=KU_n$, and $GR(S^{1,0})\cong KO_n$
(as we will see in Example \ref{WR(X^G)}), so $WR_n(S^{1,0})$ is the 
cokernel of the canonical map $KU_n\to KO_n$. When $n=0$, we have
$W(\C)\cong WR(S^{1,0})\cong\Z/2$.

If $V$ is a curve then $K_{0}(V)$ can differ from $KR(V)$,
even modulo~2 (see \cite[4.12.1]{KW}),
and $GW_{0}(V)$ differs from $GR(V)$, but we
shall see in Theorem \ref{curves:W=WR} 
that $W(V)\cong WR(V)$.

There are two non-connective spectra associated with the abelian group
${\mGW_0}(V)$. One, which we write as $\uGW(V)$, 
arises from applying a hermitian analogue of
the Waldhausen infinite loop space
machine (iterated $S_\mathdot$) to the category of strictly perfect
complexes with the usual duality $E^*=\Hom(E,\cO_V)$; 
see \cite[Defn.\,5.4]{Schlichting.Fund}.
The other, which we write as $\GW(V)$,
was introduced in \cite{Schlichting.MV} and is called the
``Karoubi-Grothendieck-Witt spectrum'' in \cite{Schlichting.Fund};
it is obtained using algebraic suspensions; 
see also Definition 8.6 and Remark 8.8 of {\it loc.\,cit.}.
There is
a canonical map $\uGW(V)\to\GW(V)$.
We write ${\mGW_n}(V)$ and ${\kGW_n}(V)$ for the homotopy groups 
of the spectra ${\mGW}(V)$ and ${\kGW}(V)$.
 Since we are assuming that $V$ 
is always a quasi-projective variety, the maps 
$\mGW_n(V)\to\kGW_n(V)$ are isomorphisms for $n\ge0$;
see \cite[8.14]{Schlichting.Fund}.
When $V$ is regular then these maps are isomorphisms for all $n\in \Z$.
By \cite[9.7]{Schlichting.Fund}, $\GW(V)$ satisfies Nisnevich descent;
$\uGW(V)$ does not.

\medskip 
As is customary nowadays, we  can consider the shifted
Grothendieck-Witt groups ${\kGW_{n}^{[i]}}(\cE)=\pi_n\GW^{[i]}(X)$ ($i\in\Z$)
and have ${\kGW_n}(\cE)={\kGW}_n^{[0]}(\cE)$, 
as in \cite[5.7]{Schlichting.Fund}.
These groups fit into long exact sequences 
\begin{equation}\label{seq:GW[i]}
\to {\kGW_{n}^{[i-1]}}(\cE) \map{F} K_{n}(\cE)%
\overset{H}{\to }{\kGW_{n}^{[i]}}(\cE)\to
\kGW_{n-1}^{[i-1]}(\cE) \to.
\end{equation}
where $F$ and $H$ are the forgetful and hyperbolic maps and $K_n$ 
stands for the homotopy groups of the non-connective $K$-theory spectrum; 
see \cite[Thm.\,8.11]{Schlichting.Fund}. 

Following Karoubi, we shall define the $n$-th higher Witt group  
$W_{n}^{[i]}(\cE)$ of $\cE$ to be the cokernel of 
$K_n(\cE)\smap{H} {\kGW_{n}^{[i]}}(\cE)$. 
When $n=0$, $W_0^{[i]}(\cE)$ coincides with the group 
$W^i(\cE)$ defined by Balmer in \cite{Balmer99}, and
$W_n(\cE)=W_{n}^{[0]}(\cE)$. (If $V$ is smooth and $n<0$, then 
$W_n^{[i]}(V)$ also agrees with Balmer's $W^{i-n}(V)$.) We also set
${}_{-1}{\kGW_n}(\cE)={\kGW_{n}^{[2]}}(\cE)$ and 
${}_{-1}W_n(\cE)=W_{n}^{[2]}(\cE)$.

In Section \ref{sec:Brumfiel} and Appendix \ref{sec:Marco} 
we will also use the spectra $GW^{[i]}(V)$ and its homotopy groups 
$GW^{[i]}_n(V)=\pi_nGW^{[i]}(V)$, which generalize $GW^{[0]}(V)=GW(V)$.
The spectra $GW^{[i]}(V)$  have the same connective cover as 
$\GW^{[i]}(V)$, but its negative degree groups $GW_n^{[i]}(V)$
are the Balmer Witt groups $W^{i-n}(V)$ for $n<0$; see
\cite[5.7]{Schlichting.Fund}.  These are the Grothendieck-Witt spectra 
associated with the category of chain complexes of vector bundles on $V$
and duality $E \mapsto\Hom_{\cO_V}(E,\cO_V[i])$.
Again, the comparison map $GW_n^{[n]}(V) \to \GW_n^{[i]}(V)$ 
is an isomorphism for $n\ge0$ in general, 
and for all $n\in\Z$ when $V$ is regular.

When $\cE$ is the (Hermitian) category of algebraic vector bundles 
on $V$, we shall adopt the classical notation $U_{n}(V)$ for
${\kGW_{n}^{[-1]}}(\cE)$, following \cite{MKSLN343}, so that
the sequence \eqref{seq:GW[i]} reduces to the classical sequence
\begin{equation}\label{seq:U-K-GW}
\to U_{n}(V) \map{F} K_{n}(V) \map{H} {\kGW_n}(V) \to U_{n-1}(V)\to.
\end{equation}%
Although ${\kGW_{n+1}^{[1]}}(\cE)$ would be classically written as the
group $V_n(V)$, we will avoid that awkward notation.

When $\cE$ is the Hermitian category $\cE_X$ of Real vector bundles 
on $X$, we shall write $GR_{n}^{[i]}(X)$ for $GW_{n}^{[i]}(\cE_X)$. 
Following this tradition, we shall write $UR_{n}(X)$ for $GR_{n}^{[-1]}(X)$ 
but avoid the awkward $VR_{n}$ notation. In this case, 
\eqref{seq:GW[i]} becomes the exact sequence:
\begin{equation}\label{eq:UR-KR-GR}
KR_{n+1}(X)\!\to \!GR_{n+1}(X)\!\to \!UR_{n}(X)\to
KR_{n}(X)\to GR_{n}(X).  
\end{equation}
In particular, $WR(X)$ is a subgroup of $UR_{-1}(X)$. 
By \cite[\S 5.4]{Schlichting.Fund}, the map 
${\kGW_{\ast }^{[\ast ]}}(V)\to GR_{\ast }^{[\ast ]}({\color{red}V_{\C}})$ is a
homomorphism of bigraded rings.

\section{Real Hermitian $K$-theory and equivariant $KO$}
\label{sec:GR=KOG}

Let us write $G$ for the cyclic group $\{1,\sigma \}$, so that 
a Real space $X$ is a $G$-CW complex. 
The forgetful functor from $\cE_X$ to the
category of $G$-equivariant $\R$-linear vector bundles on $X$ defines
a forgetful functor from $KR(X)$ to the Grothendieck group $KO_{G}(X)$
of $G$-equivariant $\R$-linear bundles on $X$.


Given a $G$-equivariant $\R$-linear vector bundle $F$, we may extend 
the involution on $F$ to an anti-linear involution on the complex bundle
$\C\otimes F$ by $\sigma (v+iw)=\sigma (v)-i\sigma (w)$ ($v,w\in F_{x}$).
This makes $\C\otimes F$ into a Real bundle on $X$, 
and defines a homomorphism 
\begin{equation}\label{eq:KOG-KR}
\rho: KO_{G}(X)\to KR(X).  
\end{equation}
Given a $G$-invariant Riemannian metric on $F$, viewed as an 
$\R$-linear symmetric form $\theta :F\map{} F^{\ast}$ on $F$, the formula
$\theta(v+iw)=\theta(v)+i\theta(w)$ defines a Real symmetric form 
$(\C\otimes F,\theta)$ (as $\theta\sigma=\sigma\theta$).
Since the metric is well defined up to homotopy, the
isomorphism class of $(\C\otimes F,\theta)$ is well defined. This
defines a homomorphism $\psi:KO_{G}(X)\to GR(X)$ which refines
the map $\rho$ in the sense given by our next result.

\begin{theorem}\label{GR=KOG} 
For any Real space $X$, the homomorphism
\[
\psi: KO_{G}(X) \overset{\simeq}\to GR(X)
\]
is an isomorphism of rings. The composition of the isomorphism $\psi$ with 
the forgetful map $GR(X)\to KR(X)$ is the map $\rho$ in \eqref{eq:KOG-KR}.

The hyperbolic map $KR(X)\to GR(X)$ is the composition of the
isomorphism $\psi: KO_{G}(X)\cong GR(X)$ with the map
\[ 
KR(X)\to KO_{G}(X)
\]  
which associates to a Real bundle its underlying real vector bundle,
equipped with the induced action of $G$.
\end{theorem}

\begin{proof}
We first show that the map $KO_{G}(X)\to GR(X)$ is a ring homomorphism.
Given two $G$-equivariant real vector bundles $F_1$ and $F_2$,
with $G$-invariant metrics, the tensor product of the metrics is a
metric on $F_1\otimes F_2$, and $\C\otimes F_1\otimes F_2$ with the
associated form is the product of the $(\C\otimes F_i,\theta_i)$,
as claimed.

Next, we show that $KO_{G}(X)\to GR(X)$ is an isomorphism.
Suppose that $E$ is a Real vector bundle with a nondegenerate Real
symmetric form $\theta$. If $\langle\;,\;\rangle$ is a $G$-invariant
Hermitian 
metric on the complex bundle underlying $E$, define $T:E\to E$ by
the formula $\langle Tv,w\rangle=\theta(v,w)$. 
Then $T$ is $\R$-linear and self-adjoint, because 
\begin{equation*}
\langle v,Tw\rangle=\langle Tw,v\rangle=\theta(w,v)=\theta(v,w)=\langle
Tv,w\rangle.
\end{equation*}
It follows that all eigenvalues of $T$ are real and nonzero. Changing the
metric up to homotopy, we may even assume that $\pm1$ are the only
eigenvalues of $T$. Since the metric and $\theta$ are $G$-invariant, so is $%
T $. Thus the family of $+1$ eigenspaces $F_{x}\subset E_{x}$ form an
equivariant real sub-bundle $F$ of $E$, and the subspaces $iF_{x}$ are the $%
-1$-eigenspaces of $T$, i.e., $E$ is the complexification of $F$, and $F$ is
independent of the choice of Hermitian 
metric. The assignment 
$(E,\theta)\mapsto F$ defines a map $GR(X)\to KO_{G}(X)$, and a
routine check shows that it is inverse to the map $KO_{G}(X)\to GR(X)$.

It is clear that the composition $KO_{G}(X)\to GR(X)\to KR(X)$ is 
$F\mapsto\C\otimes F$. Moreover, the composition 
$KR(X)\to KO_{G}(X)\to GR(X)$ sends a Real bundle $E$ 
(with a $G$-invariant Riemannian metric) to the Real
symmetric form $(\C\otimes_{\R}E,\theta)$, which is
isomorphic to the hyperbolic form $(E\oplus E^{*},h)$. This verifies the
rest of the assertions.
\end{proof}

\begin{subremark}\label{GR*=KO*G}
By taking appropriate suspensions of $X^+$ 
(i.e., $X$ with a disjoint basepoint), we also get isomorphisms
$GR^{-n}(X)\map{\cong} KO_G^{-n}(X)$ for $n>0$, 
as well as quotient isomorphisms
\[
WR^{-n}(X)\map{\cong} \coker\big\{ KR^{-n}(X)\to KO_G^{-n}(X)\big\}.
\] 
Since these isomorphisms are compatible with the external tensor product 
of vector bundles, we have  $\Z$-graded ring isomorphisms 
\[GR^{*}(X)\map{\cong} KO_G^{*}(X) ~\textrm{and}~
WR^{*}(X)\map{\cong}\coker\!\big\{ KR^{*}(X)\to KO_G^{*}(X)\big\}.
\]
This generalizes the observation in Definition \ref{def:GR} that
$WR(X)$ is a quotient ring of $GR(X)$.
\end{subremark}

As mentioned after Definition \ref{def:WR}, quasi-projective varieties
which are not compact still have the homotopy type of a finite
$G$-CW complex. Recall our notational convention that we write
$GR_n$ for $GR^{-n}$ and $WR_n$ for $WR^{-n}$.

\begin{corollary}\label{GR.fingen} 
Suppose that $X$ has the homotopy type of a finite $G$-CW complex, 
such as the Real space associated to a variety over $\R$. 
Then the groups $GR_{n}(X)$ and $WR_{n}(X)$
are finitely generated and 8-periodic with respect to $n.$
\end{corollary}

\begin{proof}
As both $KR^*(X)$ and $KO_G^*(X)$ are 8-periodic, and homotopy invariant,
this is immediate from Theorem \ref{GR=KOG} and Remark \ref{GR*=KO*G}.
\end{proof}


\begin{examples}\label{WR(X^G)}
(a)
If $G$ acts trivially on $X$, $WR(X)\cong WR{\,}'(X)\cong KO(X)$
as abelian groups.
Indeed, $KR(X)\cong KO(X)$, $GR(X)\cong KO_{G}(X)\cong KO(X)\oplus KO(X)$;
the map $KR(X)\to GR(X)$ is identified with the diagonal, 
while the forgetful map is identified with addition.

In this case, the cokernel of $KO(X)\to KO_G(X)$ is isomorphic to the
ring $KO(X)$, so by the above remarks the map $WR(X)\to KO(X)$ is a
ring isomorphism. This is a version of Brumfield's generalization 
of the signature map $W(X)\to\Z^c$ mentioned in the Introduction.

(b)
At the other extreme, when $G$ acts freely on $X$, 
Theorem \ref{GR=KOG} allows us to identify the hyperbolic map with the map 
\begin{equation*}
KR(X)\to KO(X/G)\cong KO_{G}(X)\cong GR(X).
\end{equation*}
\end{examples}

\begin{example}\label{WRfree} 
Following Atiyah's notation,
let $S^{p,0}$ denote the sphere $S^{p-1}$ with the antipodal
involution; $S^{p,0}/G$ is the $(p\!-\!1)$-dimensional real 
projective space, $\RP^{p-1}$. If $p\geq 3$, the isomorphism 
$KR(S^{p,0})\cong \Z\oplus KO^{p+1}(\pt)$ is derived in Appendix \ref{errata1},
and we have $GR(S^{p,0})\cong KO(\RP^{p-1})$ by 
Theorem \ref{GR=KOG} and Example \ref{WR(X^G)}(b), 
so the hyperbolic map is described as a sum%
\[
KR(S^{p,0})\cong KO^{0}(\pt)\oplus KO^{p+1}(\pt)\to KO(\RP^{p-1})
\]
The group $KO(\RP^{p-1})$ has been computed by Adams \cite{Adams}:
it is the direct sum of the rank summand $KO^{0}(\pt)\cong \Z$ and 
a cyclic group of exponent $2^{f},$ where $f$ is the number of 
integers $i$ in the range $0<i<p$ 
with $i\equiv0,1,2$ or $4$ mod $8$ (see for instance 
\cite[IV.6.46]{MKbook}).  
If $p\equiv0,1,2,4\pmod8$, the second component 
$KO^{p+1}(\pt)\to KO(\RP^{p-1})$ must be reduced to 0, 
because $KO^{p+1}(\pt)=0$.  Thus the Witt group is
\[
WR(S^{p,0}) \cong \widetilde{KO}(\RP^{p-1}) \cong \Z/2^f\Z
\]
if $p\equiv0,1,2,4\pmod8$. For the other values of $p$,
the Witt group equals either $\Z/2^f$ or $\Z/2^{f-1}$. 
To see this, let $r$ be maximal such that $r<p$ and
$r\equiv0,1,2,4\pmod8$. Then $KO(\RP^{p-1})\to KO(\RP^{r-1})$ 
is surjective with kernel $\Z/2$, and the composition from
$KO^{p+1}(\pt)$ to $KO(\RP^{r-1})$ factors through
$KO^{r+1}(\pt)=0$ by naturality.



This example shows that we can get arbitrary $2$-primary
order in the topological Witt groups $WR(S^{p,0})$ by varying $p$.
\end{example}

\begin{example}
Let $V'$ be a complex variety, and $V$ this variety regarded as
defined over $\R$. By choosing real coordinates $(z,\overline{z})$ 
on $V$, we see that the Real space $X$ of $V$ is $G\times Y$, 
where $Y=V'(\C)$. The Real Witt group $WR(V)$ is the cokernel of 
$KR(G\times Y)\to KO_{G}(G\times Y)$; this map is the realification map in
topological $K$-theory $KU(Y)\to KO(Y)$, since $KR(G\times Y)=KU(Y)$ and
$KO_{G}(G\times Y)=KO(Y)$. Of course, $W(V)$ is an
algebra over $WR(\Spec\C)=\Z/2$.
\end{example}

\begin{example}\label{ex:canonical}
When $G$ acts freely on $X$, there is a canonical element $\gamma$ of
$GR(X)$ and hence $WR(X)$, given by the symmetric form $-1$ on
the trivial complex bundle $E=X\times\C$ with involution
$\sigma{(x,z)}=(\sigma(x),\bar{z})$. Since it has rank~1, 
$\gamma$ is nonzero; its image is nonzero under the augmentation 
$WR(X)\to\Z/2$ defined by any connected component of $X/G$.
There is also an $\R$-linear line bundle $L=X\times_G\R$ over $X/G$,
associated to the principal $G$-bundle $X\to X/G$
(where $G$ acts on $\R$ by the sign representation);
the isomorphism $KO(X/G)\map{\cong} KO_G(X)$ sends $L$
to $L_X$, the trivial bundle $X\times\R$ with the involution
$(x,t)\mapsto(\sigma(x),-t)$, and the isomorphism 
$$\psi: KO_G(X)\map{\cong}GR(X)$$ 
of Theorem \ref{GR=KOG} 
identifies $L_X$ with the canonical element $\gamma$.
\end{example}

\begin{lemma}\label{lem:L=-1}
Suppose that $G$ acts freely on $X$. If $F$ is an $\R$-linear
$G$-vector bundle on $X$ with $\psi(F)=(E,\varphi)$ then
$\psi(F\otimes L_X)=(E,-\varphi)$.

Hence the composition $GR(X)\map{} KR(X) \map{H} GR(X)$
sends $\psi(F)$ to $\psi(F\otimes(1+L_X))$, and $WR(X)$ is a
quotient of $KO_G(X)/(1+L_X)$.
\end{lemma}

\begin{proof}
Since $\psi$ is a ring homomorphism, if suffices to recall from
Example \ref{ex:canonical} that $\psi(L_X)$ is $\gamma=(X\times\C,-1)$. 
\end{proof}

\medskip\goodbreak
The Grothendieck-Witt group ${}_{-1}GR(X)$, which is associated to
skew-symmetric forms, is related to the Grothendieck group 
$KR_{\H}(X)$ of Real quaternion bundles on $X$ 
by Theorem \ref{-GR=KRH} below. The cohomology theory for $KR_{\H}$
is developed in \cite{KW:KRA}; here is a 1-paragraph summary.

Letting $\sigma:\H\to\H$ denote conjugation by $j$, a
{\it Real quaternion bundle} on a Real space $X$ is a 
quaternion bundle $E$ (each fiber $E_x$ has the structure of
a left $\H$-module) with an involution $\tau$  
on $E$ compatible with the involution on $X$
such that each map $\tau:E_x\to E_{\sigma(x)}$ satisfies
\[
\tau(a\cdot e) = \sigma(a)\cdot \tau(e),
\qquad a\in \H,\quad e\in E_x.
\]
Since $\sigma(i)=-i$, the underlying complex vector bundle has the
structure of a Real vector bundle in the sense of Atiyah \cite{Atiyah}.

We write $KR_{\H}(X)$ for the Grothendieck group of Real quaternion
bundles on $X$. Passing to the underlying Real vector bundle induces a
canonical map $KR_\H(X) \to KR(X)$. Here is the skew-symmetric
analogue of Theorem \ref{GR=KOG}.


\begin{theorem}\label{-GR=KRH} 
The Grothendieck-Witt group $_{-1}GR(X)$ 
of a Real space $X$ is naturally isomorphic to $KR_{\H}(X)$.

Using this isomorphism, the forgetful map $_{-1}GR(X)\to KR(X)$ may
be identified with the canonical map $KR_{\H}(X)\to KR(X)$.
\end{theorem}

\begin{proof}
Given a Real quaternion bundle $E$, choose a $\tau$-equivariant 
Hermitian metric $\langle\;,\;\rangle$ on the underlying complex bundle.
Replacing the metric by 
$\langle e_1,e_2\rangle+ \overline{\langle je_1,je_2\rangle}$,
we may also assume that $\langle je_1,je_2\rangle$ is the complex
conjugate of $\langle e_1,e_2\rangle$. Then $\theta$, defined by
\[
\theta(e_1,e_2)=\langle j e_1,e_2\rangle,
\]
is a skew-symmetric $\C$-bilinear form: 
$\theta(ae_1,be_2)=ab\theta(e_1,e_2)$ and
\[
\theta(e_2,e_1)=\langle j e_2,e_1\rangle =
\overline{\langle e_1,je_2\rangle} =
\langle j e_1,j^2 e_2\rangle = -\theta(e_1,e_2).
\]
Thus $(E,\tau,\theta)$ defines an element in ${}_{-1}GR(X)$. 
Since the metric is well defined up to homotopy,
this yields a map $KR_{\H}(X)\to {}_{-1}GR(X)$.
By inspection, the composite with the forgetful map 
${}_{-1}GR(X)\to KR(X)$ is the canonical map $KR_\H(X) \to KR(X)$.

Conversely, suppose that $E$ is a Real vector bundle, 
equipped with a non-degenerate skew-symmetric $\C$-bilinear form $\theta$. 
Choose a $\sigma$-equivar-iant Hermitian metric $\langle\;,\;\rangle$ 
on $E$, and define the automorphisms $J$, $J^*$ of $E$ by 
$\langle Je_1,e_2\rangle=\theta(e_1,e_2) = 
 \overline{\langle e_1,J^*e_2\rangle}$.
(See \cite[Ex.\,I.9.21 and I.9.22c]{MKbook}.) 
Then $J$ commutes with $\sigma$ and is complex antilinear:
\[
\langle J(ie_1),e_2\rangle = \theta(ie_1,e_2) = i\theta(e_1,e_2)
=i \langle Je_1,e_2\rangle = \langle (-i)Je_1,e_2\rangle.
\]
Moreover, $J^*=-J$ because
\[
\langle Je_1,e_2\rangle=\theta(e_1,e_2)=-\theta(e_2,e_1)=
\overline{-\langle e_2,J^*e_1\rangle} = -\langle J^*e_1,e_2\rangle.
\]
Since $\langle JJ^*e,e \rangle = \langle J^*e,J^*e\rangle>0$ for all $e$,
we may change the metric up to homotopy to assume 
that $JJ^*=1$ and hence that $J^2=-1$. 
Thus $(E,i,J)$ is a quaternionic bundle on $X$. 

Since the choices are well defined up to homotopy, this gives a map from $%
{}_{-1}GR(X)$ to $KR_{\H}(X)$. By inspection, this map is an inverse to the
map in the first paragraph.
\end{proof}

\begin{corollary}\label{cor:8periodic}
Suppose that $X$ has the homotopy type of a finite $G$-CW complex, 
such as the Real space associated to a variety over $\R$. 
Then the groups ${}_{-1}GR_{n}(X)$ and ${}_{-1}WR_{n}(X)$ 
are finitely generated and 8-periodic with respect to $n.$

The groups $GR^{[i]}_n(X)$
are all finitely generated, 4-periodic in $i$ and 8-periodic in $n$.
\end{corollary}

\begin{proof}
We have $GR^{[i]}_n(X)\cong GR^{[i+4]}_n(X)$ by
\cite[5.9]{Schlichting.Fund}. The finitely generation
is obtained from Corollary \ref{GR.fingen},
\eqref{seq:GW[i]} and induction on $i$.
\end{proof}

\begin{remark}\label{rem:GW^c}
We can also introduce comparison Grothendieck-Witt groups 
$GW_{n}^{c}(V)=\pi_n \GW^{c}(V)$
which fit into exact sequences%
\begin{equation*}
{\kGW_{n+1}}(V)\to GR_{n+1}(V)\to GW_{n}^{c}(V)\to
{\kGW_{n}}(V)\to GR_{n}(V).
\end{equation*}%
Here $GW^{c}(V)$ is the homotopy fiber of $\GW(X)\to\mathbb{GR}(X)$.
This sequence will be used repeatedly in the next few sections.
\end{remark}

\newpage\goodbreak

\section{Algebraic $K$-theory of varieties over $\R$, revisited}
\label{sec:KR}

In this section we revisit the results found in our paper \cite{KW} 
and generalize them in two directions. Theorem \ref{singular.KW} below
gives a sharper bound in the comparison of algebraic $K$-theory 
with Atiyah's Real $K$-theory, and also drops the assumption of 
smoothness (thanks to a better understanding of varieties over $\R$ of
infinite \'{e}tale dimension). To state the result, we need some notation.

Given a variety $V\!,$ let $\bbK(V)$ denote the (non-connective)
spectrum representing the algebraic $K$-theory of $V$. We write $K_{n}(V)$
for $\pi_n\bbK(V)$ and $K_{n}(V;\Z/q)$ for the homotopy groups 
$\pi_{n}(\bbK(V);\Z/q)$ of $\bbK(V)$ with coefficients $\Z/q$.

If $X$ is a Real space, let $\mathbb{KR}(X)$ denote the spectrum
representing Atiyah's $KR$-theory on $X.$ 
We write $KR_{n}(X)$ and $KR_{n}(X;\Z/q)$ 
for the homotopy groups of this spectrum. 

If $V$ is a variety over $\R$, the passage from algebraic vector 
bundles on $V$ to Real vector bundles on $V_\C$ induces a spectrum map 
$\bbK(V)\to \mathbb{KR}(V_\C)$; on homotopy groups it yields the
maps $K_n(V)\to KR_n(V_\C)$ studied in \cite{KW}.
%
By abuse of notation, we shall write $KR_{n}(V)$ for $KR_{n}(V_{\C})$, etc. 
Here is the main result of this section.

\begin{theorem}\label{singular.KW} 
Let $V$ be a $d$-dimensional variety over $\R$,
possibly singular, with associated Real space $V_{\C}$. 
For all $q$, the map $\bbK(V)\to\mathbb{KR}(V)$ 
induces an isomorphism 
\begin{equation*}
K_{n}(V;\Z/q)\to KR_{n}(V;\Z/q)
\end{equation*}
for $n\geq d-1$, and a monomorphism for $n=d-2.$
\end{theorem}

The corresponding theorem in \cite[4.8]{KW} required $V$ to be nonsingular,
and only established the isomorphism for $n\geq d$. Similar results have
been proved independently in \cite{Heller-Voineagu}.

\begin{proof}
First we assume that $V$ is a nonsingular variety. 
When $V(\R)=\varnothing$, the theorem was established in \cite[4.7]{KW}; 
when $V(\R)\neq\varnothing$, it was established in \cite[4.8]{KW}, except
for the cases $n=d-1,d-2$. These two cases were handled by Rosenschon and 
{\O}stvaer in \cite[Thm.\,2]{RO}.

For singular $V$, we first note that $\bbK(V;\Z/q)\cong\bbK(V_\red;\Z/q)$ and
$\bbK(V;\Z/q)\,\map{\cong}\, \mathbb{KH}(V;\Z/q)$,
where $\mathbb{KH}$ denotes homotopy $K$-theory;
if $V$ is affine, this is \cite[1.6 and 2.3]{Weibel.KH}; 
the general case follows by Zariski descent.
Since $\mathbb{KR}(V)=\mathbb{KR}(V_\red)$, 
we may assume $V$ is reduced.  

We deduce the result for
reduced varieties by induction on dim($V$). 
Given $V$ with singular locus 
$Z$, choose a resolution of singularities, $V'\to V$, and
set $Z'=Z\times_{V}V'.$ 
By Haesemeyer  \cite{Haesemeyer}, $\mathbb{KH}(V;\Z/q)$ satisfies 
$cdh$ descent. It follows that we have a fibration
sequence 
\begin{equation*}
\bbK(V;\Z/q)\to\bbK(V';\Z/q)\times\bbK(Z;\Z/q)\to\bbK(Z';\Z/q).
\end{equation*}
On the other hand, by excision for $(V_{\C}',Z_{\C}')\to(V_{\C},Z_{\C})$, 
we have a fibration sequence 
\begin{equation*}
\mathbb{KR}(V;\Z/q)\to\mathbb{KR}(V';\Z/q)\times\mathbb{KR}(Z;\Z/q)
\to\mathbb{KR}(Z';\Z/q).
\end{equation*}
There is a natural map between these fibration sequences. 
Because $\dim(Z)$ and $\dim(Z')$ are at most $d-1$, we know by 
induction that the maps $K_{n}(Z';\Z/q)\to KR_{n}(Z';\Z/q)$ 
are isomorphisms for $n\geq d-2$, and injections for $n=d-3$.

Since $V'$ is smooth, $K_{n}(V';\Z/q)\to KR_{n}(V';\Z/q)$ 
is an isomorphism for $n\ge d-1$, and an injection for $n=d-2$. 
The result now follows from the 5--lemma,
applied to the long exact homotopy sequences.
\end{proof}

Let $\bbK^{c}(V)$ denote the homotopy fiber of 
$\bbK(V)\to \mathbb{KR}(V)$, and set $K_{n}^{c}(V)=\pi _{n}\bbK^{c}(V)$.
These are called the {\it $K$-theory comparison groups}.
Note that the previous theorem may be stated as 
$K_{n}^{c}(V;\Z/q)=0$ for $n\geq d-2.$
Since $q$ is arbitrary, the universal coefficient formula yields:

\begin{corollary}
\label{Kc-u.d.} If $n\geq d-2$, the groups $K_{n}^{c}(V)$ are uniquely
divisible. Moreover, 
the group $K_{d-3}^{c}(V)$ is torsionfree.
\end{corollary}


\begin{example}\label{ex:KRcurve} 
Let $V$ be a smooth projective curve of genus $g$,
defined over $\R$, such that $V_{\R}$ is not empty but has 
$\nu>0$ connected components (circles). Using Theorem \ref{singular.KW},
the finite generation of $KR^*(V)$ 
and the calculations of 
$K_*(V)$ in \cite[0.1]{PW}, it is not hard to check that 
\begin{align*}
KR^{0}(V)& =\Z^{2}\oplus (\Z/2)^{\nu -1} & KR^{4}(V)& =%
\Z^{2} \\
KR^{1}(V)& =\Z^{g} & KR^{5}(V)& =\Z^{g}\oplus (\Z%
/2)^{\nu -1} \\
KR^{2}(V)& =0 & KR^{6}(V)& =(\Z/2)^{\nu +1} \\
KR^{3}(V)& =\Z^{g} & KR^{7}(V)& =\Z^{g}\oplus (\Z%
/2)^{\nu +1}. 
\end{align*}%
For example, the calculation $KR^3(V)=KR_5(V)$ follows from the fact that
$K_4(V)_\tors=(\Q/\Z)^g$ and $K_5(V)$ is divisible \cite[0.1]{PW}. 
By Theorem \ref{singular.KW}, 
$KR_5(V;\Z/q)\cong K_5(V;\Z/q)$ is $(\Z/q)^g$ for all $q$,
and the result follows from universal coefficients.
%
The calculation of $KR^5(V)=KR_3(V)$ follows from 
$K_2(V)_\tors=(\Q/\Z)^g\oplus(\Z/2)^{\nu+1}$ and
$K_3(V)=D\oplus(\Z/2)^{\nu-1}$, $D$ divisible. Hence
$KR_3(V;\Z/2q)\cong K_3(V;\Z/2q)$ is $(\Z/2q)^g\oplus(\Z/2)^{2\nu-1}$,
but the image of $KR_3(V;\Z/4q)\to KR_3(V;\Z/2q)$ is the group
$(\Z/2q)^g\oplus(\Z/2)^{\nu-1}$ for all $q$.

The image of $K_n(V)\to KR^{-n}(V)$ is the torsion subgroup for all $n>0$.
For $n=0$, we have
$K_{0}(V)\cong \Z^{2}\oplus (\Z/2)^{\nu-1}\oplus(\R/\Z)^{g}$ 
by \cite[1.1]{PW90}; this group
surjects onto $KR^{0}(V)$, and $(\R/\Z)^{g}$ is the image of 
$K_{0}^{c}(V)\cong \R^{g}$ in $K_{0}(V)$. We also have $%
K_{-1}^{c}(V)=0$ and $K_{-2}^{c}(V)\cong \Z^{g}$.
\end{example}

\begin{example}\label{KR.affinecurve}
Suppose that $V$ is a smooth affine curve, obtained from a smooth
irreducible projective curve $\bar V$ of genus $g$ by removing $r>0$ points.
Then $KR_0(V)\cong\Z\oplus(\Z/2)^\lambda$, where $\lambda\ge0$ is the number
of closed components (circles) of $V_\R$, and $K_0(V)\to KR_0(V)$
is a surjection whose kernel is the quotient of $(\R/\Z)^g$ by a
finitely generated subgroup. 

This calculation follows from the 
compatibility of the localization sequence 
\vspace{-5pt}
\[
\Z^r \to K_0(\bar V) \to K_0(V) \to 0
\]
with its $KR$-analogue \cite[3.4]{KW}, and \cite[1.4]{PW90}.
\end{example}

\bigskip\goodbreak

\section{Witt groups for curves over $\R$}
\label{sec:curves}

The purpose of this section is to show that the groups $W(V)$ 
and $WR(V)$ are isomorphic for curves $V$\!, and to explicitly 
compute this Witt group when $V$ is smooth projective. Our computations 
recover some older results of Knebusch \cite{Knebusch} for $W(V)$. 
We will prove analogous results for the co-Witt groups $W'$ of curves
in Section \ref{sec:co-Witt}.

The results in this section are intended to illustrate the more general
results which we will give later on in the paper, although
the results in this section will not be used in the rest of this paper.
Our proofs will use results from Sections \ref{sec:williams}, 
\ref{sec:WR} and \ref{sec:exponents} below.

\smallskip

Here is our main theorem about curves. It is reminiscent of the 
results in \cite{PW} for $K$-groups of curves. 


\begin{theorem}\label{curves:W=WR} 
Let $V$ be any curve over $\R$, possibly singular. 
Then the natural map 
\begin{equation*}
W(V)\overset{\theta}{\to}WR(V)
\end{equation*}%
is an isomorphism.
\end{theorem}

The skew-symmetric analogue ${}_{-1}W(V) \map{\sim} {}_{-1}WR(V)$
is also true, but not very interesting, because we shall see in
Proposition \ref{skewWR(curve)} that both terms are zero.

\begin{proof}[Proof of Theorem \ref{curves:W=WR}]
Consider the following commutative diagram,
where, as in \ref{rem:GW^c} and \ref{Kc-u.d.}, the superscript `${c}$' 
in $K^{c},GW^{c}$ and $U^{c}$ means ``comparison groups.'' 
\begin{equation*}
\begin{array}{ccccccc}
K_{0}^{c}(V) & \overset{h}{{\longrightarrow }} & GW_{0}^{c}(V) & \to
& U_{-1}^{c}(V) & \to & K_{-1}^{c}(V) \\ 
\downarrow &  & \downarrow &  & \downarrow &  & ~\downarrow 0 \\ 
K_{0}(V) & \overset{h}{{\longrightarrow }} & GW_{0}(V) & \to & 
U_{-1}(V) & \to & K_{-1}(V) \\ 
\downarrow &  & \downarrow &  & \downarrow &  & \downarrow \\ 
KR_{0}(V) & \overset{h}{{\longrightarrow }} & GR_{0}(V) & \to & 
UR_{-1}(V) & \to & KR_{-1}(V) \\ 
&  &  &  & \downarrow &  &  \\ 
&  &  &  & U_{-2}^{c}(V). &  & 
\end{array}%
\end{equation*}%
The second and third rows are the exact sequences \eqref{seq:U-K-GW} and 
\eqref{eq:UR-KR-GR}, with homological indexing. 
By Bass-Murthy \cite[Ex.\,III.4.4]{WK}, $K_{-1}(V)$ is a free abelian group.
Since $K_{-1}^{c}(V)$ is 2-divisible by Corollary \ref{Kc-u.d.},
the map $K_{-1}^{c}(V)\to K_{-1}(V)$ is zero, 
and $K_{-1}(V)$ injects into $KR_{-1}(V)$.

By a diagram chase, the kernel of $\theta:W(V)\to WR(V)$ 
is a quotient of $U_{-1}^{c}(V)$, which is 2-divisible
by Corollary \ref{U=UR}, and the cokernel of $\theta$ injects into 
$U_{-2}^{c}(V)$, which has no 2-torsion by Corollary \ref{U=UR}. 
By Theorem \ref{W=WR.mod2tors} and Theorem \ref{W-WR.exponent}, 
$\ker(\theta)$ is a 2-primary group of bounded exponent, 
and $\coker(\theta)$ is a finite 2-group
(the cokernel is finitely generated by Corollary \ref{GR.fingen}.)
It follows that $\ker(\theta)$ and $\coker(\theta)$ are zero.
\end{proof}

\medskip

To illustrate the role of the Real Witt group, we now calculate $WR(V)$ 
more explicitly when $V$ is a smooth projective curve over $\R$.
It will be a function of the genus $g$ and the number $\nu $ of components
of the space of real points of $V$. Given Theorem \ref{curves:W=WR}, that 
$W(V)$ is isomorphic to the topological group $WR(V)$, we recover and
slightly improve the algebraic calculations of Knebusch \cite{Knebusch}.

\begin{subex}
If $V$ is a curve defined over $\C$ (a complex curve of genus $g$),
it is well known that $W(V)\cong(\Z/2)^{2g+1}$; 
see \cite[p.\,263]{Knebusch.Queens}. 

To compute $WR(X)$, we must distinguish between the Riemann surface 
$X'=\Hom_\C(\Spec\,\C,V)$ of complex points over $\C$ and the space 
$X=\Hom_\R(\Spec\,\C,V)$ of complex points over $\R$; 
we have $X=X'\times G$.  
Since $KR(X)=KU(X')$ and $GR(X)=KO_G(X)=KO(X')$, $WR(X)$ is the 
cokernel of $KU(X')\to KO(X')$. A simple calculation using the
Atiyah-Hirzebruch spectral sequences for $KU$ and $KO$
shows that $W(V)\cong(\Z/2)^{2g+1}$.
\end{subex}

If $V$ is geometrically connected (i.e., not defined over $\C$),
the associated Real space $X=V_{\C}$ is a Riemann
surface of genus $g$ with an involution $\sigma $. Thus $\sigma $ acts on 
$H^{1}(X,\Z)\cong \Z^{2g}$. If we take coefficient $\Z[1/2]$, 
we get an eigenspace decomposition for this induced involution.

\begin{lemma}\label{H1.decomp}
Both eigenspaces of $\sigma$ on $H^{1}(X;\Z[1/2])$ have rank $g$.
\end{lemma}

\begin{proof}
(Cf.\,\cite[2.3]{PW}) We consider $H=H^{1}(X;\Z)$ as a module over
the group $G=\{1,\sigma \}$. As an abelian group, $H\cong \Z^{2g}$
so if $\Z(1)$ denotes the sign representation then $H^{1}(X;\mathbb{Z%
})=\Z^{a}\oplus\Z(1)^{b}\oplus\Z[G]^c$, where $%
a+b+2c=2g$. We need to show that $a+c=b+c=g$.

Now the complex variety $V'=V\otimes_{\R}\C$ has reduced Picard group
$\Pic^0(V')\cong H\otimes(\R/\Z)$, so we have 
$\Pic^0(V')^G=(\R/\Z)^{a+c}\oplus(\Z/2)^b$. 
If the real points $V_{\R}$ consist of $\nu>0$ circles, 
then Weichold proved (in his 1882 thesis \cite{Whd}) that 
$\Pic^0(V)\cong (\R/\Z)^g\times(\Z/2)^{\nu-1}$; 
see \cite[1.1]{PW}. In this case, $\Pic^0(V)\cong\Pic^0(V')^G$, 
and it follows that $a+c=g$.

If there are no real points on $V$, then Klein proved 
(in the 1892 paper \cite{Klein}) that 
$\Pic^0(V)\cong (\R/\Z)^g$. In this case $\Pic^0(V)$ is
a subgroup of $\Pic^0(V')^G$ of finite index, so again $a+c=g$. (More
precisely, if $g$ is even, then $a=b=0$, while if $g$ is odd then $a=b=1$;
see \cite[1.1.2]{PW90}.)
\end{proof}

If $X$ is compact, we shall write $(X\times\R)^+$ for the 1-point
compactifi\-cation $(X\times\R)\cup\{\pt\}$ 
of $X\times\R$. If $X$ is a Real space, we regard $(X\times\R)^+$ 
as the Real space in which the involution sends $(x,t)$ to $(\sigma x,-t)$.
The reduced group $\widetilde{KO}_G((X\times\R)^+)=
KO_G((X\times\R)^+,\pt)$ is written as $KO_G(X\times\R)$ in Appendix
\ref{app:Banach}, to be consistent with the notation
in \cite[II.4.1]{MKbook}. We avoid this notation here, 
as it conflicts with our notation for $KO_G(X)$ when $X=V_\C$
for a non-projective variety.

\begin{theorem}\label{WR=KOG} 
Let $V$ be a smooth projective curve over $\R$ and
let $X=V_{\C}$ be the space of its complex points. Then there is an
isomorphism:%
\begin{equation*}
WR(X)\cong \widetilde{KO}_{G}((X\times\R)^+).
\end{equation*}
\end{theorem}

\begin{proof}
Since $\dim (X)=2$, we have $KU^{1}(X)\cong H^{1}(X,\Z)\cong{\Z}^{2g}$ 
and (by Example \ref{ex:KRcurve}) $KR^{1}(X)\cong \Z^{g}$.
These are connected by the complexification $KR^{\ast}\to KU^{\ast}$ 
and realification maps $KU^{\ast }\to KR^{\ast }$; see \cite{Atiyah}. 
The composition $KU^*\to KR^*\to KU^*$ is $1-\sigma $, since it is 
induced by the functor sending $E$ to $E\oplus \sigma ^{\ast }(\bar{E})$.
Inverting $2$, the invariant subspace of $-\sigma $ acting on 
$KU^{1}(X)[1/2]$ is $\Z[1/2]^{g}$, by Lemma \ref{H1.decomp}. Thus 
$KR^{1}(X)\cong \Z^{g}$ injects into $KU^{1}(X)^{-\sigma }$ with
index a power of~2. 
Since the composition 
\begin{equation*}
KR^{1}(X)\to KO_{G}^{1}(X)\to KO^{1}(X)\to KU^{1}(X)
\end{equation*}%
is injective, so is the first map $KR^{1}(X)\to KO_{G}^{1}(X)$. We
now use the exact sequence established in Theorem \ref{cup-product}, 
\begin{equation*}
KR(X)\to KO_{G}(X)\to\widetilde{KO}_G((X\!\times\!\R)^+)\to
KR^{1}(X)\to KO_{G}^{1}(X),
\end{equation*}%
and the identification of Theorem \ref{GR=KOG} to conclude that 
$WR(X)$, the cokernel of $KR(X)\to KO_{G}(X))$,
is $\widetilde{KO}_G((X\!\times\!\R)^+)$.
\end{proof}

Let $G=C_{2}$ act on $[-1,1]$ as multiplication by $\pm 1$. Recall from 
\cite[II.4]{MKbook} that the group $\widetilde{KO}_G((X\!\times\!\R)^+)$ 
is isomorphic to the relative group 
$KO_{G}(X\times[-1,1],X\times S^{0})$; since
$KO_G^*(X\times[-1,1])\cong KO_G^*(X)$ and 
$KO_G^*(X\times S^0)\cong KO^*(X)$, we have an exact sequence:
\vspace{6pt}
\begin{equation}\label{seq:KOG-KO}
KO_{G}^{-1}(X)\!\to\!KO^{-1}(X)\!\to\!\widetilde{KO}_G((X\!\times\!\R)^+)%
\!\to\! KO_{G}(X)\!\to\! KO(X)
\end{equation}

\begin{lemma}\label{Witt:Z/4}
Suppose that $V$ has no $\R$ points, and is 
geometrically connected (i.e., not defined over $\C$). 
Then $WR(V)$ surjects onto $\Z/4$.
\end{lemma}

\begin{proof}
Let $F$ denote the function field of $V$.
The generator of $W(F)$ is represented by 
$\langle1\rangle=-\langle-1\rangle$ so twice it is
$2\langle1\rangle=\langle1\rangle-\langle-1\rangle$, 
and the signature map $I(F)/I^2(F)\smap{\cong} F^\times/F^{\times2}$ 
sends $2\langle1\rangle$ to $[-1]$.
Since $I(F)/I^2(F)$ is a square-zero ideal in $W(F)/I^2(F)$,
there is a surjection from $W(F)/I^2(F)$ to $\Z/4$ if and only if
$2\langle1\rangle\ne0$ in $I(F)/I^2(F)$, which is true
only when $-1$ is a square in $F$, i.e., 
iff $F$ contains $\C$. 
\end{proof}

\begin{subremark}
Let $V$ be a smooth variety over $\R$ of dimension $\ge2$
which is geometrically connected (i.e., not defined over $\C$).
If $V$ has no $\R$ points, then the ring $WR(V)$ surjects onto $\Z/4$.
This is immediate from Lemma \ref{Witt:Z/4}, given
Bertini's Theorem and Zariski's Main Theorem
(see \cite[III.7.9]{Hart}), which guarantees that $V$ contains a 
smooth curve $C$ which is also geometrically connected; the ring map
$WR(V)\to WR(C)\cong W(C)\to\Z/4$ sends the image of the subring $\Z$
onto $\Z/4$.
\end{subremark}

\begin{theorem}\label{W(curve)}
Let $V$ be a geometrically connected, smooth projective algebraic curve
over $\R$ of genus $g$ without $\R$-points.
Then the Witt group $W(V)$ is 
\begin{equation*}
W(V)\cong WR(V)\cong 
\Z/4\oplus (\Z/2)^{g}. 
\end{equation*}
\end{theorem}

Theorem \ref{W(curve)} recovers 
Knebusch's result in \cite[Corollary 10.13]{Knebusch}) that the 
kernel $I$ of $W(V)\to\Z/2$ is $(\Z/2)^{g+1}$. 


\begin{proof}
Since $X^G=\emptyset$, $G=C_{2}$ acts freely on $X$ and if we set $Y=X/G$
then $KO^*_G(X)=KO^*(Y)$. Therefore \eqref{seq:KOG-KO} becomes
the exact sequence%
\begin{equation*}
KO^{-1}(Y)\to KO^{-1}(X)\to \widetilde{KO}_G((X\!\times\!\R)^+)
\to KO(Y)\to KO(X).
\end{equation*}%
In fact, $\widetilde{KO}_G((X\times\R)^+)$ may be identified with 
the reduced $KO$-group of the cone $K$ of the map $X\to Y$. 
Using $\chi(X)=2\chi(X/G)=2-2g$, 
we see that $H^{1}(X/G;\Z/2)\cong (\Z/2)^{g+1}$\!. Therefore 
$H^2(K;\Z/2)\cong(\Z/2)^{g+1}$.

We now use a classical result: if $M$ is a connected $CW$-complex of
dimension $\leq 3,$ 
the Atiyah-Hirzebruch spectral sequence collapses to yield an isomorphism
\begin{equation*}
KO(M)\cong \Z\times H^{1}(M;\Z/2)\times H^{2}(M;\Z/2),
\end{equation*}%
and (by the Cartan Formula for Stiefel-Whitney classes)
the group law $\ast $ on the right hand side is given by the formula%
\begin{equation*}
(n,w_{1},w_{2})\ast (n',w_{1}',w_{2}^{\prime
})=(n+n',w_{1}+w_{1}',w_{2}+w_{2}^{\prime
}+w_{1}w_{1}').
\end{equation*}%
It follows that $\widetilde{KO}(K)$ is an extension of 
$H^1(K)\cong\Z/2$ (with generator $x$) by $H^2(K)\cong(\Z/2)^{g+1}$, 
with $x*x=(0,x^2)$.
\end{proof}

Now consider the case when $V$ has $\R$-points. These points form
the subspace $X^G$ of $X$, which is homeomorphic to a disjoint union of 
$\nu>0$ copies of $S^1$ (with $G$ acting trivially on $S^1$).
From Example \ref{WR(X^G)}, we know that 
$WR(S^1)\cong KO(S^1)\cong\Z\oplus\Z/2$ and 
if $x$ is any $\R$-point on $S^1$ then the map $WR(S^1)\to WR(x)\cong\Z$
is a split surjection. Picking a $\R$-point on each component of $X^G$
gives a map $\tau:WR(X)\to\Z^\nu$, independent of the choice of the points.
We can now recover another result
of Knebusch by topological methods; see \cite[Theorem 10.4]{Knebusch}:

\begin{theorem}\label{W(real curve)} 
Let $V$ be a smooth projective curve over $\R$
of genus $g$ with $\nu>0 $ connected real components. Then the Witt group is 
\begin{equation*}
W(V)\cong WR(V) \cong \Z^\nu\oplus(\Z/2)^{g}.
\end{equation*}
More precisely, the image of the signature $W(V)\map{\tau}\Z^\nu$ is the 
subgroup $\Gamma$ of rank $\nu$ consisting of sequences $(a_{1},...,a_{\nu})$
such that either all the $a_{i}$ are even or all the $a_{i}$ are odd, 
and we have a split exact sequence 
\begin{equation*}
0\to(\Z/2)^{g}\to WR(V)\overset{\sigma}{\to}\Gamma\to0.
\end{equation*}
\end{theorem}

\begin{proof}
To analyze $WR(X)$, we consider a small closed collar neighborhood $T$ of $%
X^{G}$ and the closure $K$ of the complement $X\setminus T$; they cover $X$
and intersect in a (trivial) $G$-cover $S=(X^G)\times G$ of $X^G$. ($S$ is $%
2\nu$ circles.) For $Z=K$, $T$ and $S$, we still have injections from $%
KR^{1}(Z)$ into $KO_{G}^{1}(Z)$. Therefore, we obtain a Mayer-Vietoris exact
sequence%
\begin{equation*}
\overset{\partial}{\to} WR(X)\to WR(K)\oplus WR(T)\to
WR(S)\to
\end{equation*}%
The group $G$ acts freely on the subspaces $K$ and $S$ of $X\setminus X^G$,
and almost the same calculation as in the proof of Theorem \ref{W(curve)}
shows that 
\begin{equation*}
\widetilde{KO}_G((K\!\times\!\R)^+)\cong WR(K)\cong (\Z/2)^{g}.
\end{equation*}%
On the other hand, $WR(T)$ is the direct sum of $\nu$ copies of $%
WR(S^1)\cong KO(S^1)\cong\Z\oplus\Z/2$ (see Example \ref%
{WR(X^G)}) and $WR(S)$ is the direct sum of $\nu$ copies of the cokernel $%
\Z/2\oplus\Z/2$ of the map from $KR(S^1\times G)\cong
KU(S^1)\cong\Z$ to $KO(S^1)\cong\Z\oplus(\Z/2)$.

This shows that the previous exact sequence terminates in%
\begin{equation*}
\overset{\partial}{\to} WR(X)\to (\Z/2)^{g}\oplus \Z%
^{\nu}\oplus (\Z/2)^{\nu}\to (\Z/2)^{\nu}\oplus(%
\Z/2)^{\nu} \to0.
\end{equation*}%
The last map in this sequence sends $(u,v,w)$ to $(\overline{v}%
+\theta(u),w+\theta (u)),$ where $\overline{v}$ is the class of $v$ modulo 2
and where $\theta$ is the composition of the sum operation $(\Z%
/2)^{g}\to \Z/2$ and the diagonal $\Z/2\to (%
\Z/2)^{\nu}.$ Therefore, the end of this sequence can be written as%
\begin{equation*}
\overset{\partial}{\to} WR(X)\to (\Z/2)^{g}\oplus \Gamma
\to 0,
\end{equation*}%
where $\Gamma$ is the subgroup of $\Z^{\nu}$ defined in 
the statement of Theorem \ref{W(real curve)}.

It remains to show that the map $\partial$ is zero. Using Theorem \ref%
{WR=KOG}, the left hand side of the Mayer-Vietoris exact sequence can be
extended to%
\[
\widetilde{KO}_G^{-1}(K\!\times\!\R)\oplus
\widetilde{KO}_G^{-1}((T\!\times\!\R)^+)\to
\widetilde{KO}_G^{-1}((S\!\times\!\R)^+)\overset{\partial}{\to} WR(X)\to.
\]
It suffices to show that the map $\widetilde{KO}_G^{-1}((T\!\times\!\R)^+)\to
\widetilde{KO}_G^{-1}((S\!\times\!\R)^+)$ is onto.

Writing the 1-point compactification of $T\!\times\!\R$ as the
union of $T\!\times\![-1,1]$ and the (contractible) closure of its complement,
we see that $\widetilde{KO}_G^{-1}((T\times\R)^+)$ is isomorphic to 
$\nu$ copies of $KO^{-1}(S^{1}).$ Using a similar cover of the 1-point 
compactification of $S\!\times\!\R$, we see that 
$\widetilde{KO}_G^{-1}((S\!\times\!\R)^+)$ is
isomorphic to $\nu$ copies of $KO^{-2}(S^1)$, and that the map between these
groups is isomorphic to $\nu$ copies of the cup product $KO^{-1}(S^1)\to
KO^{-2}(S^1)$ with the generator of $KO^{-1}(\pt)$. Since this cup
product is onto, so is the map $\widetilde{KO}_G^{-1}((T\times\R)^+)\to
\widetilde{KO}_G^{-1}((S\times\R)^+)$.
\end{proof}


\goodbreak


\begin{proposition}\label{skewWR(curve)}
For every curve $V$ over $\R$, we have  $_{-1}W(V)=0$
and ${}_{-1}WR(V)=0.$
\end{proposition}

\begin{proof}
We first show that $_{-1}WR(X)=0$ for $X=V_\C$.
By Theorem \ref{-GR=KRH},
$_{-1}GR_{0}(X)\cong KR_{\H}(X)$.
Since $KR_{\H}(Y)=\Z$ for any connected 
CW-complex $Y$ of dimension $\le3$
(by the Atiyah-Hirzebruch spectral sequence), the proofs of 
Theorems \ref{W(curve)} and \ref{W(real curve)} go through. 
We leave the details of the calculation to the reader.
(Use the quaternionic analogue of the spectral sequence (A.2) 
in \cite{KW}, converging to $KR_{\H}^*(X)$.)

We shall need the following classical facts:
if $R$ is a field or local ring then 
(i) $_{-1}\GW_0(V)=\Z$ and $_{-1}W_0(R)=0$, and
(ii) the infinite symplectic group $Sp(R)$ (for the trivial involution) 
is perfect; see \cite{Klingenberg} for instance.  This implies that 
$_{-1}\GW_1(R)=0$ and that $_{-1}W_1(R)=0$. 

For any (singular) curve $V$ over $\R$, the spectrum $\GW^{[2]}(V)$ satisfies
Zariski descent by \cite[9.7]{Schlichting.Fund}. The descent spectral
sequence amounts to an exact sequence for each $\GW_n^{[2]}(V)$, 
one of which is:
\[
0 \to H^1(V,GW_1^{[2]}) \to \GW_0^{[2]}(V) \to H^0(V,GW_0^{[2]})\to 0.
\]
The sheaf $GW_1^{[2]}\!=\!{}_{-1}GW_1$ vanishes, as its stalks are
${}_{-1}GW_1(\cO_{X,x})=0$. Since ${}_{-1}GW_0(\cO_{X,x})=\Z$, the sheaf 
$GW_0^{[2]}={}_{-1}GW_0$ is $\Z$,
we get $\GW_0^{[2]}(X)=H^0(X,\Z)$ and hence $_{-1}W_0(V)=0$. 
\end{proof}

Here is a different proof of Proposition \ref{skewWR(curve)}
for smooth curves. For smooth $V$, Balmer and Walter constructed
a spectral sequence in \cite{BalmerW} converging to the 
Balmer Witt groups $W^{[p+q]}(V)$, with $E_1^{p,q}=0$ 
unless $0\le p\le\dim(V)$ and $q\equiv0\pmod4$.
If $V$ is a curve, the
spectral sequence collapses to yield
${}_{-1}W_0(V)=W^{[2]}(V)=0$; see \cite[10.1b]{BalmerW}.

\begin{proof}[Alternative proof for singular affine curves]
Suppose that $V\!=\Spec(A)$.
It is well known that $K_0(A)=K_0(A_\red)$ and
$_{-1}GW_0(A)={}_{-1}GW_0(A_\red)$, so ${}_{-1}W(A)={}_{-1}W(A_\red)$. 
Therefore we may assume that $A$ is reduced. 

If $B$ is the normalization of
$A$ and $I$ is the conductor ideal, we have the following diagram:
\[\xymatrix{
K_1(B/I) \ar[d]\ar[r] & {}_{-1}GW_1(B/I) \ar[d]\ar[r] & _{-1}W_1(B/I)\ar[d]
\ar[r]&0.\\
K_0(A) \ar[d]\ar[r] & {}_{-1}GW_0(A) \ar[d]\ar[r] & {}_{-1}W_0(A)\ar[d]
\ar[r]&0.\\
\atop{K_0(B)\oplus}{K_0(A/I)}\ar[d]\ar[r]&
\atop{{}_{-1}GW_0(B)\oplus}{{}_{-1}GW_0(A/I)} \ar[d]\ar[r] & 
\atop{{}_{-1}W_0(B)\oplus}{{}_{-1}W_0(A/I)} \ar[d]\ar[r]&0.
\\
K_0(B/I) \ar[r]^{\cong~} &{}_{-1}GW_0(B/I) \ar[r]& {}_{-1}W_0(B/I)\ar[r]&0.
}\]
The left two columns are exact
(the $GW$ column is exact by \cite[III(2.3)]{Bass343})
and the horizontal (exact) sequences define ${}_{-1}W_n$.

As $\Spec(B)$ is a smooth curve, we have already seen that 
${}_{-1}W_0(B)=0$. 
As $A/I$ and $B/I$ are products of artin local rings,
we have  $_{-1}W_0(A/I)= {}_{-1}W_1(B/I)=0$.
A diagram chase now shows that $_{-1}W_0(A)=0$.
\end{proof}

\medskip
\section{Williams' conjecture in Real Hermitian $K$-theory}
\label{sec:williams}

The purpose of this section, achieved in Theorem \ref{GW=GR}, 
is to compare 
$GW_*(V;\Z/2^\nu)$ and $GR_{\ast }(V;\Z/2^\nu)$. For this, we establish 
an analogue for Real spaces (Theorem \ref{Williams-GR}) of 
a general conjecture relating $K$-theory to Hermitian $K$-theory,
formulated by Bruce Williams in \cite[p.\,667]{Williams}.

\goodbreak
Williams' conjecture has been verified
in many cases of interest; see \cite[1.1, 1.5, 2.6]{BKOS} and the
references cited there. To state the result of \cite[2.6]{BKOS}
in our context,
we write $\bbK(V)$ for the nonconnective $K$-theory spectrum of a 
variety $V$ and $\GW(V)$ for Karoubi's Grothendieck-Witt 
spectrum of $V$ \cite{Schlichting.Fund},
using the usual duality $E^*=\Hom(E,\cO_{V})$. 
By construction, $K_n(V)=\pi_n\bbK(V)$ and ${\kGW_n}(V)=\pi_n\GW(V)$ for $n\ge0$.
Since $E^{**}\!\cong\!E$, the duality induces an involution on $\bbK(V)$.

\begin{theorem}\label{Williams} 
Let $V$ be a variety (over $\R$ or $\C$ for simplicity). 
Then the map of spectra $\GW(V)\to\bbK(V)^{hG}$ 
is a 2-adic homotopy equivalence. Here $G=C_{2}$ acts by duality.
\end{theorem}

Although ``$X\to Y$ is a 2-adic homotopy equivalence'' 
means that $X\,\widehat{\;}\to Y\,\widehat{\;}$ is a homotopy equivalence,
the property we use is that the induced map
$\varprojlim_r \pi_*(X;\Z/2^r)\map{}\varprojlim_r \pi_*(Y;\Z/2^r)$
is an isomorphism.

\medskip
Let $X$ be a Real space, and let $\mathbb{GR}(X)$ denote the spectrum
associated to the Real Grothendieck-Witt theory of $X$. 
%
As in Section \ref{sec:KR}, let $\mathbb{KR}(X)$ denote the spectrum
associated to the $KR$-theory of $X$.
By construction, $KR_n(X)=\pi_n\mathbb{KR}(X)$ and 
$GR_n(X)=\pi_n\mathbb{GR}(X)$.  

From the Banach algebra point of view, $\mathbb{KR}(X)$ is the usual
topological $K$-theory spectrum associated to the algebra $A$ of 
continuous functions $f:X\to \C$ such that 
$f(\overline{x})=\overline{f(x)},$ while $\mathbb{GR}(X)$ is the 
topological Hermitian $K$-theory spectrum of $A$ 
(see Appendices \ref{app:Banach}, \ref{sec:Marco}). 

The following theorem proves the analogue of Williams' conjecture 
for Real Hermitian K-theory in a topological setting.
Another proof using Banach algebras is given in Appendix \ref{sec:Marco}.

The group $G=C_{2}$ acts on $\mathbb{KR}(X)$ and on $\bKU(X)$, 
by sending a Real bundle (resp., a complex bundle) to its dual bundle.
The canonical map $KO(X)\to \bKU(X)^{hG}$ is a homotopy equivalence;
see \cite{MKBanach} and \cite[p.\,808]{BK}. 

\begin{theorem}\label{Williams-GR} 
Let $X$ be a compact Real space, and let $G=C_{2}$ act by duality 
on the spectrum $\mathbb{KR}(X)$. 
Then we have a 2-adic homotopy equivalence: 
\begin{equation*}
\mathbb{GR}(X)\simeq\mathbb{KR}(X)^{hG}.
\end{equation*}
\end{theorem}

\begin{proof}
In addition to duality, 
there are three non-trivial actions of $G$ on $\bKU(X)$, 
whose arise from the action of $G$ (a) on the space $X$, 
(b) on the complex vector bundles over $X$ 
(ignoring the action on $X$), sending a bundle $E$ to
its complex conjugate bundle $E'$ and (c) the product of
these actions, sending a bundle $E$ to the bundle $F$ defined by 
$F_x=E'_{\sigma x}$. 
We shall write the corresponding homotopy fixed 
point spectra as $\bKU(X)^{hG_X}$, $\bKU(X)^{hG'}$ and
$\bKU(X)^{hG'_X}$, respectively. 

Using the Banach algebra approach, 
we have homotopy equivalences: 
\[
\mathbb{KO}(X)\simeq \bKU(X)^{hG}, \qquad
\mathbb{KR}(X)\simeq \mathbb{KU}(X)^{hG'}. 
\]
(The first equivalence is verified in \cite{MKBanach} but is older;
the argument for $hG'$ is given in \cite[pp.\,809-810]{BK}.)
Thus
\[
\mathbb{KO}(X)^{hG'} \simeq (\mathbb{KU}(X)^{hG})^{hG'} \simeq 
(\mathbb{KU}(X)^{hG'})^{hG}\simeq \mathbb{KR}(X)^{hG}.
\]
Since the action (b) is trivial on $\mathbb{KO}(X)$, 
$\mathbb{KO}(X)^{hG_X} \simeq \mathbb{KO}(X)^{hG'}$
after 2-adic completion.
Thus Theorem \ref{Williams-GR} is equivalent to the assertion that 
\[
\mathbb{GR}(X)\map{\simeq}\mathbb{KO}(X)^{hG_X}
\]
is a homotopy equivalence after 2-adic completion.

By Theorem \ref{GR=KOG}, this is equivalent to the
Atiyah-Segal completion theorem with coefficients $\Z/2^\nu$:
\begin{equation*}
KO_{G}(X)\,\widehat{\;}\map{\simeq}KO(EG\times _{G}X)\,\widehat{\;}.
\end{equation*}%
More precisely, on the 0-spectrum level --- 
which is sufficient for our purpose --- we have (after completion) 
\begin{equation*}
KO(EG\times _{G}X)\map{\simeq}\pi _{0}\Hom_{G}(EG,\mathbb{KO}(X)).
\end{equation*}%
On the other hand, the 2-adic completion of $KO_{G}(X)$ coincides with 
the completion of $KO_{G}(X)$ via the fundamental ideal $I $ of $RO(G)$,
followed by 2-adic completion.
Indeed, the filtration of $RO(G)\cong \Z\oplus \Z$ 
by the powers of $I$ is given by the sequence of ideals 
$\{\Z\oplus 2^{n}\Z\}$. 
Therefore, $KO(EG\times_{G}X)\,\widehat{\;}$ 
\ is isomorphic to the 2-adic completion of $KO_{G}(X).$
\end{proof}


\begin{remm}
If $G$ acts freely on $X,$ the homotopy equivalence%
\begin{equation*}
\mathbb{GR}(X)\simeq\mathbb{KO}(X/G)\simeq\mathbb{KO}(X)^{hG_X}
\end{equation*}
is obvious and we don't need $2$-adic completions to prove it. When the
action of $G$ is trivial, another approach would be to repeat the argument
in \cite[p.\,809]{BK}; we would then need $2$-adic completion for the
statement. Another proof of Theorem \ref{Williams-GR} is to use a 
Mayer-Vietoris argument when the space of fixed points $X^{G}$ has 
an equivariant tubular neighborhood, which often happens in applications.
\end{remm}

For the next result, let $GW^c_*(V)$ denote the homotopy groups of the
homotopy fiber of $\GW(V)\to\mathbb{GR}(V)$. As in Remark \ref{rem:GW^c},
the comparison groups $GW_{n}^{c}(V)$ fit into an exact sequence:
\begin{equation}  \label{eq:GW^c}
{\kGW_{n+1}}(V)\to GR_{n+1}(V)\to GW_{n}^{c}(V)\to
{\kGW_{n}}(V)\to GR_{n}(V).
\end{equation}

\begin{theorem}\label{GW=GR} 
Let $V$ be an algebraic variety over $\R$ of dimension $d$ (with 
or without singularities). Then (for all $\nu>0$) the canonical map 
\[
\GW(V)\to\mathbb{GR}(V)
\]
induces isomorphisms ${\kGW_n}(V;\Z/2^\nu)\to KO^{-n}_G(X;\Z/2^\nu)$
for $n\geq d-1$, and monomorphisms ${\kGW_{d-2}}(V;\Z/2^\nu)\to
KO^{-d+2}_G(X;\Z/2^\nu)$.

In other words, the groups $GW_{n}^{c}(V)$ are uniquely 2-divisible for 
all $n\ge d-2$, and $GW_{d-3}^{c}(V)$ is 2-torsionfree.
\end{theorem}

\begin{proof}
By Theorems \ref{GR=KOG}, \ref{Williams} and \ref{Williams-GR}, 
$\GW^c(V)\to\bbK^c(V)^{hG}$ is a 2-adic homotopy equivalence. 
The result now follows from Corollary \ref{Kc-u.d.}. 
\end{proof}

We can also compare the homotopy fiber $U(V)={\kGW^{[-1]}}(V)$ of the hyperbolic
map $\bbK(V)\to\GW(V)$ with the homotopy fiber $%
UR(V)=GR^{[-1]}(V)$ of $\mathbb{KR}(V)\to\mathbb{GR}(V)$ The following
corollary is immediate from Corollary \ref{Kc-u.d.} 
and Theorem \ref{GW=GR}.  %

\begin{corollary}\label{U=UR}
The maps $\pi _{n}(\mathbb{U}(V);\Z/2^{\upsilon })\to 
\pi_{n}(\mathbb{UR}(V);\Z/2^{\upsilon})$ are isomorphisms for $n\geq
d-1$ and a monomorphism for $n=d-2.$ Hence, the homotopy groups $%
U_{n}^{c}(V) $ of the homotopy fiber $\mathbb{U}^{c}(V)$ of $\mathbb{U}%
(V)\to \mathbb{UR}(V)$ are uniquely 2-divisible for $n\geq d-2$
and $U_{d-3}^{c}(V)$ is $2$-torsion free.
\end{corollary}

\goodbreak

\begin{remark}
We may extend the previous results to the symplectic setting. 
Williams' conjecture says that the map 
$_{-1}\GW(V)\to \mathbb{K}(X)^{h_{-}G}$ is a 
2-adic homotopy equivalence. It is proven in \cite[1.1]{BKOS}; 
see Appendix \ref{sec:Marco}, Theorem \ref{Williams-Banach}. 
Williams' conjecture for Real
Hermitian $K$-theory takes the form%
\[  
_{-1}\mathbb{GR}(X)\cong \mathbb{KR}(X)^{h_{-}G}
\]  
where $G=C_{2}$ and $h_{-}G$ denotes the action of $G$ on the 
spectrum $\mathbb{KR}(X)$ described in \cite[p.\,808]{BK}. 
Another approach to this result and Theorem \ref{GW=GR}, using 
different methods, may be found in
Appendix \ref{sec:Marco}. 

A more general approach, based on the same idea as 
Theorem \ref{GW=GR}, is related to twisted $K$-theory and
will be given in our paper \cite{KW:KRA}. Recall that
$GR(X)\cong GW_0(A)$, where $A$ is the Banach algebra of 
continuous functions $f:X\to\C$ satisfying 
$f(\sigma(x))=\overline{f(x)}$
(see Appendix \ref{app:Banach}). The group ${}_{-1}GR(X)$
is isomorphic to $GW_0(M_2A)$, where $M_2A$ is the algebra of 
$2\times2$ matrices over $A$ and the involution is
$\left({\atop{a}c}{\atop{b}d}\right)\mapsto
\left(\atop{\bar{d}}{-\bar{c}}{\atop{-\bar{b}}{\bar{a}}}\right)$.


If $G$ acts trivially on $X$, so that $KR(X)\cong KO(X)$
(Example \ref{WR(X^G)}), we also have ${}_{-1}GR(X)\cong KU(X)$.
In this case, these general constructions reduce
to a 2-adic homotopy equivalence
\[
\mathbb{KU}(X) \cong \mathbb{KO}(X)^{h_-G}.
\]
This reflects the fact that the symplectic group $Sp_{2n}(\R)$
is homotopy equivalent to its maximal compact subgroup 
(the unitary group $U_n$).
\end{remark}

Copying the proof of Theorem \ref{GW=GR}, with Theorem \ref{GR=KOG} replaced
by Theorem \ref{-GR=KRH}, the above remarks establish the following result.

\begin{theorem}\label{GW=GR{-1}}
Let $V$ be an algebraic variety over $\R$ of dimension $d$, 
with or without singularities. Then the maps 
\[
_{-1}\GW(V)\map{} {}_{-1}\mathbb{GR}(V) 
\]
induce isomorphisms $\!{}_{-1}{\kGW_n}(V;\Z/2^\nu)\to {}_{-1}GR_n(V;\Z/2^\nu)$
for $n\!\ge\!d\!-\!1$, and monomorphisms 
${}_{-1}{\kGW_{d-2}}(V;\Z/2^\nu)\to{}_{-1}GR_{d-2}(X;\Z/2^\nu)$.
\end{theorem}

\begin{rem}
Defining the comparison groups ${}_{-1}GW_{n}^{c}(V)$ as in
Remark \ref{rem:GW^c}, Theorem \ref{GW=GR{-1}} shows that
${}_{-1}GW_{n}^{c}(V)$ is uniquely 2-divisible for all $n\ge d-2$, 
and the group ${}_{-1}GW_{d-3}^{c}(V)$ is 2-torsionfree.
\end{rem}

%

\section{Brumfiel's theorem}\label{sec:Brumfiel}

The purpose of this section is to state and prove Theorem \ref{Brumfiel}
below, which is a generalization of Brumfiel's theorem 
\cite{Brumfiel84, Brumfiel87} (over $\R$).

We first state the version of Brumfiel's theorem we shall use
in the next section, comparing the Witt groups $W(V)$ and $WR(V_\R)$
of a scheme $V$ of finite type over $\R$.
Recall that $V_{\R}$ is the Hausdorff topological space of 
$\R$-points of $V$.

\begin{theorem}[Brumfiel's Theorem]\label{thm:Brumfiel}
For any quasi-projective variety $V$ over $\R$,
the signature map $W(V)\smap{\gamma} \!KO(V_{\R})$ is an
isomorphism modulo 2-primary torsion. The same is true for the 
higher signature map $W_{n}(V)\to KO^{-n}(V_{\R})$ for all integers $n$.
\end{theorem}

If $V=\Spec(A)$, the higher Witt groups $W_n(A)$ in this theorem 
are defined as the cokernel of the hyperbolic map from $K_n(A)$ to 
$GW_n(A)=\pi_nBO(A)^+$ if $n>0$, and as $W_0(S^{|n|}A)$ if $n<0$; 
see Karoubi \cite{MKfiltration}. The definition for quasi-projective
$V$ is given in \cite{Schlichting.Herm}.

Brumfiel proved Theorem \ref{thm:Brumfiel} for $V=\Spec(A)$ and $n=0$
in \cite{Brumfiel84}, and for all $n$ in \cite{Brumfiel87}.
Another proof of surjectivity was later given in \cite[Thm.\,15.3.1]{BCR}. 
As noted in the Math Review of \cite{Brumfiel84}, the proof in loc.\,cit.\ 
is somewhat sketchy.
This is another reason for our presentation of a new proof of this result.

Now let $V$ be a scheme over $\Z[1/2]$.
We write $L(V)$ for the spectrum representing Balmer's Witt groups:
$\pi_nL(V)=W_B^{-n}(V)$ \cite[7.2]{Schlichting.Fund}. 
We recall the definition of this spectrum in \ref{def:L(V)} below.
It is proven in \cite[Lemma A.3]{Hornbostel-Schlichting} that
$W_{n}(A)[1/2] \cong W^{-n}_B(A)[1/2]$ for all integers $n$.
The same proof applies to show that
$W_{n}(V)[1/2] \cong W^{-n}_B(V)[1/2]$
when $V$ is not affine.
Therefore, up to 2-primary torsion, 
we can replace the groups $W_n(V)$ 
by the homotopy groups $\pi_nL(V)$. Theorem \ref{thm:Brumfiel}
is obtained by taking the homotopy groups $\pi_n$ of the equivalence of
spectra stated in the following theorem.


\begin{theorem}\label{Brumfiel}
Let $V$ be a scheme of finite type over $\R$ and assume that 
$V$ has an ample family of line bundles.  Then there
is an equivalence of spectra, natural in $V$:
\[
L(V)[1/2] \map{\sim} KO(V_{\R})[1/2].
\]
\end{theorem}

Recall from Section \ref{sec:GR-WR} that there is a morphism 
$\uGW(V)\to\GW(V)$ of spectra which is an equivalence on connected covers.
We also need the element $\eta$ of ${\mGW_{-1}^{[-1]}}(\R) \cong W_0(\R)$ 
corresponding to the element $1\in W_0(\R)$; see \cite[\S6]{Schlichting.Fund}.
The following definition is taken from
Definition 7.1 and Proposition 7.2 of \cite{Schlichting.Fund}.

\begin{definition}\label{def:L(V)}
The spectrum $L(V)$ is $\uGW(V)[\eta^{-1}]$, 
i.e., the colimit of the sequence
\[
\uGW^{[0]}(V) \map{\eta\cup} S^1\wedge \uGW^{[-1]}(V) \map{\eta\cup} 
S^2 \wedge \uGW^{[-2]}(V) \map{\eta\cup}\cdots
\]
Similarly, the spectrum $\L(V)$ is $\GW(V)[\eta^{-1}]$, 
i.e., the colimit of the sequence
\[
\GW^{[0]}(V) \map{\eta\cup} S^1\wedge \GW^{[-1]}(V) \map{\eta\cup} 
S^2 \wedge \GW^{[-2]}(V) \map{\eta\cup}\cdots.
\]
The canonical map $\uGW(V)\to\GW(V)$ induces a map $L(V)\to\L(V)$,
and $L[1/2]\simeq\L[1/2]$ by \cite[Lemma 8.16]{Schlichting.Fund}.
Since  $\GW(V)$ satisfies Nisnevich descent by  
\cite[Theorem 9.6]{Schlichting.Fund}, it follows that
$L[1/2]\simeq\L[1/2]$ also satisfies Nisnevich descent.
\end{definition}


\medskip
We will need the {\it real \'etale topology}, defined in 
\cite[1.2.1]{Scheiderer} and \cite{CosteRoy}, whose sheafification 
functor will be written as $a_{\ret}$.  Covering families are families of
\'etale maps which induce a surjective family of associated real spectra.
Stalks are henselian local rings with real closed residue fields.  
The category of sheaves on $X_{ret}$ is equivalent to the category of sheaves on the topological space $X_r$ where for affine $X=\Spec A$, we have $X_r=\Sper A$, the real spectrum of $A$; see \cite[Theorem 1.3]{Scheiderer}.

We will
also need the Nisnevich topology with sheafification functor $a_{\nis}$.
For a topological space $X$ and a discrete set $S$, we will write
$C(X,S)$ for the set of continuous functions $X \to S$.

\begin{lemma}
\label{lem:WNis=Wret}
{(1)}
For any henselian local ring $R$ with $\frac{1}{2}\in R$, the signature map induces an isomorphism 
$$W(R)[1/2] \cong C(\Sper R, \Z[1/2]).$$
\item{(2)}
For any scheme $V$ over $\Z[1/2]$, the Nisnevich sheaf 
$a_{\nis}W[1/2]$ on $Sch_{|V}$ is a sheaf for the real \'etale topology.
In particular, the following sheafification map on $Sch_{|V}$ 
is an isomorphism of presheaves:
\[
a_{\nis}W[1/2]\ \map{\cong}\ a_{\ret}W[1/2].
\]
\end{lemma}

\begin{proof}
For the first part,
let $k$ be the residue field of the henselian local ring $R$.
The quotient map $R\to k$ induces a commutative diagram
\[
\xymatrix{
  W(R)[1/2] \ar[r] \ar[d] & W(k)[1/2] \ar[d]\\
  C(\Sper R, \Z[1/2]) \ar[r] & C(\Sper k, \Z[1/2]) }
\] 
in which the top horizontal map is an isomorphism, by the rigidity
property of Witt groups \cite[Satz 3.3]{KnebuschRigidity}.  The right
vertical map is an isomorphism, by the classical computation (up to
$2$-primary torsion) of the Witt ring of a field; see for instance
\cite[\S6]{Scharlau}.  For the lower horizontal map, let $A$ be a
commutative ring and $S$ a discrete set.  If $\Max_r(A) \subset \Sper
A$ denotes the subspace of closed points, then the restriction map
induces an isomorphism $C(\Sper A,S) \cong C(\Max_r(A),S)$.  This
follows for instance from \cite[Proposition and Definition II
2.2]{AndradasBroeckerRuiz}.  If $A$ is local with residue field $k$
then $\Sper (k) \subset \Max_r(A)$, and if, moreover, $A$ is henselian
local, then $\Sper (k) = \Max_r(A)$; see for instance
\cite[Proposition II.2.4]{AndradasBroeckerRuiz}.  This shows that the
lower horizontal map in the diagram is an isomorphism, and hence, so
is the left vertical map.

For the second part of the lemma, we note that the first part implies 
that the signature map induces an isomorphism of presheaves between
$a_{\nis}W[1/2]$ and the presheaf sending $A$ to $C(\Sper A,\Z[1/2])$;
since both are sheaves for the Nisnevich topology.
The latter is a sheaf in the real \'etale topology,
so $a_{\nis}W[1/2]$ is too.
\end{proof}


There is a ``local'' model structure on presheaves of spectra on $V_{\ret}$,
in which a map $F\to F'$ is a (local) weak equivalence if
the sheafification of $\pi_*F\to\pi_*F'$ is an isomorphism.
Let $F\to F^{\ret}$ denote the fibrant replacement of $F$ 
in the real \'etale topology.  We say that $F$ satisfies
{\it descent for the real \'etale topology} if $F(U)\to F^\ret(U)$
is a weak equivalence (of spectra) for each $U$ in $V_\ret$.

\begin{example}\label{KO-descent}
$KO$ satisfies descent for the real \'etale topology.
This follows from \cite[4.10]{duggerIsaksen:A1Ralization}, since 
$U_\R$ is a covering space of $V_\R$ 
for each real \'etale cover $U \to V$.
(See also \cite[5.2]{duggerIsaksen:A1Ralization}
and its obvious real \'etale analog.)
\end{example}

\begin{theorem}\label{thm:L=Lret}
Let $V$ be a finite dimensional noetherian $\Z[1/2]$-scheme 
with an ample family of line bundles.
Then $L(V)[1/2]$ satisfies descent for the real \'etale topology.
\end{theorem}

\begin{proof}
The map from $L_*(V)[1/2]=L_*^{\nis}(V)[1/2]$ to $L_*^{\ret}(V)[1/2]$ 
is the map on abutments of a map of strongly convergent descent 
spectral sequences which on the $E_2^{p,q}$-page is the map
\begin{equation*}\label{eqn:NisToRet}
H^p_{\nis}(V,a_{\nis}L_{-q}[1/2]) \to H^p_{\ret}(V,a_{\ret}L_{-q}[1/2]).
\end{equation*}
When $q\ne0\pmod4$, both coefficient sheaves are zero 
by \cite[Thm.\,5.6]{Balmer:TWGII};
hence the spectral sequences have $E_2^{p,q}=0$ for $q\ne0\pmod4$.
Since $L_0(V)=W_0(V)$, the maps $a_{\nis}L_{-q}[1/2]=a_{\ret}L_{-q}[1/2]$
are isomorphisms for $q=0\pmod 4$ by Lemma \ref{lem:WNis=Wret}
($a_\ret L_0=a_\ret W_0$). 
Now, for any sheaf of abelian groups $A$ on $V_{\ret}$, 
such as $a_{\nis} W_0$, the natural map 
$H^p_{\nis}(V,A) \to H^p_{\ret}(V,A)$ is an isomorphism by
\cite[Proposition 19.2.1]{Scheiderer}. Thus
the morphism of spectral sequences is an isomorphism on $E_2$-terms.
Hence, it is an isomorphism on abutments.
\end{proof}

From now on assume that $V$ is a scheme of finite type over the real
numbers $\R$.  We have a canonical map of ring spectra
\begin{equation}\label{eqn:GW-KO}
\GW(V) \to \GW^{\topl}(V_{\R}) = KO(V_{\R})[\eps]/(\eps^2\!-\!1) 
\quad\map{\eps\mapsto-1}\quad KO(V_{\R}),
\end{equation}
where $\eps$ corresponds to $\langle -1\rangle \in GW_0^{\topl}(V_{\R})$.

Most of the rest of this section is preparation for the proof of the
following proposition which will be given at the end of this Section. 

\begin{proposition}\label{L->KO}
The map \eqref{eqn:GW-KO} induces a natural map of spectra,
$L(V)[\frac12]\!\to\!KO(V_{\R})[\frac12]$, 
which is a weak equivalence for $V=\Spec \R$.
\end{proposition}

Using this proposition, we can now prove
Theorem \ref{Brumfiel}. 

\begin{proof}[Proof of Theorem \ref{Brumfiel}]
Consider the commutative diagram of presheaves of spectra on $V_\ret$:
\[\xymatrix{
L[1/2]\ar[d] \ar[r]^{\simeq} & L[1/2]^{\ret} \ar[d]\\
KO[1/2] \ar[r]^{\simeq} & KO[1/2]^{\ret}\!.
}\]
The top horizontal map is a weak equivalence by Theorem \ref{thm:L=Lret},
and the lower horizontal map is a weak equivalence
by Example \ref{KO-descent}.
The right vertical map induces a map of descent spectral sequences for 
the real \'etale topology. We claim that this map is an isomorphism
on the $E_2$-page for all $V$. This will imply that the right vertical map 
(and hence the left map) is also an equivalence, proving the theorem.

To verify the claim, consider the inclusion of topological spaces
$V_{\R}\subset V_r$.  Restriction induces an equivalence 
$\Sh(V_r) \to \Sh(V_{\R})$ of categories of sheaves; see for instance
\cite[Theorem 7.2.3]{BCR} where $V_r$ is written as $\tilde{V}$.
In particular, we have isomorphisms $H^p_\ret(V,\pi_qL[1/2])\cong
H^p(V_r,\pi_qL[1/2]) \cong H^p(V_{\R},\pi_qL[1/2])$ and
\[
H^p_\ret(V,\pi_qKO[1/2])\cong H^p(V_r,\pi_qKO[1/2]) \cong
H^p(V_{\R},\pi_qKO[1/2]).
\]
Now the restrictions of $\pi_qL[1/2]$ and
$\pi_qKO[1/2]$ to $V_\ret$ are the constant sheaves $\pi_qL(\R)[1/2]$
and $\pi_qKO(\R)[1/2]$, respectively.  Since the map $L(\R)[1/2]\to
KO(\R)[1/2]$ in Proposition \ref{L->KO} is an equivalence, the claim
follows.
\end{proof}


In order to compare $L(V)$ with $KO(V_{\R})$, it will be convenient to
give a different but equivalent description of $\L(V)$.

For each pointed $k$-scheme $V$, let $\widetilde{GW}_*^{[n]}$ denote the
reduced theory associated with $GW^{[n]}$, that is,
$\widetilde{GW}_*^{[n]}(V)$ is the kernel of the split surjective map
$GW_*^{[n]}(V) \to GW_*^{[n]}(k)$ given by the inclusion of the base point
into $V$.  

By homotopy invariance \cite[9.8]{Schlichting.Fund}, the boundary map
in the Bass Fundamental Theorem \cite[Theorem 9.13]{Schlichting.Fund}
for $\Gm=\Spec k[T,T^{-1}]$ induces isomorphisms
\[
\delta: \widetilde{GW}^{[1]}_1(\Gm) \map{\cong} GW_0(k), \quad
\widetilde{GW}_0(\Gm) \map{\cong} GW_{-1}^{[-1]}\cong W_0(k)
\]
fitting into the diagram
\[\xymatrix{
\widetilde{GW}^{[1]}_1(\Gm) \ar[r]^{\eta\cup} \ar[d]_{\delta}^{\cong} 
& \widetilde{GW}_0(\Gm) \ar[d]_{\delta}^{\cong} \\
GW_0(k) \ar[r]^{\hspace{-4ex}\eta\cup} & GW^{[-1]}_{-1}(k) \rlap{=W(k)}.
}\]
As the horizontal maps are the boundary map of the Bott sequence 
\cite[8.11]{Schlichting.Fund}, the diagram commutes up to 
multiplication with $-1$ by the usual Verdier exercise
\cite[10.2.6]{WH}.
In particular, there is a unique element 
$$
[T]\in \widetilde{GW}^{[1]}_1(\Gm)
$$
such that $\delta [T] = 1 \in GW_0(k)$.

\begin{lemma}\label{lem:EtaCupT}
We have $\eta \cup [T] = 1-\langle T \rangle \in GW_0(\Gm)$.
\end{lemma}

\begin{proof}
As the diagram anti-commutes, $\delta(\eta\cup[T])=-\eta$.
On the other hand, the composition of the right vertical map 
with $GW_{-1}^{[-1]}\cong W_0(k)$ is the usual boundary map 
$GW_0(k[T,T^{-1}]) \to W_0(k)$ as in \cite{MilnorHusemoller},
which sends $\langle T \rangle $ to $1\in W_0(k)$ and 
sends $1\in GW_0(k[T,T^{-1}])$ to 0. 
Hence $\delta(\langle T\rangle-\langle1\rangle)=\eta$; 
the lemma follows, since the right vertical map is an 
isomorphism of abelian groups.
\end{proof}

Write $\GWH^{[n]}$ for the homotopy invariant version of
Grothendieck-Witt theory, that is, $\GWH^{[n]}(V)$ is the realization
of the simplicial spectrum $q\mapsto \GW^{[n]}(V\times \Delta^q)$
where $\Delta^q = \Spec k[t_0,...,t_q]/(\sum t_i = 1)$.  
The map to the final object $\Delta^q \to \Spec k$ induces
a natural map $\GW^{[n]} \to \GWH^{[n]}$.

\begin{lemma}
\label{lem:CupT}
For any pointed $k$-scheme $V$, the cup product with 
$[T] \in \widetilde{GW}^{[1]}_1(\Gm)$ induces an equivalence  
$$[T]\cup: S^1\wedge \widetilde{\GWH}^{[n]}(V) \simeq \widetilde{\GWH}^{[n+1]}(\Gm\wedge V).$$
\end{lemma}

\begin{proof}
This follows from the definition of $[T]$ and the commutative diagram
$$\xymatrix{
\widetilde{GW}^{[1]}_1(\Gm)\otimes \widetilde{\GWH}^{[n]}_p(V) \ar[r]^{\cup}
\ar[d]_{\delta\otimes 1}^{\cong} & 
\widetilde{\GWH}^{[n+1]}_{p+1}(\Gm^{\wedge1}\wedge V) \ar[d]^{\delta}_{\cong}\\
GW_0(k)\otimes \widetilde{\GWH}^{[n]}_p(V) \ar[r]^{\cup}_{\cong}
 & 
\widetilde{\GWH}^{[n]}_{p}(V)
}$$
where the tensor product is over $GW_0(k)$.
In particular, the lower horizontal map is an isomorphism.
The vertical maps are the boundary maps in the Bass Fundamental Theorem
which are isomorphisms, by homotopy invariance.
\end{proof}

\begin{corollary}
\label{cor:LasGWwithEtaTilde}
Set $\tilde{\eta}=1-\langle T\rangle \in \widetilde{GW}_0(\Gm)$.
Then for any $k$-scheme $V$, the ``stabilized'' Witt-theory spectrum 
$\L(V)$ of \cite[8.12]{Schlichting.Fund} is the colimit of the sequence
\[
\GWH(V) \map{\tilde{\eta}\cup} \widetilde{\GWH}(\Gm^{\wedge 1}V_+)
        \map{\tilde{\eta}\cup} \widetilde{\GWH}(\Gm^{\wedge 2}V_+) 
        \map{\tilde{\eta}\cup} \cdots
\]
\end{corollary}

\begin{proof}
 The sequence in the corollary is equivalent to the homotopy invariant
version of the sequence in Definition \ref{def:L(V)}
in view of Lemmas \ref{lem:EtaCupT} and \ref{lem:CupT}.  Since stabilized Witt
  groups are homotopy invariant, this colimit computes $\L(V)$.
\end{proof}

\begin{lemma}\label{lem:KOmultby2}
Under the map $\GW(V) \to KO(V_{\R})$ of \eqref{eqn:GW-KO}, 
the cup-product with the element $\tilde{\eta}=1-\langle T \rangle$
of $\widetilde{GW}_0(\Gm)$ induces multiplication by $2$ on $KO(V_{\R})$.
\end{lemma}

\begin{proof}
Recall that $\widetilde{KO}_0(S^0) = KO_0(\pt)=\Z$.
We have to show that the map
\[
\widetilde{GW}_0(\Gm) \to 
\widetilde{GW}^{\topl}_0(\R^\times) = \widetilde{GW}^{\topl}_0(S^0) 
\quad\map{\eps \mapsto -1}\quad \widetilde{KO}_0(S^0)=\Z
\]
sends $\tilde{\eta}=1-\langle T\rangle$ to $2$.
The first map between reduced Grothendieck-Witt groups is
induced by the inclusion of the Laurent polynomial ring
$\R[T,T^{-1}]$ into the ring $C(\R^\times)$ of continuous functions 
$\R^\times \to \R$ sending $T$ to the standard inclusion.  Therefore, the map 
$GW_0(\Gm) \to GW^{\topl}_0(S^0)$ sends $1-\langle T\rangle$ to $0$ 
on the base point component $+1$ of $\{\pm 1\} = S^0 \subset \R^\times$ 
and to $1-\langle-1\rangle$ on the non-base point component $-1$ of $S^0$.
This element corresponds to $1-\eps\in\widetilde{KO}(S^0)[\eps]/(\eps^2-1)$
and is thus sent to $2$ in $KO(\pt)=\Z$.
\end{proof}

\begin{proof}[Proof of Proposition \ref{L->KO}]
Inverting the multiplication by $\tilde{\eta}=1-\langle T\rangle$ on
both sides of the (homotopy invariant version of the) map
(\ref{eqn:GW-KO}) yields a map of spectra
\[
\L(V) = \GWH(V)[\tilde{\eta}^{-1}] \to KO(V_\R)[\tilde{\eta}^{-1}]
 = KO(V_{\R})[1/2]
\]
in view of Corollary \ref{cor:LasGWwithEtaTilde} and Lemma \ref{lem:KOmultby2}.
Hence, we obtain a natural transformation 
\begin{equation}\label{eqn:L-KO}
L(V)[1/2]=\L(V)[1/2] \to KO(V_{\R})[1/2]
\end{equation}
which is multiplicative on homotopy groups.
In particular, this map is periodic of period $4$.
When $V=\Spec\,\R$, it is an isomorphism since 
it is an isomorphism in degree $0$, 
and zero in degrees $\not\equiv\!0$ mod $4$.
\end{proof}

\goodbreak 
\section{Exponents for Real Witt groups}\label{sec:WR}

In this section we show that, for any finite $G$-CW complex $X$, 
the kernel and cokernel of the restriction map 
$WR(X)\to WR(X^G)=KO(X^G)$ are 2-primary groups of bounded exponent,
depending on $\dim(X)$. In particular, this is true for the 
kernel and cokernel of $WR(V)\to WR(V_{\R})=KO(V_{\R})$
for every variety $V$ over $\R$. 
The following number is useful; it was introduced in Example \ref{WRfree}.

\begin{defn}\label{def:f}
For $d\ge1$, let $f(d)$ denote the number of integers 
$i$ in the range $1\le i\le d$ with $i=0,1,2$ or $4$ mod $8$. 
\end{defn}

%

\medskip
When $X^G$ is empty, i.e., the involution acts freely on the 
Real space $X$, we only need to bound the exponent of $WR(X)$.
Recall from Example \ref{WRfree} that $WR(S^{d+1,0})\cong\Z/2^f$,
where $S^{d+1,0}$ is $S^{d}$ with the antipodal involution and $f=f(d)$,
(at least if $d\equiv0,1,3,7\pmod8$; if $d\equiv2,4,5,6\pmod8$
the group is $\Z/2^f$ or $\Z/2^{f-1}$).

For example, we saw in Theorem \ref{W(curve)} that 
$WR(X)\cong\Z/4$ for the Riemann sphere $X\cong S^{3,0}$ defined 
by $X^2+Y^2+Z^2$. This example is typical in the following sense.

\begin{theorem}\label{WR:exponent} 
Suppose that the action of $G=C_{2}$ on $X$ is free.
Then the Real Witt group $WR(X)$ is a $2$-primary torsion group 
of exponent $2^f$, where $f=f(d)$ and $d=\dim(X)$.

The same statement is true for the co-Witt group ${WR\,}'(X)$.
\end{theorem}

\begin{proof}
The map $KO(X/G)\cong KO_{G}(X)\map{\rho} KR(X)$ of \eqref{eq:KOG-KR},
composed with the forgetful map $KR(X)\to KO_{G}(X)\cong KO(X/G)$
is multiplication by $1+L$, where $L$ is the canonical line bundle
associated to $X\to X/G$ and described in Example \ref{ex:canonical}. 
Since $L$ is classified by a cellular map
$\alpha:X/G\to\RP^\infty$, whose image lies in the $d$-skeleton $\RP^d$,
the bundle $L$ is the pullback $\alpha^*\xi$ via $\alpha:X/G\to\RP^d$ 
of the canonical line bundle $\xi$ over $\RP^d$. 
Now $(\xi+1)^{f}-2^f=0$ in $KO(\RP^{d})$, by Example \ref{WRfree}.
Applying $\alpha^*$ yields $2^f=(L+1)^f$ in $KO(X/G)$, and so
$KO(X/G)/(L+1)$ has exponent $2^f$. As $WR(X)$ is a quotient
of $KO(X/G)/(L+1)$, by Lemma \ref{lem:L=-1}, $WR(X)$ also has exponent $2^f$.

The assertion for the co-Witt group is proven in the same way,
taking kernels instead of cokernels.
\end{proof}

Theorem \ref{WR:exponent} is also true in the relative case.
If $Y$ is a $G$-invariant subcomplex of $X$, we write
$WR(X,Y)$ for the cokernel of the map $KR(X,Y)\to KO_{G}(X,Y)$.
The relative co-Witt group $WR^{\,'}(X,Y)$ is defined similarly,
taking kernels instead of cokernels.

\begin{subvariant}
Suppose that that $G$ acts freely on $X-Y$, 
where $Y$ is a $G$-invariant subcomplex. 
Then $WR(X,Y)$ and $WR'(X,Y)$ are
2-primary torsion groups of exponent $2^f$, $f=\dim(X-Y)$.
\end{subvariant}

\begin{proof}
When $G$ acts freely on $X$, the formal argument of 
Theorem \ref{WR:exponent} goes through, for then 
$KO_G(X,Y)\cong KO(X/G,Y/G)$ and its given endomorphism is multiplication
by $(1+L)$. Next, suppose that $Y$ contains an open subset $U$ of $X$
which contains $X^G$; then $KO_G(X,Y)\cong KO_G(X-U,Y-U)$ by excision,
and the result follows from the free case. The general result now
follows from Lemma \ref{colim} below.
\end{proof}


\begin{sublem}\label{colim}
If $X$ is a $G$-space and $Y$ an invariant subspace, so that 
$G$ acts freely on $X-Y$, then
\[  
\varinjlim\nolimits_Z KO_G(X,Z) \map{\cong} KO_G(X,Y),
\]
where $Z$ runs over the set of $G$-invariant closed neighborhoods of $Y$.
\end{sublem}

\begin{proof}
This is the $KO_G$-version of a standard argument in topological
$K$-theory, given for example in Lemma II.4.22 (p.\,91) of \cite{MKbook}.
\end{proof}

The effect of Lemma \ref{colim} is to extend assertions about
$KO_G(X,Y)$ from the case where $G$ acts freely on $X$ to 
the case where $G$ acts freely on $X-Y$.

In order to extend Theorem \ref{WR:exponent} to the case when 
$G$ is not acting freely, we begin with some general remarks which
hold in any Hermitian category $\cE$. Consider the involution 
$\gamma:(E,\varphi)\mapsto(E,-\varphi)$ on ${\mGW_0}(\cE)$.
If $K_0(\cE)$ has the trivial involution, the hyperbolic map 
$K_0(\cE)\to {\mGW_0}(\cE)$ is equivariant, so 
$\gamma$ acts on the Witt group $W(\cE)$.

When $\cE_X$ is the category of Real vector bundles on $X$, 
and $G$ acts freely on $X$, we know from 
Lemma \ref{lem:L=-1} that the involution is multiplication by the 
canonical element $\gamma$ in $GR(X)$.

\begin{lemma}\label{W-GW.map} 
For any Hermitian category $\cE$, $\gamma$ acts as
multiplication by $-1$ on $W(\cE)$, and $1-\gamma$ induces a
functorial map 
\[
f_{\cE}:W(\cE)\to {\mGW_0}(\cE)
\]
whose composition with the projection onto $W(\cE)$
is multiplication by 2 on $W(\cE)$. Moreover, 
$\gamma f_{\cE}(x)=f_{\cE}(\gamma x)=-f_{\cE}(x)$ for all $x\in W(\cE)$.
\end{lemma}

\begin{proof}
Since $(E,q)\oplus (E,-q)$ is a metabolic form (the diagonal copy of $E$ is
a Lagrangian), $[E,\varphi ]+[E,-\varphi ]=0$ in $W(\cE)$.
Therefore, $\gamma $ acts on the quotient $W(\cE)$ as multiplication
by $-1$.

Consider the endomorphism of ${\mGW_0}(\cE)$ sending the class of an
Hermitian module $(E,\varphi)$ to the formal difference $(E,\varphi)-(E,-%
\varphi).$ It sends the class of a hyperbolic form to zero, so it induces a
map $f_{\cE}$, as claimed. For $x\in W(\cE)$, the image of $%
f_{\cE}(x)$ in $W(\cE)$ is $x-\gamma(x)=2x$.
\end{proof}

Let $GW^{-}(\cE)$ denote the antisymmetric
part of ${\mGW_0}(\cE)$ under the involution $\gamma$. 
Lemma \ref{W-GW.map} states that the image of the map $f_{\cE}$ 
lies in $GW^{-}(\cE)$, and if $\bar{x}\in W(\cE)$ is the image of
$x\in  GW^{-}(\cE)$ then $f_{\cE}(\bar{x})=(1-\gamma)x=2x$. This proves:

\begin{subcor}\label{GW-}
The induced map
$f_{\cE}:W(\cE) \to GW^{-}(\cE)$ and 
the projection $GW^{-}(\cE)\to W(\cE)$ 
have kernels and cokernels which have exponent $2$.
In particular, we have an isomorphism 
\[
GW^{-}(\cE)\otimes \Z[1/2]\cong W(\cE)\otimes \Z[1/2].
\]
\end{subcor}

Combined with Example \ref{WR(X^G)}, our next result shows that $WR(X)$
differs from $KO(X^G)\cong WR(X^G)$ by at most 2-primary torsion.

\begin{theorem}\label{WR-KO.exponents} 
For any Real space $X$, the kernel and cokernel of the restriction
\[
WR(X)\overset{\gamma }{\to }WR(X^{G})=KO(X^{G})
\]
are 2-primary groups of exponent 
$2^{1+f}$, where $f=f(d)$ and $d=\dim(X - X^G)$.
If $X^G$ is a retract of $X$, the exponent is $2^f$.

The same statement is true for the restriction map of co-Witt groups,
$WR{\,}'(X)\to WR{\,}'(X^G)=KO(X^G)$.
\end{theorem}

\begin{proof}
The group $G$ acts freely on $X-X^{G}$.
We have a commutative diagram whose rows are exact by 
Definition \ref{def:GR} and whose left two columns are exact by excision:
\begin{equation*}
\begin{array}{ccccccc}
KR(X,X^G) & \to & GR(X,X^G) & \to & WR(X,X^G) & \to & 0 \\ 
\downarrow &  & \downarrow &  & \downarrow &  &  \\ 
KR(X) & \to & GR(X) & \to & WR(X) & \to & 0 \\ 
\downarrow &  & \downarrow &  & \downarrow\gamma &  &  \\ 
KR(X^{G}) & \to & GR(X^{G}) & \to & WR(X^{G}) & \to
& 0 \\ 
\downarrow &  & \downarrow &  & \downarrow &  &  \\ 
KR^{1}(X,X^G) & \to & GR^{1}(X,X^G) & \to & WR^{1}(X,X^G) & \to & 0.%
\end{array}%
\end{equation*}
Suppose that $a\in WR(X)$ is such that $\gamma(a)=0$. If
$f_X$ denotes the map $WR(X)\to GR(X)$ defined in Lemma \ref{W-GW.map},
$f_{X}(a)$ is an element of $GR(X)$ whose image in $WR(X)$ is $2a$,
and whose image in $GR(X^{G})$ is $f_{X^G}(\gamma a)=0$. 
Thus $f_X(a)$ comes from $GR(X,X^G)$ and 
$2a$ is in the image of $WR(X,X^G)$; since 
$2^{f}\cdot WR(X,X^G)=0$ by Theorem \ref{WR:exponent}, $2^{1+f}a=0.$

If $X^G$ is a retract of $X$, we have a split short exact sequence
\[
0 \to WR(X,X^G)\to WR(X) \to WR(X^G) \to 0.
\]
Since $2^f\cdot WR(X,X^G)=0$, the conclusion of the theorem is obvious.

The proof is analogous for the cokernel of $\gamma$. Let $\bar{a}$ denote
the image of an element $a\in WR(X^{G})$ in $\coker(\gamma)$. 
By Theorem \ref{WR:exponent}, ${WR}^{1}(X,X^G)$ is also a group of exponent
$2^{f}$, so $2^{f}\cdot f_{X^{G}}(a)$ vanishes in $GR^{1}(X,X^G)$ and
hence comes from an element $b$ of $GR(X)$. Since the image of $b$ in 
$WR(X^{G})$ is $2^{f}\cdot2a$, $2^{1+f}\bar{a}=0$ in $\coker(\gamma)$.

Finally, the assertions for the co-Witt groups may be proven in the
same way, by reversing the arrows in the above diagram and using
kernels in place of cokernels.
\end{proof}

\begin{rem}\label{rem:Wn-WRn} 
The previous considerations are also valid for the Witt
groups ${}_{\eps}WR_n(X)$ with $n\in\Z$ and $\eps=\pm1$,
since these groups are modules over the ring $WR(X)={}_{+1}WR(X)$.
For example, if $G$ acts freely on $X$ then $2^f=0$ in $WR(X)$
by Theorem \ref{WR:exponent} so the groups ${}_{\eps}WR_n(X)$
all have exponent $2^f$.
The analogues of Theorem \ref{WR-KO.exponents}, concerning the map
$WR_n(X)\to WR_n(X^G)\cong KO_n(X^G)$,
are left to the reader.
\end{rem}


%
%
Brumfiel's theorem \ref{thm:Brumfiel} allows us to compare $W$ and $WR$.

\begin{theorem}
\label{W=WR.mod2tors} 
If $V$ is any variety over $\R$, with associated
Real space $X$, the comparison map $W(V)\to WR(X)$ is an isomorphism
modulo 2-primary torsion, with finite cokernel. 

The same is true for the maps $W_{n}(V)\to WR_{n}(X)$ for all $n\in\Z$.
\end{theorem}

\begin{proof}
The composition of $W(V)\to WR(X)$ with the restriction map 
$WR(X)\map{\gamma}WR(X^{G})$ and the isomorphism $WR(X^G)\cong KO(X^G)$
of Example \ref{WR(X^G)}(a) is a map 
\begin{equation*}
W(V)\to WR(X)\map{\gamma}WR(X^{G})\cong KO(V_{\R}).
\end{equation*}
It is induced by the functor which associates to an algebraic vector bundle
on $V$ its underlying topological real bundle over the space 
$V_{\R}=X^G$. The composition $W(V)\to KO(V_{\R})$ is an
isomorphism modulo 2-primary torsion by Brumfiel's Theorem \ref{thm:Brumfiel}.
Hence the kernel of $W(V)\to WR(X)$ is a 2-primary torsion group.

Since $X$ has the homotopy type of a finite $G$-CW complex, 
the kernel and cokernel of
the map $WR(X)\to WR(X^{G})$ are 2-groups of bounded exponent by
Theorem \ref{WR-KO.exponents}. Hence the cokernel of $W(V)\to WR(X)$
is also a 2-primary torsion group, as claimed. Since $WR(V)$ is 
finitely generated, the cokernel is in fact a finite group.

The result for $W_{n}(V)\to WR_{n}(X)$ follows by the same argument, 
using Remark \ref{rem:Wn-WRn}.
\end{proof}

By Lemma \ref{2period}, ${}_{\eps}W_n(V)\ \map{u_2}\ {}_{-\eps}W_{n+2}$
is an isomorphism modulo $2$-torsion. The same is true for
${}_{\eps}WR_n(V)\ \map{u_2}\ {}_{-\eps}WR_{n+2}$.
Therefore 
Theorem \ref{W=WR.mod2tors} extends to the skew-symmetric case:

\begin{corollary}\label{-W=-WR}
If $V$ is any variety over $\R$, the maps
${}_{-1}W_{n}(V)\to {}_{-1}WR_{n}(V)$ are isomorphisms modulo
2-primary torsion for all $n\in\Z$.
\end{corollary}

\goodbreak
\bigskip 
\section{Exponents for $W(V)$}
\label{sec:exponents}

We now come back to our setting of an algebraic variety $V$ defined over $%
\R$. By Theorem \ref{W=WR.mod2tors},
we know that for all $n$ the kernel of the map%
\[
\theta_n: W_{n}(V) \map{} WR_{n}(V)
\]
is a 2-primary torsion group, and the cokernel is a finite 2-group.

Unfortunately, $W(V)=W_0(V)$ is not finitely generated in general;
there are complex 3-folds for which $W(V)/2$ is not finite.
Parimala pointed out in \cite{Parimala} that for a smooth 
3-fold $V$, the Witt group $W(V)$ is finitely generated if and 
only if $CH^{2}(V)/2$ is finite. (Parimala assumed that $V$ was affine,
but this assumption was removed by Totaro \cite{Totaro-Witt}.)
Based upon the work of Schoen \cite{Schoen}, 
Totaro \cite{Totaro} has recently shown that the group $CH^2(V)/2$ is 
infinite for very general abelian 3-folds $V$ over $\C$; hence
$W(V)=W(V)/2$ is infinite. If we regard $V$ as being 
defined over $\R$ by restriction of scalars then the map
$W(V)\to WR(V)$ is defined and its kernel is generally infinitely generated.
If $V_0=\Spec(A)$ is an affine open subvariety of $V$, 
$CH^2(V_0)/2$ will also be infinite, so 
$W(A)$ is not finitely generated either.

Our next result shows that the (possibly infinite) kernel of $\theta_n$
has a bounded exponent. It will be used to give more precise bounds in
Theorem \ref{W-WR.exponent}.
In addition to Theorem \ref{W=WR.mod2tors}, our
proof uses the main results in \cite{BKOS} and a variant of 
Bott periodicity, given in Appendix \ref{app:Bott}, which was 
originally proved in \cite{MKAnnalsH}.
A similar result (with a different proof) has been given by 
Jacobson in \cite{Jeremy}.

\begin{proposition}\label{exp.2^N}
For all $n\ge0$ and $d$ there exists an integer $N$
such that, for every $d$-dimensional algebraic variety $V$ over $\R$,
the kernel of the map 
$W_{n}(V)\overset{\theta_n}{\to}WR_{n}(V)$ is killed by $2^{N}.$

The same is true for the kernel of ${}_{-1}W_n(V)\to {}_{-1}WR_n(V)$.
\end{proposition}

\begin{proof}
Let $GW^-$ and $GR^-$ denote the antisymmetric parts of ${\mGW}$ and $GR$
for the involution defined in Section \ref{sec:WR}.
In the commutative diagram%
\begin{equation*}
\begin{array}{ccc}
GW^{-}_{n}(V) & \map{}  & GR^{-}_{n}(V) \\ 
\downarrow  &  & \downarrow  \\ 
W_{n}(V) & \map{\theta_n}  & WR_{n}(V).
\end{array}%
\end{equation*}%
 the kernels and cokernels of the vertical maps have exponent~2,
by Corollary \ref{GW-}.
By Theorem \ref{W=WR.mod2tors}, $\theta_n$ is an isomorphism modulo
2-primary torsion.  Therefore, 
we have an isomorphism for all $n$:
\begin{equation*}
GW^{-}_{n}(V)\otimes\Z[1/2]\cong GR^{-}_{n}(V)\otimes \Z[1/2].
\end{equation*}%
Now let $GW_{n}^{c,-}(V)$ denote the antisymmetric subgroup of
the comparison group $GW_{n}^{c}(V)$. It fits into a
chain complex 
\begin{equation*}
GR^{-}_{n+1}(V)\to GW_{n}^{c,-}(V)\to GW^{-}_{n}(V)\to GR^{-}_{n}(V)
\to GW^{c,-}_{n-1}(V)
\end{equation*}%
whose homology groups have exponent~2.
It follows that $GW_{n}^{c,-}(V)$ is a 2-primary torsion group for all $n$.

Now suppose that $n\geq d-2.$ 
Then $GW_{n}^{c}(V)$ and $GW_{n-1}^{c}(V)$ are 2-torsionfree
by Theorem \ref{GW=GR}. 
Therefore the subgroups $GW_{n}^{c,-}(V)$ and $GW_{n-1}^{c,-}(V)$
must be zero.
It follows that the kernel and cokernel of $GW^{-}_{n}(V)\to GR^{-}_{n}(V)$ 
have exponent~2.
Because of the commutative square above, the kernel
and cokernel of $W_{n}(V)\to WR_{n}(V)$ have 
exponent $8$, proving the theorem when $n\ge d-2$.

In order to prove the theorem for general $n\geq 0,$ we use the periodicity
maps $W_{n}\to W_{n+4}$ and $W_{n+4}\to W_{n},$ 
established in \cite{MKAnnalsH}, whose composition is multiplication 
by 16 by Corollary \ref{4period}. (See Appendix B for more details). 
Therefore, if we choose an
integer $k$ such that $n+4k\geq d-2,$ we see that the kernel and cokernel of
the map $W_{n}(V)\overset{\theta }{\to }WR_{n}(V)$ are killed by $%
8.16^{k}=2^{4k+3}.$

The final assertion is proven the same way, using
Corollary \ref{-W=-WR} and Theorem \ref{GW=GR{-1}} in place of
Theorems \ref{W=WR.mod2tors} and \ref{GW=GR}.
\end{proof}

\begin{subremark}\label{p.189:Knebusch} 
Suppose that $V$ is a smooth (or even divisorial) variety with no 
real points. Knebusch proved in \cite[p.\,189, Theorem 3]{Knebusch.Queens}
that $W(V)$ is a 2-primary torsion group of bounded exponent. 
Combining this with Theorem \ref{W=WR.mod2tors},
we get another proof of Proposition \ref{exp.2^N} 
for this class of varieties.
\end{subremark}

We can improve Proposition \ref{exp.2^N}, 
giving explicit bounds for the exponents of the kernel and cokernel.
We first consider large $n$.

\begin{proposition}\label{WR/W.2-group} 
Let $V$ be a variety over $\R$. 
If $n\ge\dim(V)-2$, the kernel and cokernel of 
$\theta_n:W_n(V)\to WR_n(V)$ have exponent~2.

The kernel and cokernel of ${}_{-1}W_n(V)\to {}_{-1}WR_n(V)$
also have exponent~2.
\end{proposition}


\begin{proof}
For $n\ge\dim(V)-2$ we consider the following diagram.
\begin{equation*}
\begin{array}{ccccccc}
K_{n}^{c}(V) & \to & GW_{n}^{c}(V) &  &  &  &  \\ 
\downarrow &  & \downarrow u &  &  &  &  \\ 
K_{n}(V) & \to & {\kGW_{n}}(V) & \overset{v}{\to } & W_{n}(V) & 
\to & 0 \\ 
\downarrow &  & \downarrow \beta &  & \downarrow \theta _{n} &  &  \\ 
KR_{n}(V) & \to & GR_{n}(V) & \overset{w}{\to } & WR_{n}(V)
& \to & 0 \\ 
&  & \downarrow \partial &  &  &  &  \\ 
&  & GW_{n-1}^{c}(V) &  &  &  & 
\end{array}%
\end{equation*}%
By Theorem \ref{W=WR.mod2tors}, $\ker(\theta_{n})$ and $\coker(\theta_{n})$ 
are $2$-primary torsion groups. 
On the other hand, by Lemma \ref{W-GW.map}, there are maps 
$v':W_{n}(V)\to {\kGW_{n}}(V)$ and $w':WR_{n}(V)\to GR_{n}(V),$ 
compatible with $\beta$ and $\theta_n$, 
such that $v\circ v'$ and $w\circ w'$ are multiplication by $2$. 
Since $GW_{n-1}^{c}(V)$ is
2-torsionfree by Theorem \ref{GW=GR}, the map 
$\coker(\theta _{n})\to GW_{n-1}^{c}(V)$ induced by $\partial w'$
must be zero. That is, if $a\in WR_{n}(V)$ then $w'(a)=\beta (b)$
for some $b$, and hence $2a=w(w'(a))$ is $\theta_{n}(v(b))$. Thus 
$\coker(\theta_{n})$ has exponent~2.

The image $D$ of $u:GW_{n}^{c}(V)\to {\kGW_n}(V)$ 
is a divisible 2-group, and its image $v(D)$ in $\ker(\theta_{n})$ 
is zero by Proposition \ref{exp.2^N}.
If $x\in \ker(\theta_n)$ then $v'(x)$ is in $\ker(\beta)=D$ and hence 
$2x=v(v'(x))$ is in $v(D)=0$. Thus $2\ker(\theta_n)=0$.

To obtain the final assertion, replace Theorem \ref{W=WR.mod2tors}
by Corollary \ref{-W=-WR}.
\end{proof}

Let $V$ be a variety over $\R$.  The 
{\it higher signature maps} are the maps 
\[
W_n(V)\map{}WR_n(V_{\R})\cong KO_n(V_{\R}).
\]
studied by Brumfiel (Theorem \ref{thm:Brumfiel}).
If $V_{\R}$ has $c$ connected components, the classical 
{\it signature map} $W(V)\to\Z^c$ may be regarded as the 
higher signature 
followed by the rank $KO(V_{\R})\map{}\Z^c$.
Combining Theorem \ref{WR-KO.exponents} and Proposition \ref{WR/W.2-group},
we obtain the following result.
Note that $X=V_{\C}$ and $X-X^G$ have dimension 
$2\dim(V)$ as CW complexes.

\begin{corollary}
Let $V$ be a variety over $\R$ of dimension $d$.
If $n\ge d-2$, the kernel and cokernel of the higher signature map
$W_n(V) \to KO_n(V_{\R})$ have exponent $2^{f+2}$,
where $f=f(2d)$.
\end{corollary}

If $x\in\R$, recall that $\lceil x\rceil$ denotes the least integer $n\ge x$.

\begin{theorem}\label{W-WR.exponent} 
Let $V$ be an algebraic variety over $\R$ of dimension $d$. 
The kernel and cokernel of the comparison map 
\[
W(V)\to {}{}WR(V)
\]
are of exponent $4\cdot 16^{m}$, where  $m=\lceil(d-2)/8\rceil$.

The same statement is true for the co-Witt groups $W'(V)$,
and for the skew-symmetric Witt groups $_{-1}W(V)$.
\end{theorem}

\goodbreak

\begin{proof}
Fix $n=8m$. As $n\ge d\!-\!2$, the kernel and cokernel of  
\[
\theta_{n}: W_{n}(V) \map{} WR_{n}(V)
\]
have exponent $2$, by Proposition \ref{WR/W.2-group}.
Consider the commutative diagram: 
\[
\begin{array}{ccccc}
{}W_{0}(V) & \map{u_{n}} & {}W_{n}(V) & \map{u_{-n}} & W_{0}(V)\\ 
\downarrow \theta &  & \downarrow \theta _{n} &  & \downarrow \theta \\ 
{}WR_{0}(V) & \map{u_{n}} & {}WR_{n}(V) & \map{u_{-n}} & WR_{0}(V)%
\end{array}%
\]
where $u_{n}$ and $u_{-n}$ are the periodicity maps of 
Lemma \ref{Bott.correct}. 
By Corollary \ref{period.correct}, the compositions 
$u_{-n}u_{n}$ are multiplication by $2\cdot 16^{m}$. 
Since the kernel and cokernel of $\theta _{n}$ have exponent~2, the kernel
and cokernel of $\theta $ have exponent $4\cdot 16^{m}$.
The proof for the skew-symmetric Witt groups is the same.

The theorem is proved for co-Witt groups 
in the same way, 
taking into account that the maps $u_{-8}$ may be factored through
the co-Witt groups $W'_{-8}(\Z[1/2])$ and $WR{\,}'_{-8}(\Z[1/2])$; 
see the proof of Lemma \ref{Bott.correct}.
\end{proof}

\begin{corollary}\label{W:exponent}
If $V$ has no $\R$-points, $W(V)$ has exponent $2^{2+4m+f(2d)}$.
\end{corollary}

\begin{proof}
Combine Theorems \ref{WR:exponent} and \ref{W-WR.exponent}.
\end{proof}

Corollary \ref{W:exponent} is not the best possible bound.
When $\dim(V)=1$, and $V$ is a smooth projective curve 
with no real points, $W(V)$ has exponent 2 or~4 by Theorem \ref{W(curve)}.
When $\dim(V)=2$, we know from Example \ref{W(surface)} that
$W(V)$ has exponent 4 or~8. 

\begin{theorem}\label{signature:8m+2}
The kernel and cokernel of the signature
$W(V) \to KO(V_{\R})$ have exponent $2^{3+4m+f}$,
where $f=f(2d)$ and $m=\lceil(d-2)/8\rceil$.
%
\end{theorem}

\begin{proof}
If $w\in W(V)$ vanishes in $KO(V_{\R})$ then $2^{1+f} w$ 
must vanish in $WR(V)$, by Theorem \ref{WR-KO.exponents}.
By Theorem \ref{W-WR.exponent}, $2^{2+4m}(2^{1+f} w)=0$.
A similar argument applies to the cokernel.
\end{proof}



\bigskip
We now turn our attention to varieties of arbitrary dimension $d$, starting
with $d\leq 8.$ We saw in Theorem \ref{curves:W=WR} that 
if $\dim(V)=1$ then $\theta:W(V)\to WR(V)$ is an isomorphism.
If $d=2$ we have the following immediate consequence of
Proposition \ref{WR/W.2-group}.

\begin{proposition}\label{W/WR.dim2}
For any variety $V$ over $\R$ of dimension $2$,
the kernel and cokernel of 
$W(V)\map{\theta} WR(V)$ have exponent at most $2.$ 

The same is true for ${}_{-1}W(V)\to{}_{-1}WR(V)$.
\end{proposition}

\begin{example}\label{W(surface)} 
If $V$ is a smooth projective surface over $\R$, Sujatha has
computed $W(V)$ in \cite[3.1--3.2]{Sujatha}.  If $V_{\R}$ has $c>0$
components, $W(V)$ is the sum of $\Z^c$ and a torsion group of the form
$(\Z/2)^m\oplus(\Z/4)^n$; if $V_{\R}=\emptyset$, $W(V)$ has the
form $(\Z/2)^m\oplus(\Z/4)^n\oplus(\Z/8)^t$ where $t\le1$.
Thus $WR(V)$ detects a nontrivial piece of $W(V)$.
\end{example}

\begin{lemma}\label{W.dim3} 
If $V$ is a smooth variety over $\R$ of dimension $d=3$, 
the torsion subgroup of $\ker(W_0(V)\to\Z/2)$ has exponent $8$. 
\end{lemma}

\begin{proof}
Let $F=\R(V)$ denote the function field of $V$. According to Pardon 
\cite[Thm.\,A]{Pardon} and Balmer--Walter Purity \cite[10.3]{BalmerW}, 
$W(V)$ injects into $W(F)$. If $I$ is the kernel of $W(F)\to\Z/2$, it is
known (see \cite[35.29]{EKM} for example) that the ideal $I^n$ of $W(F)$ is
torsionfree for $n>d=3$. Since $2\in I$, the torsion subgroup of $I$ has
exponent $2^d=8$.
\end{proof}

For 3-folds, 
the cokernel of $W(V)\to WR(V)$ has exponent 32 by
Theorem \ref{W(456)} below. (Theorem \ref{W-WR.exponent} gives an
upper bound of only 64, and
Proposition \ref{WR/W.2-group} does not apply in this case.)

\goodbreak

\begin{theorem}\label{W(456)} 
Let $V$ be an algebraic variety of dimension at most 6 
over $\R$. Then the kernel and cokernel of the associated map 
\[
\theta: {}W_{0}(V) \map{}{}WR_{0}(V)
\]
have exponent 32. The same property holds for co-Witt groups.
\end{theorem}

\begin{proof}
Consider the diagram 
\begin{equation*}
\begin{array}{ccccc}
{}W_{0}(V) & \map{u_{4}} & {}W_{4}(V) & \map{u_{-4}} & {}W_{0}(V) \\ 
\downarrow \theta &  & \downarrow \theta _{4} &  & \downarrow \theta \\ 
{}WR_{0}(V) & \map{u_{4}} & {}WR_{4}(V) & \map{u_{-4}} & {}WR_{0}(V).%
\end{array}%
\end{equation*}%
By Corollary \ref{4period}, 
the horizontal compositions are multiplication by 16. 
By Proposition \ref{WR/W.2-group}, both the kernel and cokernel of 
$\theta_{4}$ have exponent~2.
The result now follows from a diagram chase.
\end{proof}

For varieties of dimension 4--6, Theorem \ref{W(456)} says that in
passing from $W(V)$ to $WR(V)$ we lose up to 5 powers of~2.
For varieties of dimension $d$, $7\leq d\leq 10$, we lose another power
of~2.

\begin{theorem}\label{W(7-10)}
Let V be an algebraic variety over $\R$ of dimension $d$ with $7\leq
d\leq 10$. Then the kernel and cokernel of $\theta :W(V)\to WR(V)$
have exponent 64. 
\end{theorem}

\begin{proof}
The proof of Theorem \ref{W(456)} goes through, 
using Corollary \ref{period.correct}.
Alternatively, this follows from Theorem \ref{W-WR.exponent}.
\end{proof}

\goodbreak
Combining Theorem \ref{WR-KO.exponents} with Theorems \ref{W(456)} 
and \ref{W(7-10)}, we get a slight improvement over the exponent
$2^{7+f(2d)}$ predicted by Theorem \ref{signature:8m+2} 
for varieties of dimension $d\le10$.

\begin{corollary}
If $\dim(V)\le6$ (resp., $\dim(V)\le10$) the kernel and cokernel 
of the signature map $W_0(V)\to KO(V_{\R})$ have exponent
$64\cdot2^f$ (resp., $128\cdot2^f$). Here $f=f(2\dim V)$.
\end{corollary}

\goodbreak
Here is the general result for varieties of dimension $d>10$. It is proved
by the same method. 

\begin{theorem}\label{W(11+)}
Let $V$ be a variety of dimension $d\ge11$ over $\R$. 
Then the kernel and cokernel of $W(V)\to WR(V)$ have exponent $2^{r}$, 
where $2^{r}$ is given by the following table (for $m>0)$:

\qquad If $d=8m, 8m\pm1$ or $8m+2$ then $2^{r}=4.16^{m}$

\qquad If $d=8m+3$ or $8m+4$ then $2^{r}=16^{m+1}$ 

\qquad If $d=8m+5$ or $8m+6$ then $2^{r}=4.16^{m+1}$. 


\smallskip\noindent
The kernel and cokernel of the signature map $W(V)\to KO(V_{\R})$
have exponent $2^{r+f+1}$, where $f=f(2d)$.
\end{theorem}

\begin{proof}
If $d=8m+3$ or $8m+4$, we consider the commutative diagram
\[\xymatrix{W_{0}(V)\ar[d]\ar[r]^{u_2} & 
{}_{-1}W_2(V) \ar[d] \ar[r] \ar@/^/[l]
& {}_{-1}W_{8m+2}(V) \ar[d] \ar@/^/[l] \\
WR_{0}(V)\ar[r]^{u_2} &{}_{-1}WR_2(V) \ar[r]\ar@/^/[l] & 
{}_{-1}WR_{8m+2}(V), \ar@/^/[l]}
\]
whose first pair of left-right composites are multiplication by 4
(see Lemma \ref{2period}), and whose second pair 
of left-right composites are multiplication by $2.16^m$, 
by Corollary \ref{period.correct})
By the skew-symmetric version of Proposition \ref{WR/W.2-group},
the kernel and cokernel of the right vertical map
have exponent~2.
The assertion about $2^r$ follows by diagram chasing.

When $d\not\equiv3,4\pmod8$, the assertions in the first part 
(about $2^r$) follow from Theorem \ref{W-WR.exponent}.
The last part of the theorem (about the signature map) is a consequence 
of the first part and Theorem \ref{WR-KO.exponents}. 
\end{proof}

\section{Co-Witt groups}\label{sec:co-Witt}

We briefly turn our attention to co-Witt groups, 
i.e., to the kernels $_{\varepsilon}W'_{n}$ and $_{\varepsilon}WR{\,}'_{n}$ 
of the forgetful maps $_{\varepsilon}{\kGW_n}\to K_n$ and
$_{\varepsilon}GR_n\to KR_n$,

Composing the injection ${}_{\varepsilon}W'_n(V)\to {}_{\varepsilon}{\kGW_n}(V)$
with the surjection ${}_{\varepsilon}{\kGW_n}(V)\to {}_{\varepsilon}W_n(V)$ 
gives a map $c:{}_{\varepsilon}W'_n(V) \to {}_{\varepsilon}W_n(V)$.
There is a similar map 
${}_{\varepsilon}WR{\,}'_n(V) \map{} {}_{\varepsilon}WR_n(V)$.
Recall from \cite[p.\,278]{MKAnnalsH} that the given maps $c$ fit
into compatible 12-term sequences, part of which is
\begin{equation}\label{eq:12term}
\begin{array}{ccccccc}
k'_n &\map{}&{}_{\varepsilon}W'_n(V)&\map{c}
            &{}_{\varepsilon}W_n(V) &\map{}&k_n
\\ 
\downarrow &&\downarrow \theta' && \downarrow & & \downarrow   \\ 
kr'_n &\map{}&{}_{\varepsilon}WR{\,}'_n(V)&\map{c}
             &{}_{\varepsilon}WR_n(V) &\map{}&kr_n
\end{array}
\end{equation}
where $k_{n}=H_1(G,K_{n}V)$ and $k'_{n}=H^{1}(G,K_{n}V)$,
$kr_n=H_1(G,KR_nV)$ and $kr'_n=H^{1}(G,K_{n}V)$.
($G$ acts on $K_n$ and $KR_n$ by duality.) 

\begin{lemma}\label{12-term}
The kernels and cokernels of 
${}_{\varepsilon}W'_n(V)\map{c}{}_{\varepsilon}W_n(V)$ and
${}_{\varepsilon}WR{\,}'_n(V)\map{c}{}_{\varepsilon}WR_n(V)$
have exponent~2.
\end{lemma}

\begin{proof}
Since $H_1(G,\!-\!)$ and $H^1(G,\!-\!)$ have exponent~2, the groups
$k_n$, $k'_n$, $kr_n$ and $kr'_n$ have exponent~2. The result 
is immediate from \eqref{eq:12term}.
\end{proof}

We now restrict to curves for simplicity.
Given the isomorphism $W(V)\cong WR(V)$ of Theorem \ref{curves:W=WR}
(and Proposition \ref{skewWR(curve)}),
the following result is immediate from Lemma \ref{12-term} 
and \eqref{eq:12term}.

\begin{corollary}\label{W'.upto2}
Let $V$ be a curve over $\R$. 
Then 
$W'(V)\map{\theta'}WR{\,}'(V)$
has kernel of exponent~2 and cokernel of exponent~4.

The same is true of
${}_{-1}W'(V)\map{\theta'} {}_{-1}WR{\,}'(V)$.
\end{corollary}

\goodbreak
In order to improve on this corollary, we need a
technical result.

\begin{proposition}\label{W'{-1}}
Let $V$ be a smooth curve over $R$ and $\varepsilon =\pm 1.$ Then the
canonical map $\theta'$ is an isomorphism:
\[
\theta': {}_{\eps}W'_{-1}(V)\map{\cong}\; {}_{\eps}WR{\,}'_{-1}(V).
\]
\end{proposition}

\begin{proof}
Consider the following commutative diagram%
\begin{equation*}
\begin{array}{ccccccc}
0 & \to & _{\varepsilon }W'_{-1}(V) & \map{\cong} & 
{}_{\varepsilon }{\kGW_{-1}}(V) & \to & K_{-1}(V)=0 \\ 
&  & \downarrow \theta' &  & \downarrow &  & \downarrow \\ 
0 & \to & _{\varepsilon }WR{\,}'_{-1}(V) & \to & 
_{\varepsilon }GR_{-1}(V) & \to & KR_{-1}(V)%
\end{array}%
\end{equation*}%
By a chase, $\ker(\theta')$ equals the kernel of 
$_{\varepsilon }{\kGW_{-1}}(V)\to{}_{\varepsilon}GR_{-1}(V).$ 
It is 2-divisible by Theorems \ref{GW=GR} and \ref{GW=GR{-1}},
and exponent~4 by Proposition \ref{WR/W.2-group}.
Therefore, $\ker(\theta')=0$.

To prove the surjectivity of $\theta'$, we use the sequence 
\eqref{seq:U-K-GW}: 
\[
K_{0}(V)\to {}_{-\varepsilon }U(V)\to {}_{\varepsilon
}{\kGW_{-1}}(V)\to K_{-1}(V)=0.
\]
This shows that $_{\varepsilon}W'_{-1}(V)$ is the cokernel of 
$K_{0}(V)\to{}_{-\varepsilon}U(V)$ (note the change of symmetry). 
Therefore, it is
enough to prove the surjectivity of the map between 
$\coker(K_{0}(V)\to {}_{-\varepsilon}U(V))$ and 
$\coker(KR_{0}(V)\to {}_{-\varepsilon }UR(V))$.
For this, we write the same diagram as in the proof of 
Theorem \ref{curves:W=WR}, with $U$-theory in lieu of $GW$-theory: 
\begin{equation*}
\begin{array}{ccccccc}
K_{0}(V) & \map{\alpha} & _{-\varepsilon }U_{0}(V) & \to & 
_{\varepsilon }{\kGW_{-1}}(V) & \to & K_{-1}(V)=0 \\ 
\downarrow &  & \downarrow \delta &  & \downarrow &  & \downarrow \\ 
KR_{0}(V) & \map{\beta} & _{-\varepsilon }UR_{0}(V) & \to & 
_{\varepsilon }GR_{-1}(V) & \to & KR_{-1}(V) \\ 
\downarrow &  & \downarrow \gamma &  & \downarrow &  & \downarrow \cong
\\ 
K_{-1}^{c}(V) & \to & _{-\varepsilon }U_{-1}^{c}(V) & \to & 
_{\varepsilon }GW_{-2}^{c}(V) & \to & K_{-2}^{c}(V)%
\end{array}
\end{equation*}
In this case, the top two rows are the exact sequences
\eqref{seq:U-K-GW} and \eqref{eq:UR-KR-GR}, and
$K_0(V)\to KR_0(V)$ is onto by Example \ref{ex:KRcurve}.
Hence the cokernel of $\coker(\alpha)\to\coker(\beta)$
injects into $_{-\varepsilon }U_{-1}^{c}(V)$.

The exact sequence in the second column may be extended to:
\begin{equation*}
_{-\varepsilon}U_{0}(V)\map{\delta} {}_{-\varepsilon}UR_{0}(V)\map{\gamma}
{}_{-\varepsilon }U_{-1}^{c}(V)\to {}_{-\varepsilon}U_{-1}(V)
\to {}_{-\varepsilon}UR_{-1}(V).
\end{equation*}%
Since $K_{-1}(V)=0,$ we see from \eqref{seq:U-K-GW} that the Witt group 
$_{-\varepsilon }W_{0}(V)$ is the group $_{-\varepsilon }U_{-1}(V)$.
We now distinguish the cases $\eps=\pm1$.

If $\varepsilon =1,$ the group ${}_{-1}W(V)$ is $0$ by
Proposition \ref{skewWR(curve)} and therefore $\gamma$ is onto. 
On the other hand, $_{-\varepsilon}UR_{0}(V)$ is finitely generated 
and $_{-\varepsilon }U_{-1}^{c}(V)$ is
uniquely $2$-divisible by Corollary \ref{U=UR}.  This implies that 
the image of $\gamma $ is a finite abelian group of odd order, i.e. 
$0$ after localizing at~2. 
Therefore, the map $\coker(\alpha)\to\coker(\beta)$ is onto.

If $\varepsilon =-1,$ the classical Witt group $W(V)={}_{-\varepsilon
}W_{0}(V)={}_{-\varepsilon }U_{-1}(V)$ is inserted in the exact sequence%
\begin{equation*}
_{-\varepsilon }U_{-1}^{c}(V)\to W(V)\to {}_{-\varepsilon
}UR_{-1}(V)\to {}_{-\varepsilon }U_{-2}^{c}(V)
\end{equation*}%
Since $_{-\varepsilon }U_{-1}^{c}(V)$ is uniquely 2-divisible and 
$W(V)=WR(V)$ is finitely generated, the first map is reduced to $0.$ 
This implies again that $\gamma$ is onto and therefore 
$_{-\varepsilon}U_{-1}^{c}(V)=0$. 
We now finish the proof as in the case $\varepsilon =1.$
\end{proof}

\goodbreak
\begin{theorem}\label{W'(curves)} 
Let $V$ be a smooth irreducible curve over $\R$. Then the map $\theta'$
between co-Witt groups is an injection whose cokernel $E$
is a finite group of exponent 2. 
\begin{equation*}
0\to {}_{\varepsilon }W'(V) \map{\theta'}~{}_{\varepsilon}{WR}'(V)\to E\to 0.
\end{equation*}
If $V$ is projective of genus $g$
then $\textrm{rank}(E)=g$; if $V$ is obtained from a projective curve 
by removing $r$ points then $g\le\textrm{rank}(E)\le g+r-1$.
\end{theorem}

\begin{proof}
We extend \eqref{eq:12term} slightly to the diagram
\begin{equation*}
\begin{array}{ccccccccccc}
{}_{-\varepsilon }W_{-1}' & \to & k_{0}' & 
\to & {}_{\varepsilon }W' & \to & {}_{\varepsilon }W
& \to & k_{0} & \to & _{-\varepsilon }W_{1} \\ 
\quad \downarrow \cong && \;\downarrow \cong && \downarrow \theta' && 
\downarrow \cong &  & \downarrow \rlap{\text{onto}} &  & \downarrow \cong\\ 
{}_{-\eps}WR{\,}'_{-1} & \to & kr_{0}^{'} & 
\to & {}_{\varepsilon }WR{\,}' & \to & {}_{\varepsilon
}WR & \to & kr_{0} & \to & _{-\varepsilon }WR_{1}.
\end{array}%
\end{equation*}%
%

We first explain the decorations on the vertical maps in the diagram.
Now $K_0(V)\to KR_0(V)$ is a split surjection (by Example \ref{ex:KRcurve}),
and $G$ acts as $-1$ on the kernel $D$, which is divisible.
Therefore the vertical $k'_0\to kr'_0$ is an isomorphism,
and the vertical $k_0\to kr_0$ is a split surjection with kernel 
${}_2D=\Hom(\Z/2,D)$.
The left vertical map is an isomorphism by Proposition \ref{W'{-1}},
and the map ${}_{\varepsilon}W\to{}_{\varepsilon}WR$ is an isomorphism
by Theorem \ref{curves:W=WR} and Proposition \ref{skewWR(curve)}.

If $\varepsilon=+1$, then 
$_{-\varepsilon}W_1(V)={}_{-\varepsilon}WR_1(V)=0$
by Corollary \ref{skewWR(curve)}.  If $\varepsilon=-1$, then 
$_{-\varepsilon}W_1(V)\cong {}_{-\varepsilon}WR_1(V)=0$
by Theorem \ref{curves:W=WR}.

By a diagram chase, we have an exact sequence
\[
0 \to {}_{\varepsilon}W'(V) \to {}_{\varepsilon}WR{\,}'(V) \to {}_2D \to0.
\]
It remains to observe that if $V$ is projective then $D\cong(\R/\Z)^g$
(so $_2D=(\Z/2)^g$), and if $V$ has $r$ points at infinity then 
$D$ is the cokernel of a map $\Z^{r-1}\to(\R/\Z)^g$, which gives the
bounds $g\le \dim({}_2D)\le g+r-1$.
\end{proof}

\newpage
\section{Higher Witt and co-Witt groups of $\R$ and $\C$}\label{sec:WnR}

In this section, we determine the Witt groups 
$W_n(\R)$, $W_n(\C)$ and co-Witt groups $W'_n(\R)$, $W'_n(\C)$ for $n>0$.
We will show that the
canonical maps $W_n(\R)\to WR_n(\R)$ and $W_n(\C)\to WR_n(\C)$ are 
almost isomorphisms, and similarly for $W'_n(\C)\to WR{\,}'_n(\C)$.
We also show that the maps $W'_n(\R)\to WR{\,}'_n(\R)$ are isomorphisms 
for all $n>0$. 

For $\R$, recall from Example \ref{WR(X^G)} 
that both $WR_n(\R)=WR_n(\Spec\R)$ (called $W^\topl_n(\R)$ in 
Appendix \ref{app:Bott}) and the co-Witt groups $WR{\,}'_n(\R)$ are 
isomorphic to the 8-periodic groups $KO_n(\pt)$.

\begin{theorem}\label{Wn(R)}
The homomorphism $W_n(\R)\map{\theta_n} WR_n(\R)\cong KO_n$ 
is an isomorphism for all $n>0$, except when $n\equiv0\pmod4$
when the group $W_n(\R)\cong\Z$ injects into $WR_n(\R)\cong\Z$ 
as a subgroup of index~2.
\end{theorem}

\begin{proof}
We first show that $\theta_n$ is injective for all $n>0$.
Take an element $\bar{x}$ in the kernel of $W_n(\R)\to WR_n(\R)$ and
lift it to an element $x$ of $GW_n(\R)$. 

\begin{equation*}
\begin{array}{ccccccc}
K_{n}^{c}(\R) & \to & GW_{n}^{c}(\R) &  &  &  &  \\ 
\downarrow &  & \downarrow u &  &  &  &  \\ 
K_{n}(\R) & \map{H} & GW_{n}(\R) & \map{v} & W_{n}(\R) & \to & 0 \\ 
\downarrow &  & \downarrow \beta &  & \downarrow \theta _{n} &  &  \\ 
KO_{n} & \to & GR_{n}(\R) & \overset{w}{\to } & WR_{n}(\R)
& \to & 0.
\end{array}%
\end{equation*}%

If $n\not\equiv0\pmod4$, the map $K_n(\R)\to KO_n(\pt)$ is onto.
By a diagram chase, we can modify $x$ by an element $H(a)$
to assume that $\beta(x)=0$. Since $GW^c_n(\R)$ is 2-divisible
(by Theorem \ref{GW=GR}), there is a $y\in GW^c_n(\R)$ such that
$u(2y)=x$ and hence $\bar{x}=2vu(y)$. Since the kernel of 
$\theta_n$ has exponent~2 (by Theorem \ref{WR/W.2-group})
this yields $\bar{x}=0$. Hence $\theta_n$ is an injection for these $n$.

If $n\equiv0,4\pmod8$, $2x$ comes from the 2-divisible $K_n(\R)$; 
modifying $x$ by $H(a)$, where $H(2a)=2x$, we may assume that $2x=0$.
Hence $\beta(x)$ is a torsion element in $GR_n(\R)$ lying in $\ker(w)$.
As pointed out in Example \ref{WR(X^G)}, $\ker(w)\cong KO_n\cong\Z$.
Thus $\beta(x)=0$. As before, there is a $y\in GW^c_n(\R)$ such
that $x=u(2y)$ and hence $\bar{x}=2\,vu(y)=0$.

Next, we show that $\theta_n$ is a surjection for $n\not\equiv0,4\pmod8$.
This is trivial for $n\equiv3,5,6,7\pmod8$, 
as $WR_n(\R)=0$ for these values. To see that $\theta_n$
is a surjection when $n\equiv1,2$ we need to show that the 
nonzero element of $WR_n(\R)\cong\Z/2$ is in
the image of $\theta_n$. This is true because the $J$-homomorphism 
$\pi_n^s\to KO_n$ is surjective, factoring as
\[
\pi_n^s = \pi_nB\Sigma^+\to\varinjlim\nolimits_m \pi_nBO^+_{m,m}(\R)
\to\pi_n BGL^+(\R)\to \pi_nBO = KO_n ,
\]
and $\varinjlim_m \pi_n BO^+_{m,m}(\R)=GW_n(\R)$ for $n>0$.

Now suppose that $n\equiv0,4\pmod8$. We know that $W_n(\R)$
injects into $WR_n(\R)\cong\Z$ as a subgroup of index at most~2,
by Theorem \ref{WR/W.2-group}. The 12-term sequences 
of \cite[p.\,278]{MKAnnalsH}
for $\R$ and $\R^\topl$ yield a commutative diagram
\[
\begin{array}{ccc}
\Z\cong W_n(\R) &\hookrightarrow& WR_n(\R)\cong\Z \\
\downarrow &&\downarrow\ \textrm{onto} \\
0=k_n(\R) & \map{} & kr_n(\R)\cong\Z/2.
\end{array}
\]
where $k_n(\R)=H_1(G,K_n\R)=0$ and $kr_n(\R)=H_1(G,KO_n)=\Z/2$.
This shows that the cokernel of $W_n(\R) \to WR_n(\R)$ is $\Z/2$.
\end{proof}

We have a stronger result for the co-Witt groups. Recall 
from Example \ref{WR(X^G)} that $GR_n(\R)\cong KO_n\oplus KO_n$
and that the forgetful map to $KR_n(\R)\cong KO_n$ is identified
with addition, so $WR{\,}'_n(\R)\cong KO_n$.
Regarding $\Spec(\C)$ as $X=S^0$ with the nontrivial involution, we also
see from Example \ref{WR(X^G)} that $GR_n(\C)\cong KO_n$.

\begin{proposition}
For all $n>0$, the map $W'_n(\R)\to WR{\,}'_n(\R)\cong KO_n$ 
is an isomorphism.
\end{proposition}

\begin{proof}
We first observe that the kernel and cokernel of 
$W'_n(\R)\to WR{\,}'_n(\R)$ are 2-primary torsion groups of bounded exponent.
This follows for example from Theorem \ref{Wn(R)} and the map of 
12-term sequences (see \cite[p.\,278]{MKAnnalsH})
\[
\begin{array}{ccccccc}
k'_n &\map{}& W'_n(\R)&\map{c}  & W_n(\R) &\map{}& k_n \\ 
\downarrow &&\downarrow \theta' && \downarrow & & \downarrow   \\ 
kr'_n &\map{}& WR{\,}'_n(\R)&\map{c} & WR_n(\R) &\map{}&kr_n
\end{array}
\]
since the groups $k'_n, k_n, kr'_n$ and $kr_n$ have exponent~2.

To see that $W'_n(\R)\to WR{\,}'_n(\R)$ is injective, recall from
Corollary \ref{Kc-u.d.} and Theorem \ref{GW=GR} that $K^c_n(\R)$ 
and $GW^c_n(\R)$ are 2-divisible groups. Therefore the pullback
$P$ in the following diagram is 2-divisible.
\[
\begin{array}{ccccc}
P & \hookrightarrow & GW^c_n(\R) & \to & K^c_n(\R) \\
\downarrow &&\downarrow && \downarrow \\
W'_n(\R) & \hookrightarrow & GW_n(\R) & \to & K_n(\R) \\
\downarrow &&\downarrow && \downarrow \\
WR{\,}'_n(\R)& \hookrightarrow & GR_n(\R) & \to & KO_n
\end{array}
\]
By a diagram chase, we see that the left column is exact,
so the kernel of $W'_n(\R)\to WR{\,}'_n(\R)$ is 2-divisible,
as well as having a bounded exponent, and hence must be zero.

\smallskip
Since ${WR{\,}'_n}(\R)=0$ for $n\equiv3,5,6,7\pmod8$,
it remains to consider the cases $n\equiv0,1,2,4\pmod8$. 
We will show that the cokernel of 
$W'_n(\Z[\frac12])\to\! WR{\,}'_n(\R)$ 
is a group of odd order; since it is also 2-torsion, it is zero.
For this we use the homotopy sequence associated to \eqref{square:BK}.
If $A$ is an abelian group, $A_{(2)}$ will denote the localization
of $A$ at the prime~2.

If $n\equiv0\pmod4$, we have 
$GR_n(\R)=\Z\oplus\Z$ and $GR_n(\C)=\Z$. 
By \cite{Friedlander}, $GW_n(\mathbb{F}_3)$ is either 0 or $\Z/2$.
Thus the homotopy exact sequence from \eqref{square:BK} is
\[
GW_n(\Z[{\textstyle\frac12}])_{(2)} \to
(\Z\oplus\Z)_{(2)}\oplus GW_n(\mathbb{F}_3)
\quad\map{}\quad \Z_{(2)}\to0.
\]
In addition, $K_4(\Z[\frac12])$
is a finite group of odd order, so 
$W'_n(\Z[\frac12])_{(2)}\cong GW_n(\Z[\frac12])_{(2)}$.  
Thus we have the following commutative
diagram, with the middle column exact.
\[
\begin{array}{ccccc}
W'_n(\Z[\frac12])_{(2)} & \map{\cong} & GW_n(\Z[\frac12])_{(2)}& 
\to & 0\qquad \\
\downarrow &&\downarrow &&  \\
{WR{\,}'_n}(\R)_{(2)} & \hookrightarrow & GR_n(\R)_{(2)}& \to & (KO_n)_{(2)}\quad\\
   &&\downarrow && \downarrow\,\qquad \\  
   && GR_n(\C)_{(2)}& \map{\textrm{into}} & (KU_n)_{(2)}\quad \\
\end{array}
\]
A diagram chase shows that $W'_n(\Z[\frac12])_{(2)}\to WR{\,}'_n(\R)_{(2)}$
is onto. This implies that the cokernel of 
$W'_n(\Z[\frac12])\to WR{\,}'_n(\R)$, and {\it a fortiori} 
the cokernel $C$ of $W'_n(\R)\to WR{\,}'_n(\R)$, is an odd torsion group. 
Since $C$ is a 2-primary torsion group (by Corollary \ref{W'.upto2}),
it is zero, i.e., $W'_n(\R)\to WR{\,}'_n(\R)$ is a surjection.

If $n\equiv1,2\pmod4$, we have $KO_n\cong\Z/2$, 
$GR_n(\R)\cong KO_n\oplus KO_n$,
and the maps to $KR_n(\R)\cong KO_n$ and to $GR_n(\C)\cong KO_n$
agree (both being addition). In addition,
${WR{\,}'_n}(\R)\cong KO_n$
and $GW_n(\mathbb{F}_3)$ is $\Z/2$ or $(\Z/2)^2$, so the 
homotopy exact sequence from \eqref{square:BK} is
\[
GR_{n-1}(\C)\to GW_n(\Z[{\textstyle\frac12}])_{(2)} \to (\Z/2)^2\oplus(\Z/2)^2 
\quad\map{\textrm{onto}}\quad\Z/2.
\]
From the diagram (whose middle column is exact)
\[
\begin{array}{ccccc}
 && (KO_{n-1})_{(2)} & \to & (KU_{n-1})_{(2)} \\
 && \downarrow && \downarrow \\
W'_n(\Z[\frac12])_{(2)}& \hookrightarrow & GW_n(\Z[\frac12])_{(2)}&
\to & K_n(\Z[\frac12])_{(2)} \\
\downarrow &&\downarrow && \downarrow \\
WR{\,}'_n(\R)\!\oplus\! GW_n(\mathbb{F}_3)& \hookrightarrow & 
GR_n(\R)\!\oplus\! GW_n(\mathbb{F}_3) & \to & {KO_n}\\
 &&\downarrow && \downarrow \\
  && GR_n(\C)& \map{} & (KU_n)_{(2)}
\end{array}
\]
we see that the map $WR{\,}'_n(\R)\to GR_n(\C)\cong\Z/2$ is zero.

If $n\equiv1\pmod8$, $K_n(\Z[\frac12])\cong\Z\oplus\Z/2$,
and the map from $KO_{n-1}\cong\Z$ to $KU_{n-1}$ is an isomorphism;
see \cite[VI.10.1]{WK}. 
A diagram chase shows that the vertical map 
$W'_n(\Z[\frac12])_{(2)}\to {WR{\,}'_n}(\R)\cong\Z/2$.
is onto. {\it A fortiori}, $W'_n(\R)\to{WR{\,}'_n}(\R)$ is onto.

If $n\equiv2\pmod8$, $K_n(\Z[\frac12])$ is the sum of
$\Z/2$ and a finite group of odd order; see \cite[VI.10.1]{WK}. 
Thus the map $K_n(\Z[\frac12])_{(2)}\to KO_n\cong\Z/2$ 
is an isomorphism. In this case, an easy diagram chase shows that 
$W'_n(\R)$ maps onto $WR{\,}'_n(\R)$.
This contradicts the injectivity part of the proof above.
\end{proof}

\medskip
Next we describe the Witt and co-Witt groups of $\C$, where
we consider $\C$ with the trivial involution.  Viewing
$\Spec(\C)$ as a variety over $\R$, the associated Real space 
$X$ of complex points is $S^0$ with the nontrivial involution.
by Theorem \ref{GR=KOG}, our groups $GR_n(\C)$ are just $KO_n$,
$KR_n(\C)$ is $KU_n$ and the Witt group $WR_n(\C)$ is the cokernel
of the forgetful map $KU_n\to KO_n$. Similarly, the co-Witt group
$WR{\,}'_n(\C)$ is the kernel of the complexification $KO_n\to KU_n$.
In either event, $WR_n(\C)$ and $WR{\,}'_n(\C)$ are either 0 or $\Z/2$.

Recall that $K_n(\C)$ is uniquely divisible for even $n>0$,
and is the sum of $\Q/\Z$ and a uniquely divisible group for odd $n>0$.

\begin{theorem}\label{WnC}
For $n>0$, $GW_n(\C)$ is the sum of a uniquely 2-divisible group 
and the 2-primary torsion group:
\[\begin{cases}
0,&    n\equiv 0,4,5,6\pmod8; \\
\Z/2,& n\equiv1,2     \pmod8; \\
(\Q/\Z)_{(2)},& n\equiv3,7\pmod8.
\end{cases}\]
The Witt groups are:
\[
W_n(\C) = \begin{cases} 
\Z/2,& n\equiv1,2\pmod8; \\0,& \textrm{otherwise}.
\end{cases}
\]
The co-Witt groups are:
\[
W'_n(\C)=\begin{cases}
\Z/2,&n\equiv2,3\pmod8; \\0,& \textrm{otherwise}.
\end{cases}
\]
The maps $W_n(\C)\to\! WR_n(\C)$ are isomorphisms for 
$n\!\not\equiv\!0\!\pmod8$,
and the maps $W'_n(\C)\to WR_n{\,}'(\C)$ are isomorphisms for 
$n\not\equiv1,3\pmod8$.
\end{theorem}

\begin{proof}
Since $WR_n(\C)$ and $WR{\,}'_n(\C)$ are at most $\Z/2$, 
Theorems \ref{WR-KO.exponents} and \ref{W=WR.mod2tors} imply that 
$W_n(\C)$ and $W'_n(\C)$ are 2-primary groups of bounded exponent.
Consider the following diagram, where we have omitted 
the '$(\C)$' for legibility.
\[\begin{array}{ccccccccc}
 && W'_n &\map{}& {WR{\,}'}_n && && W'_{n-1} \\
 && \downarrow && \downarrow && &&\downarrow \\
GW_n^c &\map{}& GW_n &\map{}& KO_n &\map{}& GW_{n-1}^c &\map{}& GW_{n-1} \\
\downarrow && \downarrow && \downarrow && \downarrow && \downarrow  \\
K_n^c&\map{}& K_n &\map{}& KU_n &\map{}& K_{n-1}^c &\map{}& K_{n-1} \\
\downarrow && \downarrow && \downarrow && \downarrow && \downarrow  \\
GW_n^c &\map{}& GW_n &\map{}& KO_n &\map{}& GW_{n-1}^c &\map{}& GW_{n-1} \\
 && \downarrow && \downarrow && &&\downarrow \\
 && W_n &\map{}& {WR}_n && && W_{n-1} \\
\end{array}\]
By Theorem \ref{GW=GR} and Corollary \ref{Kc-u.d.}, the 
comparison groups 
$GW_n^c$ and $K_n^c$ are uniquely 2-divisible for $n\ge0$.

If $n\equiv0,4$ then $KO_n$ and $KU_n$ are $\Z$, and $KO_{n+1}$
is a torsion group. In addition, $KO_{n-1}=KU_{n-1}=0$,
and $K_n(\C)\to KU_n$ is zero. 
A diagram chase shows that $GW_n^c\cong GW_n$, and that
$KO_n\cong\Z$ injects into $GW^c_{n-1}$ with cokernel $GW_{n-1}$.
The first fact implies that
$W_n$ and $W'_n$ are uniquely 2-divisible, hence zero. The second
fact implies that $GW_{n-1}$ is the sum of a uniquely 2-divisible
group and the 2-primary divisible group $(\Q/\Z)_{(2)}$;
this implies that $W_{n-1}$ is divisible, hence zero.

If $n\equiv4$, the map from $\Z=KO_n$ to $KU_n=\Z$ is multiplication by~2;
if $n\equiv8$ the map $KO_n\to KU_n$ is an isomorphism.
Another diagram chase shows that $W'_{n-1}$ is $\Z/2$ if $n\equiv4$,
and~0 if $n\equiv8$.

If $n\equiv5,6$, so that $KO_{n+1}=KO_n=0$, we see that $GW^c_n\cong GW_n$.
As this is uniquely 2-divisible, this forces $W_n=W'_n=0$.

This shows that $W_n=WR_n=0$ for $n\not\equiv0,1,2$, 
$W_n\not\cong WR_n$ for $n\equiv3$, $W'_n=WR{\,}'_n=0$ for
$n\not\equiv1,2,3$, and $W'_n\not\equiv WR{\,}'_n$ for $n\equiv3$.

We are left to consider the cases $n\equiv1,2\pmod8$. 
In these cases, $KO_{n+1}\to GW^c_n$ and $KO_n\to GW^c_{n-1}$
must be zero, so the top row yields split exact sequences
\[
0\to  GW^c_n \to GW_n \to KO_n \to 0.
\]
That is, $GW_n$ is the sum of the uniquely 2-divisible
group $GW^c_n$ and $\Z/2$, as asserted.

If $n\equiv2$, $K_n(\C)$ is uniquely divisible; this implies that
$W_n\smap{\cong}WR_n$ and $W'_n\smap{\cong}WR{\,}'_n$ are $\Z/2$. 
If $n\equiv1$,  the bounded exponent
of $W_n(\C)$ and the divisibility of $K_n(\C)$ 
implies that $W_n(\C)\smap{\cong}WR_n\cong\Z/2$.  

Next, we show that $W'_n(\C)=0$ when $n\equiv1$.
The bounded exponent of $W'_n(\C)$ as a subgroup of $GW_n(\C)$ 
implies that $W'_n(\C)$ is a subgroup of the 2-primary torsion 
subgroup $\Z/2$ of $GW_n(\C)$, and that 
$W'_n(\C)\to WR{\,}'_n(\C)\cong\Z/2$ is an injection.
Therefore, it suffices to show that the map from 
$W'_n(\C)$ to $WR{\,}'_n(\C)\cong\Z/2$ is zero when $n\equiv1\pmod8$.

For this, recall from \eqref{seq:GW[i]} that $W'_n(\C)$ is the image 
of the map $GW^{[1]}_{n+1}\to GW_n(\C)$, and $WR{\,}'_n(\C)\cong\Z/2$ is 
the image of the map $GR^{[1]}_{n+1}(\C)\to GR_n(\C)$.
(The group $GW^{[1]}_{n+1}=GW^{[1]}_{n+1}(\C)$ is sometimes written 
as $V_n(\C)$.)
Consider the commutative diagram with exact rows and columns:
\[\begin{array}{ccccccccc}
 && K_{n+1}(\C) &\map{}& KU_{n+1} \\ && \downarrow && \downarrow \\
GW^{[1],c}_{n+1} &\map{}& GW^{[1]}_{n+1} & \map{} & GR^{[1]}_{n+1}\\ 
\downarrow && \downarrow && \downarrow \\
GW^c_n     &\map{}& GW_n(\C)  &\map{}& GR_n(\C). 
\end{array}
\]
In the right vertical sequence we have 
$GR^{[1]}_{n+1}\!\cong\!KO_{n-1}\!\cong\!\Z$;
see \cite[3.2]{Atiyah}. 

If the map $GW^{[1]}_{n+1}\to GR^{[1]}_{n+1}$ were nonzero, $GW^{[1]}_{n+1}$ 
would contain a summand isomorphic to $\Z$. Since $K_{n+1}(\C)$ is 
divisible, this summand would inject into the subgroup $W'_n$ of $GW_n(\C)$.
But this contradicts the fact that $W'_n(\C)$ is a subgroup of $\Z/2$.
Therefore the map $GW^{[1]}_{n+1}\to GR^{[1]}_{n+1}$ is zero; 
this implies that the map from the quotient $W'_n(\C)$ of $GW^{[1]}_{n+1}$
to the quotient $WR{\,}'_n(\C)$ of $GR^{[1]}_{n+1}$ is zero. 
Since we have seen that this
map is an injection, we conclude that $W'_n(\C)=0$.
\end{proof}

\newpage \appendix

\section{Errata and addenda}
\label{app:errata} 

The purpose of this short appendix is to correct some statements in \cite{KW}
about the comparison map between $KR$-theory and
equivariant $KO$-theory.


\begin{Void}\label{errata1} 
Proposition 1.8 in \cite{KW}, which stated that if $X$ has
no fixed points then $KR^{n}(X)$ is 4-periodic, is false. Indeed, the $(p-1)$%
-sphere with antipodal involution is a counterexample for all $p\geq4$ with $%
p\neq3\mod8$. (The case $p=4$ is in \cite[3.8]{Atiyah}.)

To see this, recall that the notation $B^{p,q}$ (resp., $S^{p,q})$ denotes
the ball (resp., the sphere) in $\R^{p}\times\R^{q}=\mathbb{R%
}^{p+q}$ provided by the involution $(x,y)\mapsto(-x,y).$ In particular, the
sphere $S^{p,q}$ has dimension $p+q-1.$ The $KR$-theory exact sequence
associated to the pair $(B,S)=(B^{p,0},S^{p,0})$ is: 
\begin{equation*}
KR^{n}(B,S)\to KR^{n}(B)\to KR^{n}(S)\to
KR^{n+1}(B,S)\to KR^{n+1}(B).
\end{equation*}
By the $KR$ analog of Bott periodicity \cite{MKThesis}, we have $%
KR^{n}(B^{p,0},S^{p,0})\cong KR^{n+p}(\ast)\simeq KO^{n+p}(\ast),$ where `$%
\ast$' is a point. Moreover, if $p\geq3,$ it is proved by Atiyah
\cite[3.8]{Atiyah} (and \cite{MKThesis}) 
that this exact sequence reduces to the split short
exact sequence:%
\begin{equation*}
0\to KO^{n}(\ast)\to KR^{n}(S^{p,0})\to
KO^{n+p+1}(\ast)\to0.
\end{equation*}
Hence $KR^{n}(S^{p,0})\cong KO^{n}(\ast)\oplus KO^{n+p+1}(\ast).$ This group
is indeed periodic of period 4 if $p=3$ mod 8. However, it is not periodic
of period 4 otherwise.

The mistake in the proof of the proposition lies in the claim that the $%
KR^{\ast }(X)$-module map $KR^{n}(X)\to KR^{n+4}(X)$ is the square
of Bott periodicity in the case $X=Y\times C_{2}$ with the obvious $C_{2}$
action. This cannot be true since the map $\Z\cong KR^{4}(\ast
)\to KR^{4}(C_{2})\cong\Z$ is multiplication by 2.
\end{Void}


\begin{Void}
The mistake in Proposition 1.8 propagates to Theorem 4.7 of \cite{KW}, where
the periodicity statement has to be amended: the $KR$ groups are periodic of
period 8 and not 4.
\end{Void}

\begin{Void}
Example A.3 in \cite{KW} is complete, because the calculation shows that the
groups $KR^{n}(X)$ are 4-periodic when $V$ is a smooth curve over $\R
$ with no real points; the invocation of 1.8 is unnecessary.
\end{Void}

\begin{Void}
All other results in \cite{KW}, including the main results, are unaffected
by this error.
\end{Void}

\bigskip\goodbreak

\section{Higher Witt groups of $\R$}
\label{app:Bott} 

In this appendix, we recall some basic facts about (positive and negative)
Bott elements in higher Witt groups. We will write $W_{*}^{\topl}(\R)$ 
for the ring $WR_{*}(\Spec\R)$; this is the Hermitian theory based
upon the Banach algebra $\R$ and is isomorphic to the ring $KO_*(\pt)$
by Example \ref{WR(X^G)}.  Although $W_0(\R)\cong W^\topl_0(\R)$,
the groups $W_n(\R)$ and $W^\topl_n(\R)$ differ for general $n$
(see Theorem \ref{Wn(R)}).

The following square was shown to be homotopy cartesian 
on connective covers in \cite[Thm.\,A]{BK}
(see also \cite{BKOS} for a more conceptual proof), where the 
notation $(\;)_2\!\hat{}$~  
indicates 2-adic completion. 
\begin{equation}\label{square:BK}
\begin{array}{ccc}
\GW(\Z[\frac12])_2\!\widehat{} & \to & \GW(\R^\topl)_2\!\widehat{}~ \\ 
\downarrow &  & \downarrow \\ 
\GW(\mathbb{F}_{3})_2\!\widehat{} & \to & \GW(\C^\topl)_2\!\widehat{}~.%
\end{array}%
\end{equation}%
By \cite[3.6]{BK}, the homotopy groups of $\GW(\Z[\frac12])$ and
$\GW(\mathbb{F}_{3})$ are finitely generated, so
$\pi_n\GW(\Z[\frac12])\otimes\Z_2 \cong \pi_n(\GW(\Z[\frac12])_2\!\widehat{}$~.
From this square, we deduce the existence, for all $k>0$, of elements $%
u_{4k} $ in $GW_{4k}(\Z[\frac12])$ whose image in 
$W^\topl_{4k}(\R)\cong\Z$ is 2 times a generator; 
see \cite[1.2]{MKfiltration}.

The following lemma corrects a small mistake in \cite[p.\,200]{MKfiltration}.

\begin{lemma}\label{Bott.correct}
There is an element $u_{-8}$ in $W_{-8}(\Z[\frac12])$ whose image in 
$W_{-8}^{{\topl}}(\R)$ is 16 times a generator.

There is an element $u_{-4}$ in $W_{-4}(\Z[\frac12])$ whose image in 
$W_{-4}^{{\topl}}(\R)$ is 2 times a generator.
\end{lemma}

\begin{proof}
Set $\Z'=\Z[\frac12]$ and recall that the map from
$W_0(\Z')\cong\Z\oplus\Z/2$ to $W^\topl_0(\R)\cong\Z$ is onto. 
We have strict inclusions
or equalities coming from the 12-term exact sequence for $\Z'$ 
and $\R$ (with its usual topology) \cite{MKAnnalsH}: 
\begin{multline*}
{}_{1}W_{0}(\Z')={}_{-1}W_{-2}(\Z')={}_{1}W_{-4}'(\Z')={}_{1}W_{-4}(\Z') 
\\
={}_{-1}W_{-6}'(\Z')={}_{-1}W_{-6}(\Z')={}_{1}W_{-8}'(\Z')={}_{1}W_{-8}(\Z') 
\\
\shoveleft{
_{1}W^\topl_{0}(\R)={}_{-1}W^\topl_{-2}(\R)={}_{1}{W'}^\topl_{-4}(\R)
          \subset{}_{1}W^\topl_{-4}(\R) } \\
={}_{-1}{W'}^\topl_{-6}(\R)\subset {}_{-1}W^\topl_{-6}(\R)\subset
          {}_{1}{W'}^\topl_{-8}(\R)\subset {}_{1}W^\topl_{-8}(\R).
\end{multline*}%
The comparison between these two sequences of groups \cite{MKfiltration}
shows the existence of an element $u_{-8}$ in ${}_{1}W_{-8}(\Z')$ 
whose image in $_{1}W^\topl_{-8}(\R)\cong\Z$ is 
16 times a generator (not 8 times as claimed by mistake in 
\cite[p.\,200]{MKfiltration}).
In the same way, there is an element $u_{-4}$ in $_{1}W_{-4}(\Z')$ 
whose image in $_{1}W_{-4}^{{\topl}}(\R)\cong\Z$ is 2 times a generator.
\end{proof}

We can now correct the factor of 2 in \cite[Theorem 1.3]{MKfiltration}.

\begin{corollary}\label{period.correct} 
For any ring $A$ (or any scheme $V$) over $\R$, 
there exist ``Bott maps''
\[
{}_{\varepsilon}W_{n}(V)\to{}_{\varepsilon}W_{n+8m}(V) \text{ and } 
{}_{\varepsilon}W_{n+8m}(V)\to{}_{\varepsilon}W_{n}(V),
\]
such that both compositions are multiplication by $2\cdot16^{m}$.
\end{corollary}

\begin{proof}
If $m>0$, the first map is the cup product with $u_{8m}$ and the second map
is the cup product with $(u_{-8})^m$. Lemma \ref{Bott.correct} shows that
the compositions are multiplication by $2\cdot16^m$
in $W_0(\R)\cong\Z$.
\end{proof}

The cup-product with $u_{4}$ and $u_{-4}$
give different Bott-style maps, from
$_{\varepsilon }W_{n}(V)$ to ${}_{\varepsilon }W_{n+4}(V)$ 
and vice versa.
The following calculation is implicit in \cite[1.3]{MKfiltration}.

\begin{corollary}\label{4period} 
The image of the product $u_4 u_{-4}$ is 16 in $W_{0}(\R)\cong\Z$. 
Hence both compositions of the two maps
\[
_{\varepsilon }W_{n}(V) \map{u_4} _{\varepsilon }W_{n+4}(V)
\textrm{ and } 
_{\varepsilon }W_{n+4}(V)\map{u_{-4}} _{\varepsilon }W_{n}(V)
\]
are multiplication by $16$.
\end{corollary}

\begin{proof}
Let $x_{4k}$ denote the standard generator of $W_{4k}^{\topl}(\R)$.
It is well known that $x_8x_{-8}=1$, $x_{-4}^2=4x_{-8}$ and $x_4^2=4x_8$;
see \cite[III.5.19]{MKbook}. By Lemma \ref{Bott.correct},
the images of $u_{4}^2$ and $u_{-4}^2$ under the ring homomorphism 
$h:W_*(\Z[1/2])\to W_*^{\topl}(\R)$ are
$h(u_4^2)=(2 x_4)^2 = 16 x_8=8 h(u_8)$, and
%
%
$h(u_{-4}^2)=(2 x_{-4})^2 = 16 x_{-8} = h(u_{-8})$.
By Corollary \ref{period.correct}, $u_8u_{-8}=32$ in $W_0(\R)$. 
Hence $(u_4u_{-4})^2= 8 u_8u_{-8}=8\cdot32$ in 
$W_0(\R)\cong W_0^\topl(\R)\cong\Z$ 
and hence $u_4u_{-4} = 16$.
\end{proof}

We conclude this appendix by citing a symmetry-changing
periodicity result, which we use in Theorem \ref{W(11+)}.
It is a restatement of Theorem 3.7 in \cite{MKAnnalsO}.
Note that it provides another proof of of Lemma \ref{4period}
by iterating the periodicity map.

\begin{lemma}\label{2period}
There are elements $u_2$ in $W_2(\Z[\frac12])$, $u_{-2}$ in
$W_{-2}(\Z[\frac12])$ whose product in $W_0(\Z[\frac12])$ is 4.

If $A$ is a ring containing $1/2$, the composition of the homomorphisms
\[
{}_{\eps}W_n(A) \map{u_2}{}_{-\eps}W_{n+2}(A), \qquad
{}_{-\eps}W_{n+2}(A)\map{u_{-2}}{}_{\eps}W_n(A)
\]
(in either direction) is multiplication by $4$.
\end{lemma}

\newpage
\section{Topological Hermitian $K$-theory of involutive Banach algebras}
\label{app:Banach} 


The content of this section is essentially included in 
\cite{MKSLN343}, Section III, except at the end. 
The only originality is a new presentation of results
which were written down over forty years ago.

It is convenient to introduce the language of Clifford modules into
topological $K$-theory. (This was first pointed out by Atiyah, Bott and
Shapiro \cite{ABS} and developed in \cite{MKThesis}.) By definition, the
Clifford algebra $C^{p,q}$ is the $\R$-algebra generated by elements 
$e_{i}$ and $\varepsilon_{j}$ with $1\leq i\leq p$ and $1\leq j\leq q,$ with
the following relations:

$(e_{i})^{2}=-(\varepsilon_{j})^{2}=-1$,

$e_{i}\varepsilon_{j}+\varepsilon_{j}e_{i}=0$,

$e_{\alpha}e_{\beta}+e_{\beta}e_{\alpha}=0$ for $\alpha\neq\beta$,

$\eps_{\gamma }\eps_{\delta }+\eps_{\delta}\eps_{\gamma }=0$ 
for $\gamma\ne\delta.$

\medskip

If $A$ is a Banach algebra over $\R$, a $C^{p,q}\!$--$A$-module is a
finitely generated projective module over the tensor algebra $C^{p,q}\otimes
_{\R}A$; we write $\cE^{p,q}(A)$ for the category of $%
C^{p,q} $- $A$-modules. Using an averaging method over the $Pin$ group (see
e.g.\ \cite[p.\,185]{MKThesis}), one sees that this category is equivalent
to the category of finitely generated projective $A$-modules which are
provided with a $C^{p,q}$-module structure.


The group $K^{p,q}(A)$ is defined as the Grothendieck group of the
\textquotedblleft restriction of scalars\textquotedblright\ functor%
\begin{equation*}
\cE^{p,q+1}(A)\to\cE^{p,q}(A)
\end{equation*}
arising from $C^{p,q}\to C^{p,q+1}$. It is shown in \cite{MKThesis}
that this group is isomorphic to the classical topological $K$-group $%
K^{p-q}(A)=K_{q-p}^{{\topl}}(A)$ of the Banach algebra $A.$

It is convenient to describe the group $K^{p,q}(A)$ in terms of $\Z%
/2 $-gradings. A $\Z/2$-grading on a $C^{p,q}$ module $E$ is an
involution which provides $E$ with a $C^{p,q+1}$ module structure compatible
with the given $C^{p,q}$ module structure.

One considers triples $(E,\varepsilon,\eta),$ where $E$ is a $C^{p,q}$ $A$%
-module and $\varepsilon,\eta$ are two independent $\Z/2$-gradings. 
The group $K^{p,q}(A)$ is then generated by isomorphism classes of such
triples with the following relations:

$(E,\varepsilon,\eta)+(E',\varepsilon',\eta')=
(E\oplus E',\varepsilon\oplus\varepsilon',\eta\oplus
\eta')$

$(E,\varepsilon ,\eta )=0$ if $\varepsilon $ is homotopic to $\eta $ among
the $\Z/2$-gradings.

\medskip
\noindent
Let us assume now that the Banach algebra $A$ is provided with a continuous
automorphism $x\mapsto \overline{x}$ of order~2. 
We emphasize that $A$ is not necessarily a $C$* algebra, one reason being
that the element $1+x\overline{x}$ might not be invertible.

For any pair of integers $(p,q),$ we associate two kinds of topological
Hermitian $K$-groups to $A$, which we call $GW^{p,q}(A)$ and $U^{p,q}(A)$,
respectively. (The first group was originally called $L^{p,q}(A)$ in \cite%
{MKSLN343}.) 
These two definitions are in the same spirit as the previous definition of
the group $K^{p,q}(A).$ However, we should be careful about our definition
of $C^{p,q}$ module in the Hermitian framework. We distinguish two cases:
\goodbreak

1st case. The generators $e_{i}$ and $\varepsilon_{j}$ act by unitary
operators, i.e., $u.u^{\ast}=1$ with $u=e_{i}$ or $\varepsilon_{j}$. Such
modules are called \emph{Hermitian}; $\mathcal{Q}^{p,q}(A)$ denotes the
category of Hermitian $C^{p,q}$ modules.

2nd case. The generators $e_{i}$ and $\varepsilon_{j}$ act by antiunitary
operators, i.e., $u.u^{\ast}=-1$ with $u=e_{i}$ or $\varepsilon_{j}$ Such
modules are called \emph{skew-Hermitian}; $\mathcal{U}^{p,q}(A)$ denotes the
category of skew Hermitian $C^{p,q}$ modules.


Equivalently, the Clifford algebra $C^{p,q}$ can be provided by the
anti-involution defined by either 
$\overline{e_{i}}=-e_{i}, \overline{\varepsilon_{j}}=\varepsilon_{j}$ or 
$\overline{e_{i}}=e_{i},\overline {\varepsilon_{j}}=-\varepsilon_{j}.$ 
The categories $\mathcal{Q}^{p,q}(A)$
and $\mathcal{U}^{p,q}(A)$ are then defined more systematically as the
categories of Hermitian modules over the ring $C^{p,q}\otimes_{\R}A,$
the two cases relying on the two possible anti-involutions on the Clifford
algebra $C^{p,q}.$

Following the previous scheme for the groups $K^{p,q},$ we define the groups 
$GW^{p,q}(A)$ and $U^{p,q}(A)$ as the Grothendieck group of the restriction
of scalars functors%
\[
\mathcal{Q}^{p,q+1}(A) \to\mathcal{Q}^{p,q}(A)
\qquad\textrm{and}\qquad
\mathcal{U}^{p,q+1}(A) \to\mathcal{U}^{p,q}(A)
\]%
%
respectively. These groups can also be defined as homotopy groups of
suitable homotopy fibers, denoted respectively by $\GW^{p,q}(A)$ and 
$\mathbb{U}^{p,q}(A)$. We now state the ``fundamental theorem'' of
topological Hermitian $K$-theory \cite{MKSLN343}:

\begin{theorem}
We have natural homotopy equivalences%
\[
\GW^{p,q+1}(A)        \sim\Omega(\GW^{p,q}(A))
\qquad\textrm{and}\qquad
\mathbb{U}^{p,q+1}(A) \sim\Omega(\mathbb{U}^{p,q}(A)).
\]
%
\end{theorem}

Because of the periodicity of Clifford algebras \cite{ABS} \cite{MKThesis},
the theorem above leads to 16 homotopy equivalences. It implies that the
group $GW^{p,q}(A)$ is isomorphic to $GW_{n}^{{\topl}}(A)$ for $n=q-p$
mod 8. In the same way, the group $U^{p,q}(A)$ is isomorphic to 
$U_{n}^{{\topl}}(A)$.

One of the most remarkable cases is $p=0,q=1.$ The analysis of this case
made in \cite[p.\,338]{MKSLN343} shows that the group 
$_{\varepsilon}U^{0,1}(A)$ may be identified with the so-called $V$-group 
$_{-\varepsilon}V^{0,0}(A)$ which is the Grothendieck group
of the forgetful functor 
\vspace{6pt}
\[
{}_{-\varepsilon}\mathcal{Q}^{0,0}(A)\to\cE^{0,0}(A).
\]
We note here the change of symmetry between the $U$ and $V$ theories. This
particular result has been greatly generalized for discrete rings 
\cite{MKAnnalsH} and more generally categories with duality 
\cite{Schlichting.Fund}.

Another case of interest is $p=1,q=0.$ The following Lemma was stated
without proof in \cite{MKSLN343}.

\begin{lemma}\label{U10(A)} 
The group $U^{1,0}(A)=U_{-1}^{{\topl}}(A)$ is isomorphic to the 
Grothendieck group of the extension of scalars functor%
\[
\mathcal{Q}(A)\to\mathcal{Q}(A\otimes_{\R}\C).
\]
\end{lemma}
\goodbreak\noindent
Here $\C$ is the field of complex numbers with the trivial involution.

\begin{proof}
An element of $U^{1,0}(A)$ is given by a quadruple $(E,J,\varepsilon_{1},%
\varepsilon_{2})$ where $E$ is an Hermitian module, and $J,\varepsilon_{1},%
\varepsilon_{2}$ automorphisms such that $J^{2}=-1,J^{\ast}=J,(%
\varepsilon_{i})^{2}=1,J\varepsilon_{i}=-\varepsilon_{i}J$ and finally $%
(\varepsilon_{i})^{\ast}=-\varepsilon_{i}.$ To this quadruple one can
associate bijectively another one $(E,J,\eta _{1},\eta_{2})$ with $%
\eta_{i}=J\varepsilon_{i}.$ One still has $(\eta _{i})^{2}=1$ and $%
J\eta_{i}=-\eta_{i}J,$ but now $(\eta_{i})^{\ast}=\eta_{i}. $ Therefore, the
group $U^{1,0}(A)$ may be identified with the Grothendieck group of the
functor%
\[
\mathcal{Q}(A\otimes_{\R}C^{1,1})\to\mathcal{Q}(A\otimes_{\R}C^{1,0})
\]
where $C^{1,0}=\C$ has the trivial involution and $C^{1,1}=M_2(\R)$ has the
involution defined by $J^{\ast}=J$ and $\eta^{\ast}=\eta$ on the generators
(with $J^{2}=-1$ and $\eta^{2}=1)$. By Morita equivalence, this functor
coincides with the extension of the scalars functor above.
\end{proof}

Another topological interpretion of this result is to consider the 
algebra $B$ of continuous functions $f:S^{1}\to A\otimes_{\R}\C$ 
such that $f(\overline{z})=\overline{f(z)}$ and $f(1)=0.$ 
If we put $D^{1}=\left[ -1,1\right]$ and 
$S^{0}=\left\{ -1,1\right\}$, with the involution $x\mapsto -x$ as the
analog of complex conjugation, we see by a topological deformation of 
$\left[-1,1\right]$ into $\left\{0\right\}$ that the Hermitian 
$K$-theory of $B$ is the same as the $K$-theory of the functor
\[
\mathcal{Q}(A)\to \mathcal{Q}(A\otimes _{\R}C^{1,0})
\]
which is therefore $U_{-1}^\topl(A),$ according to Lemma \ref{U10(A)}. 

\begin{defn}
Let $A\langle z,z^{-1}\rangle$ the the algebra of ``Laurent'' series $\sum
a_{n}z^{n}$ with $a_{n}\in A$ and $\sum\left\Vert a_{n}\right\Vert <+\infty.$

Let $A_z$ denote the subalgebra of $A\langle z,z^{-1}\rangle$ 
consisting of all series $\sum a_{n}z^{n}$ with $\sum a_{n}=0.$
\end{defn}

By the usual density theorem in $K$-theory \cite[p.\,109]{MKbook} and the
theory of Fourier series, the group $GW_0(B)$ may be identified with 
$GW_0(A_{z}) $. Therefore, we get the following theorem.

\begin{theorem}
\label{Laurent} Let $A$ be an involutive Banach algebra. Then we have a long
exact sequence%
\[
\to U_{i}^{{\topl}}(A)\to GW_{i}^{{\topl}%
}(A)\to GW_{i}(A\langle z,z^{-1}\rangle )\to U_{i-1}^{\topl}(A)\to
\]
\end{theorem}

\begin{remark}\label{rem:Laurent}
The identification of $U_{-1}^\topl(A)$ with $GW_{0}(A_{z})$ in the previous
theorem is the topological analog of a theorem in (algebraic) Hermitian $K$%
-theory for any discrete ring $A$ : the group $U_{-1}(A)$ is inserted in an
exact sequence, analogous to the sequence of 
the fundamental theorem in algebraic $K$-theory, 
\begin{equation*}
0\to GW_0(A)\to {\atop{GW_0(A[z])\;\oplus}{GW_0(A[z^{-1}])}} \to GW_0(A[z,z^{-1}]) 
\to U_{-1}(A)\to0.
\end{equation*}%
See \cite[p.\,390]{MKlocalisation} and also \cite{Hornbostel-Schlichting}
for a more recent approach.
\end{remark}


We apply these general considerations to the following situation. We start
with a compact space $X$ with an involution, i.e., an action of the cyclic
group $G=C_{2}.$ Following again Atiyah's terminology, we call $X$ a Real
space and we let $x\mapsto \overline{x}$ denote the involution. In this
situation, our basic algebra $A$ will be the algebra of complex continuous
functions $f:X\to \C$ such that $f(\overline{x})=\overline{f(x)}.$ 
It is well known that $K_0(A)$ is Atiyah's Real $K$-theory, $KR(X)$
(see \cite[Ex.\,III.7.16d]{MKbook} and \cite{KW:KRA}). On the other hand, 
$GW_0(A)$ is $GR(X)$, which is isomorphic to the usual equivariant 
$KO$-theory, $KO_{G}(X)$, by Theorem \ref{GR=KOG}.

Finally, the group $GW_0(A_{z})$ 
may be interpreted as the relative group 
$KO_{G}(X\times D^1,X\times S^0)=KO_G(X\times S^1,X),$
where $D^1=[-1,1]$, $S^0=\{-1,1\}$ and $S^1$
have the involution $t\mapsto-t$.
A more synthetic version of this last definition is simply 
$KO_{G}(X\times \R)$ which is $KO_{G}$-theory with compact supports, 
$\R$ being also provided with the involution $t\mapsto -t.$ 
A corollary of this discussion is the following seemingly unknown link 
between $KR$-theory and $KO_{G}$-theory, which we shall use in
Theorem \ref{WR=KOG} of the text.

\begin{theorem}\label{cup-product} 
Let $X$ be a compact space with involution. Then we have
the following exact sequence 
\[
\to KR(X)\to KO_{G}(X)\map{\gamma}
KO_{G}(X\!\times\!\R)\to KR^{1}(X) \to KO^1_{G}(X)\map{\gamma},
\]
where the map $\gamma $ is induced by the cup-product with a generator of $%
KO_{G}(\R)\cong \Z$, where the involution acts as $-1$ on $\R$.
\end{theorem}

\begin{proof}
This sequence is the exact sequence of $U$-theory, with the substitutions 
$KO_{G}(X)$ for $GR(X)$ and $KO_{G}(X\times \R)$ for $U^{1}(X)$.
Since $\gamma $ is a map of $KO_{G}(X)$-modules, it is the (external)
cup product with an element of $KO_{G}(X\times \R)$. 
By naturality in $X$, it is the cup product with an element of 
$KO_{G}(\R)\cong \Z$. To see that this element is a generator, 
we may choose $X$ to be a point. Since $KR^{1}(X)=KO^{1}(X)=0,$ 
$\gamma$ is surjective in this case, as required.
\end{proof}

\begin{remark}
If $X$ is a ``nice'' $G$-space (i.e. with orbits having equivariant tubular
neighborhoods), it is possible to give an elementary proof of the previous
theorem by reducing it to the two extreme cases of a free $G$ action and of
a trivial $G$ action.
We leave this as an exercise for the reader.
\end{remark}

As seen in Theorem \ref{GR=KOG}, the cokernel of the map $KR(X)\to
KO_{G}(X)$ is the Real Witt group $WR(X).$ Therefore, Theorem \ref%
{cup-product} implies the following corollary which we shall use in the
computations of the Witt group of real smooth projective curves 
(Section \ref{sec:curves}).

\begin{corollary}
As a $KO_{G}(X)$ module, the Real Witt group $WR(X)$ is a submodule of $%
KO_{G}(X\times \R).$ Moreover, if the map $KR^{1}(X)\to
KO_{G}^{1}(X)$ is injective, we have $WR(X)\cong KO_{G}(X\times \R).$
\end{corollary}

\section{Williams' conjecture for Banach algebras}
\label{sec:Marco}  

Finally, although we don't need it for our applications, the following
theorem is worth mentioning; it is related to the considerations in
our Section \ref{sec:williams} and 
provides a more conceptual proof of Theorem \ref{Williams-GR}.

For a (real) Banach algebra with involution, denote by $_{\eps}\GW_\topl(A)$ its Bott periodic topological hermitian K-theory spectrum \cite{MKSLN343} with underlying $\Omega^{\infty}$-loop space homotopy equivalent to $_{\eps}GW_0(A)\times B{_{\eps}O}_\topl(A)$ and deloopings obtained using topological suspensions.
Here ${_{\eps}O}_\topl(A)$ is the $\eps$-orthogonal group of $A$ equipped with the topology induced from $A$.
Similarly, denote by $\bbK_\topl(A)$ its Bott periodic topological $K$-theory.

\begin{theorem}\label{Williams-Banach} 
Let $A$ be an arbitrary (real) Banach algebra with
involution (not necessarily $C^{\ast })$. We then have a $2$-adic homotopy
equivalence%
\begin{equation*}
_{\varepsilon}\GW_\topl(A)\simeq \bbK_\topl(A)^{h_{\varepsilon}G}
\end{equation*}%
where $h_{\varepsilon }G$ denotes the space of homotopy fixed points for the 
$\varepsilon $ action of $G$ which is detailed in \cite[p.\,808]{BK}.
\end{theorem}

\begin{proof}
Let $A$ be a (real) Banach algebra with involution. 
Recall from \cite[\S 10]{Schlichting.Fund} the spectrum
\[
GW_\topl^{[n]}(A) = |\uGW^{[n]}(A\Delta^*_\topl)|
\]
which is the realization of the simplicial spectrum 
$q\mapsto \uGW^{[n]}(A\Delta^q_\topl)$ of the $n$-th shifted Grothendieck-Witt spectrum of the discrete ring $A\Delta^q_\topl$ of continuous functions $\Delta^q_\topl \to A$ from the standard topological $q$-simplex $\Delta^q_\topl$ to $A$.
It's connective cover is equivalent to the connective cover of $_{\eps}\GW_\topl(A)$
(when $n=0\mod 4$ and $\eps =1$, or $n=2\mod 4$ and $\eps =-1$); see \cite[Prop.\,10.2]{Schlichting.Fund}.
 Its negative homotopy groups are the
Balmer Witt groups of $A$ (considered as a discrete ring); 
see \cite[Remark 10.4]{Schlichting.Fund}.
The spectrum $GW_\topl^{[n]}(A)$ is a module spectrum over the 
ring spectrum $GW_\topl(\R)$.
Recall that $\GW_\topl^{[n]}(A)$ denotes 
the module spectrum which has the same connective cover as $GW_\topl^{[n]}(A)$ 
but with negative homotopy groups obtained through deloopings using 
topological suspensions.
There is a map of module spectra $GW_\topl^{[n]}(A)\to \GW_\topl^{[n]}(A)$ 
since the source naturally maps to the version with deloopings constructed via algebraic
suspension \cite[\S 8]{Schlichting.Fund} which naturally maps to 
$\GW_\topl^{[n]}(A)$. 
Let $K_\topl(A)$ and $\bbK_\topl(A)$ denote the corresponding
topological $K$-theory versions. For instance $K_\topl(\R)$ is connective
topological real vector bundle $K$-theory and $\bbK_\topl(R)$ is Bott
periodic topological real vector bundle $K$-theory. 
We have a map of ring spectra $\bbK_\topl(\R)\to\GW_\topl(\R)$ 
since the hermitian $K$-theory of positive definite forms is equivalent 
to $\bbK_\topl(\R)$ and canonically maps to $\GW_\topl(\R)$.
In particular, $\GW_\topl(\R)$ and hence $\GW_\topl(A)$ are
Bott-periodic with period $8$, and $\GW_\topl(A)$ is obtained from $GW_\topl(A)$
by inverting the Bott element $\beta \in \pi_8GW_\topl(\R)$.

Let $\eta \in GW_{-1}^{[-1]}(\R)=W(\R)=\Z$ be a
generator. In \cite[Thm.\,7.6]{Schlichting.Fund} it is proved that the
following square of spectra is homotopy cartesian 
\begin{equation*}
\xymatrix{ \GW^{[n]}_\topl(A) \ar[r] \ar[d] & \GW^{[n]}_\topl(A)[\eta^{-1}] \ar[d]\\
(\bbK^{[n]}_\topl(A))^{hC_2} \ar[r] & (\bbK^{[n]}_\topl(A))^{hC_2}[\eta^{-1}].}
\end{equation*}%
Strictly speaking an algebraic version of this statement was proved.
However, \cite[Thm.\,7.6]{Schlichting.Fund} is a formal consequence of
\cite[Thm.\,6.1]{Schlichting.Fund} which holds for the topological versions
considered here since simplicial realization and 
inverting the Bott element preserves fibre sequences of spectra.

Let $\nu $ be an integer $\nu\ge1$. From the homotopy cartesian square
above, we obtain the homotopy cartesian square 
\[
\xymatrix{\GW^{[n]}_\topl/2^{\nu}(A)\ar[r]\ar[d]& 
          \GW^{[n]}_\topl/2^{\nu}(A)[\eta^{-1}] \ar[d]\\
(\bbK^{[n]}_\topl/2^{\nu}(A))^{hC_2} \ar[r] & 
(\bbK^{[n]}_\topl/2^{\nu}(A))^{hC_2}[\eta^{-1}].}
\]
We will show that the right vertical map is an equivalence simply by showing
that the two right hand terms are zero. This clearly implies that the left
vertical map is a weak equivalence. 
Since the diagram is a diagram of $\GW_\topl(\R)$-module spectra, we are now done by the following lemma.
\end{proof}

\begin{lemma}
\label{lem:etaNilpotent}
For each $\nu>0$, $\eta$ is nilpotent in $\GW_\topl/2^{\nu}(\R)$.
\end{lemma}


\begin{proof}
Consider the element $u=\beta \eta^8 \in GW_0(\R)= \pi_0\GW_\topl(\R)$. 
Under the map of
ring spectra $\GW_\topl(\R) \to \GW_\topl(\C)$, the element $u$ 
goes to zero simply because $\eta^8 \in \pi_{-8}\GW_\topl(\C)=\Z$ 
is in the image of the zero map
$$\pi_{-8}\uGW(\R) \to \pi_{-8}\uGW(\C) =  W(\C) = \Z/2 \to \pi_{-8}\GW_\topl(\C)=\Z.$$ 
As rings, the map $GW_0(\R) \to GW_0(\C)$ is 
$\Z[\varepsilon]/(\varepsilon^2=1) \to \Z: \varepsilon \mapsto 1$ where 
$\varepsilon$ corresponds to the form $\R \times \R \to 
\R:(x,y)\mapsto -xy$. So $u=a(1-\varepsilon)$ for some $a\in\Z$. 
We have $u^2 = a^2(1-\varepsilon)^2 = 2a^2(1-\varepsilon)$ since $%
\varepsilon^2=1$. Hence for $m=2\nu$, we have $\beta^{m}\eta^{8\cdot(m)} =
u^{(m)}= (u^2)^{\nu}=2^{\nu} a^{(m)}(1-\varepsilon)^{\nu}$ in $GW_0(\R)$. Since $\beta$ is a
unit, we are done.
\end{proof}

Note that the same argument for complex Banach algebras $A$ with involution
gives an integral equivalence 
\begin{equation*}
\GW_\topl(A)\overset{\simeq }{\longrightarrow }(\bbK_\topl(A))^{hC_{2}}
\end{equation*}%
simply because a power of $\eta $ is zero in $\GW_\topl(\C)$; see proof of Lemma \ref{lem:etaNilpotent}.


\end{document}